\RequirePackage{fix-cm}
\documentclass[twocolumn]{svjour3}          
\smartqed  

\usepackage{amsmath}
\usepackage{amssymb}
\usepackage{amsthm}
\usepackage{graphicx}
\usepackage{mathrsfs}
\usepackage{hyperref}
\usepackage{color}
\usepackage{multirow}
\usepackage{array}
\usepackage{booktabs}
\usepackage{enumitem}
\usepackage[misc,geometry]{ifsym} 

\DeclareMathOperator\diag{diag}

\DeclareMathOperator\Span{span}
\DeclareMathOperator\spect{Spect}

\theoremstyle{definition}

\begin{document}

\title{Nonlinear analysis of forced mechanical systems with internal resonance using spectral submanifolds, Part I: Periodic response and forced response curve}

\titlerunning{Nonlinear analysis of systems with internal resonance using spectral submanifolds, Part I}

\author{Mingwu Li         \and
        Shobhit Jain \and 
        George Haller
}


\institute{M. Li (\Letter)\and S. Jain\and G. Haller \at
              Institute for Mechanical Systems, ETH Z\"{u}rich \\
              Leonhardstrasse 21, 8092 Z\"{u}rich, Switzerland\\
              \email{mingwli@ethz.ch}           
}

\date{Received: date / Accepted: date}

\maketitle

\begin{abstract}
We show how spectral submanifold theory can be used to construct reduced-order models for harmonically excited mechanical systems with internal resonances. Efficient calculations of periodic and quasi-periodic responses with the reduced-order models are discussed in this paper and its companion, Part II, respectively. The dimension of a reduced-order model is determined by the number of modes involved in the internal resonance, independently of the dimension of the full system. The periodic responses of the full system are obtained as equilibria of the reduced-order model on spectral submanifolds. The forced response curve of periodic orbits then becomes a manifold of equilibria, which can be easily extracted using parameter continuation. To demonstrate the effectiveness and efficiency of the reduction, we compute the forced response curves of several high-dimensional nonlinear mechanical systems, including the finite-element models of a von K\'arm\'an beam and a plate.
\end{abstract}


\section{Introduction}
The forced response curve (FRC) of a mechanical system under harmonic excitation gives the amplitude of the periodic response of the system as a function of the excitation frequency. The FRC of a nonlinear system is significantly different from that of the linear part of the system, providing key insights into the nature of nonlinearities of the system. In particular, when a mechanical system has an internal resonance, the nonlinear behavior is often intriguingly complex~\cite{nayfeh1989modal}. Specifically, internal resonances tend to lead to energy transfer between modes~\cite{nayfeh1988undesirable,vakakis2008nonlinear,kurt2014effect,chen2017direct}, saturation~\cite{nayfeh1974nonlinear,nayfeh1988undesirable,balachandran1991observations,wood2018saturation}, localization~\cite{vakakis2001normal,kurt2014effect} and frequency stabilization~\cite{antonio2012frequency}.

The periodic orbit of a nonlinear mechanical system can be computed with various numerical methods. As the simplest method, direct numerical integration can be performed to find an asymptotically \emph{stable} periodic orbit in the steady state response if the initial condition of the forward simulation is in the basin of attraction of such a periodic orbit. Unstable periodic orbits arising in mechanics problems of practical relevance are of saddle types, and hence cannot be found in either forward or backward direct numerical simulations. In the shooting method~\cite{peeters2009nonlinearII,keller2018numerical}, the initial state is updated iteratively such that periodicity condition is satisfied. Therefore, the shooting method can locate unstable periodic orbits as well. 

To avoid numerical integration of the full system, the periodic orbit can be found with the collocation method~\cite{ascher1995numerical,dankowicz2013recipes} and the harmonic balance method~\cite{von2001harmonic,detroux2015harmonic,krack2019harmonic}. In the collocation method, the periodic orbit is approximated as a piecewise smooth function of time, expressed on each subinterval as a Lagrange polynomial, parametrized by the unknowns at the base points. The equation of motion is satisfied at a set of collocation nodes. In the harmonic balance method, the periodic orbit is approximated by a truncated Fourier series with unknown coefficients. These coefficients are solved from a set of nonlinear algebraic equations obtained by balancing the harmonics in the equation of motion.

The FRCs of low-dimensional mechanical systems can be effectively obtained from the above methods. However, mechanical systems generated from finite elements (FE) models generally contain thousands of degrees of freedom. Indeed, internal resonances have been observed in structural elements such as beams~\cite{nayfeh1974nonlinear,shaw2016periodic}, cables~\cite{kang2017dynamic}, plates~\cite{chang1993non,bilal2020experiments} and shells~\cite{thomas2005non,thomas2007non}. For such high-dimensional systems, the computational costs of the numerical methods we have surveyed are prohibitive and hence these methods are impractical. Specifically, direct numerical integration can take excessively long under weak damping, the memory need is significant for the collocation method, and the harmonic balance method is impacted by the difficulty of finding zeros for very large dimensional, nonlinear systems of algebraic equations.

\begin{sloppypar}
To reduce the computational cost, one often reduces high-dimensional systems to lower-dimensional models whose FRC can be extracted efficiently. For linear systems, decomposition into \emph{normal modes} provides a powerful tool to derive reduced-order models. For nonlinear systems, various definitions of \emph{nonlinear normal modes} (NNMs) have been developed. Specifically, Rosenberg~\cite{rosenberg1966nonlinear} defines a NNM as a synchronous periodic orbit of a \emph{conservative} system. Shaw \& Pierre~\cite{shaw1993normal} define a NNM as an \emph{invariant manifold} tangent at the origin to a linear modal subspace for a \emph{dissipative} system. It follows that the NNM is the nonlinear continuations of the linear modal subspace and hence can be used for model order reduction. Shaw and his co-workers have used Garlerkin-based approaches to calculate such NNMs for dispative systems~\cite{pesheck2002new}, with the consideration of {internal resonances}~\cite{jiang2005construction} and {harmonic excitation}~\cite{jiang2005nonlinear}, and derived reduced-order models using the obtained NNMs. 
\end{sloppypar}

\begin{sloppypar}
It has been observed that the Shaw–Pierre-type invariant surfaces are not unique even in the linearized system~\cite{neild2015use}. While there are generally infinitely many Shaw-Pierre-type invariant manifolds for each modal subspace, there exists a unique smoothest one under appropriate non-resonance conditions, as pointed out by Haller \& Ponsioen~\cite{haller2016nonlinear}. They define the smoothest invariant manifold to a spectral subspace (i.e., a direct sum of modal subspaces) as the \emph{spectral submanifold} (SSM) associated with the spectral subspace. Parameterization methods with tensor-notation~\cite{ponsioen2018automated} and multi-index notation~\cite{ponsioen2020model} have been developed to efficiently compute such SSMs. The reduced-order model for a particular mode of interest can be derived with the corresponding two-dimensional SSM. Such a reduced-order model enables explicit extraction of the backbone curve~\cite{szalai2017nonlinear,breunung2018explicit} and the FRC~\cite{breunung2018explicit,ponsioen2019analytic,ponsioen2020model} around the particular mode. In addition, isolated FRCs, namely, \emph{isolas}, can be analytically predicted with such a reduced-order model~\cite{ponsioen2019analytic}.
\end{sloppypar}

\begin{sloppypar}
Two main limitations of SSM computation in the above works are (i) reliance on the equations of motion written in the eigenbasis of the linearized systems, which is out of reach for FE problems involving very large number of degrees of freedom, and (ii) the dimension of SSM is restricted to two. Addressing these limitations, Jain \& Haller~\cite{SHOBHIT} have recently developed a computational methodology that enables local approximations to SSMs of \emph{arbitrary dimensions} up to \emph{arbitrary orders} of accuracy using only the knowledge of eigenvectors associated to the master modal subspace. A numerical implementation of these results is available in the open-source MATLAB package, SSMTool-2.0~\cite{ssmtool2}, which is capable of treating very high-dimensional finite element applications~\cite{SHOBHIT}. Model reduction to SSMs for systems with internal resonances, however, have not yet been addressed, which motivates our current study.

An alternative procedure for model reduction of nonlinear systems is the method of normal form. This method applies successive near-identity transformations to the equations of motion to remove non-resonant terms, yielding simplified equations of motion which contain only the essential (resonant) terms. Touz{\'e} and Amabili~\cite{touze2006nonlinear} have used the method of normal form first to derive reduced-order models for harmonically forced structures. These reduced-order models are obtained by restricting the truncated normal form to its invariant subspaces aligned with the modal subspaces of the linearized system. Hence, this procedure requires the full system to be expressed in its modal basis. Similarly, Neild \& Wagg~\cite{neild2011applying} applied the method of normal form for second-order systems directly.  The simplified dynamics from the normal form procedure enables analytical prediction of backbone curves~\cite{cammarano2014bifurcations} as well as FRCs~\cite{touze2006nonlinear} for systems with internal resonance. Recently, Vizzaccaro et al.~\cite{vizzaccaro2020direct} and Opreni et al.~\cite{opreni2021model} computed the reduced-order models of \cite{touze2006nonlinear} directly from physical coordinates up to cubic order of truncation. These procedures uses the same SSM parametrization approach put forward in~\cite{haller2016nonlinear,ponsioen2018automated,veraszto2020explicit,ponsioen2020model}, but is limited to geometric nonlinearities up to cubic order and to linear Rayleigh damping (cf. Jain \& Haller~\cite{SHOBHIT}). 
\end{sloppypar}

\begin{sloppypar}
\textcolor{black}{The objective of this paper is to derive reduced-order models for harmonically excited mechanical systems with internal resonances using SSMs and to extract the FRCs of such systems up to arbitrary orders of approximation. The rest of this paper is organized as follows.} Section~\ref{sec:setup} details the setup of mechanical systems. In section~\ref{sec:ssm-theory}, SSM-based reduction is discussed for systems with internal resonance. Specifically, we consider a system with $m$ of its natural frequencies satisfying a certain internal resonance relation. Then, the reduced-order model on a resonant SSM is $2m$-dimensional, independently of the dimension of the original system. Section~\ref{sec:comp-ssm} describes the computational procedure for resonant SSMs. In section~\ref{sec:redyn-ssm}, the reduced dynamics on the SSM is analyzed in detail. As we will see, the equilibrium points of the slow-phase reduced dynamics mark periodic orbits of the full system. The stability of the periodic orbits is the same as that of the equilibrium points. It follows that the extraction of the FRC of the full system is reduced to the computation of the manifold of equilibria in the reduced-order vector field, which can be easily and efficiently performed. We discuss a MATLAB toolbox developed to perform such calculations. Section~\ref{sec:examp} demonstrates the power of this toolbox with a list of examples, including von K\'arm\'an beam and plate structures with discretizations up to 240,000 degrees of freedom.
In Part II of this paper, we will focus on the bifurcation of periodic orbits, including quasi-periodic tori bifurcating from periodic orbits.
\end{sloppypar}

\section{System setup}
\label{sec:setup}
\begin{sloppypar}
We consider a periodically forced nonlinear mechanical system
\begin{equation}
\label{eq:eom-second-full}
\boldsymbol{M}\ddot{\boldsymbol{x}}+\boldsymbol{C}\dot{\boldsymbol{x}}+\boldsymbol{K}\boldsymbol{x}+\boldsymbol{f}(\boldsymbol{x},\dot{\boldsymbol{x}})=\epsilon \boldsymbol{f}^{\mathrm{ext}}(\Omega t),\quad 0<\epsilon\ll1
\end{equation}
where $\boldsymbol{x}\in\mathbb{R}^n$ is the generalized displacement vector; $\boldsymbol{M}\in\mathbb{R}^{n\times n}$ is the positive definite mass matrix; $\boldsymbol{C},\boldsymbol{K}\in\mathbb{R}^{n\times n}$ are the damping and stiffness matrices; $\boldsymbol{f}(\boldsymbol{x},\dot{\boldsymbol{x}})$ is a $C^r$ smooth nonlinear function such that
$\boldsymbol{f}(\boldsymbol{x},\dot{\boldsymbol{x}})\sim \mathcal{O}(|\boldsymbol{x}|^2,|\boldsymbol{x}||\dot{\boldsymbol{x}}|,|\dot{\boldsymbol{x}}|^2)$; and $\epsilon \boldsymbol{f}^{\mathrm{ext}}(\Omega t)$ denotes external harmonic excitation.
\end{sloppypar}

The above second-order system can be transformed into a first-order system as follows
\begin{equation}
\label{eq:full-first}
\boldsymbol{B}\dot{\boldsymbol{z}}	=\boldsymbol{A}\boldsymbol{z}+\boldsymbol{F}(\boldsymbol{z})+\epsilon\boldsymbol{F}^{\mathrm{ext}}({\Omega t})
\end{equation}
where
\begin{gather}
\label{eq:zABF}
\boldsymbol{z}=\begin{pmatrix}\boldsymbol{x}\\\dot{\boldsymbol{x}}\end{pmatrix},\,\,
\boldsymbol{A}=\begin{pmatrix}-\boldsymbol{K} 
& \boldsymbol{0}\\\boldsymbol{0} & \boldsymbol{M}\end{pmatrix},\,\,
\boldsymbol{B}=\begin{pmatrix}\boldsymbol{C} 
& \boldsymbol{M}\\\boldsymbol{M} & \boldsymbol{0}\end{pmatrix},\nonumber\\
\boldsymbol{F}(\boldsymbol{z})=\begin{pmatrix}\boldsymbol{-\boldsymbol{f}(\boldsymbol{x},\dot{\boldsymbol{x}})}\\\boldsymbol{0}\end{pmatrix},\,\,
\boldsymbol{F}^{\mathrm{ext}}(\Omega t) = \begin{pmatrix}\boldsymbol{f}^{\mathrm{ext}}(\Omega t)\\\boldsymbol{0}\end{pmatrix}.
\end{gather}
\textcolor{black}{One benefit of the first-order formulation~\eqref{eq:zABF} is that the coefficient matrices $\boldsymbol{A}$ and $\boldsymbol{B}$ are symmetric when the matrices $\boldsymbol{M}, \boldsymbol{C}, \boldsymbol{K}$ are symmetric, which is often the case for mechanics problems. Nonetheless, we formulate our computation procedure for the general first-order system~\eqref{eq:full-first}.}

Solving the linear part of~\eqref{eq:full-first} leads to the generalized eigenvalue problem
\begin{equation}
\boldsymbol{A}\boldsymbol{v}_j=\lambda_j\boldsymbol{B}\boldsymbol{v}_j,\quad \boldsymbol{u}_j^\ast \boldsymbol{A}=\lambda_j \boldsymbol{u}_j^\ast \boldsymbol{B},
\end{equation}
where $\lambda_j$ is a generalized eigenvalue and $\boldsymbol{v}_j$ and $\boldsymbol{u}_j$ are the corresponding \emph{right} and \emph{left} eigenvectors, respectively. This eigenvalue problem has $2n$ eigenvalues, which can be sorted in the decreasing order based on their real parts
\begin{equation}
\mathrm{Re}(\lambda_{2n})\leq\mathrm{Re}(\lambda_{2n-1})\leq\cdots\leq\mathrm{Re}(\lambda_{1})<0.
\label{eq:eig_sort}
\end{equation}
\textcolor{black}{In this work}, we have assumed that the real parts of all eigenvalues are strictly less than zero and hence the equilibrium point of the linearized system $\boldsymbol{B}\dot{\boldsymbol{z}}=\boldsymbol{A}\boldsymbol{z}$ is asymptotically stable. 
\begin{remark}
We have listed all eigenvalues here for completeness. However, as we will see, it is not necessary to calculate all eigenvalues in SSM analysis because \textcolor{black}{the computation procedure of SSM proposed} in~\cite{SHOBHIT} is used in this study. \textcolor{black}{In this procedure, invariant manifolds and their reduced dynamics are computed in physical coordinates using only the master modes associated with the invariant manifold.}
\end{remark}
\textcolor{black}{
\begin{remark}
We sort the eigenvalues~\eqref{eq:eig_sort} based on their real parts following~\cite{haller2016nonlinear}. This ordering is useful in identifying the slowest decaying modes. The SSMs constructed around the slowest modes are the most relevant for model reduction as they attract nearby full system trajectories~\cite{haller2016nonlinear}. To this end, the \texttt{eigs} routine in \textsc{matlab} can be used to compute a small subset of eigenmodes with the smallest real parts. For the commonly employed Rayleigh damping model in structural dynamics, i.e.,
\begin{equation}
\label{eq:RayleighDamping}
\boldsymbol{C}=\alpha\boldsymbol{M}+\beta\boldsymbol{K},
\end{equation}
the eigenvalues of the linear system are given by
\begin{align}
\label{eq:weakDampFreq}
\lambda_{2i-1,2i}& =-\frac{\alpha+\beta\omega_i^2}{2}\pm\mathrm{i}\omega_i\sqrt{1-
\left(\frac{\alpha}{2\omega_i}+\frac{\beta\omega_i}{2}\right)^2},
\end{align}
where $\omega_i$ denotes the $i$-th natural frequency of the undamped linear system. We note that with  $0\leq\alpha\ll\omega_i$ and $0<\beta\ll1$, i.e., under light damping, the ordering~\eqref{eq:eig_sort} provides the commonly used ordering of increasing natural frequencies.
\end{remark}
}

\section{Non-autonomous SSM for systems with internal resonance}
\label{sec:ssm-theory}
We consider the following $2m$-dimensional \emph{master}  spectral subspace 
\begin{equation}
\mathcal{E}=\Span\{\boldsymbol{v}^\mathcal{E}_1,\bar{\boldsymbol{v}}^\mathcal{E}_1,\cdots,\boldsymbol{v}^\mathcal{E}_m,\bar{\boldsymbol{v}}^\mathcal{E}_m\}.
\end{equation}
We assume that $\mathcal{E}$ is underdamped, i.e., its spectrum is of the following
\begin{equation}
\spect(\mathcal{E}) = \{\lambda^\mathcal{E}_1,\bar{\lambda}^\mathcal{E}_1,\cdots,\lambda^\mathcal{E}_m,\bar{\lambda}^\mathcal{E}_m\}
\end{equation}
with $\mathrm{Im}(\lambda_j^\mathcal{E})\neq0$ for $j=1,\cdots,m$. \textcolor{black}{We expect the spectral subspace $\mathcal{E}$ to be composed of internally resonant modes of the system. As such, the eigenvalues in $\spect(\mathcal{E})$ may be any arbitrary subset of the $2n$ eigenvalues in the ordering \eqref{eq:eig_sort}.}

We further assume that the algebraic multiplicity of each eigenvalue in $\spect(\mathcal{E})$ is equal to the geometric multiplicity of the eigenvalue. The eigenvectors are then chosen such that
\begin{equation}
\label{eq:ortho-norm-eigenvectors}
\left(\boldsymbol{u}_i^{\mathcal{E}}\right)^\ast\boldsymbol{B}\boldsymbol{v}_j^{\mathcal{E}}=\delta_{ij},\,\, \left(\bar{\boldsymbol{u}}_i^{\mathcal{E}}\right)^\ast\boldsymbol{B}\boldsymbol{v}_j^{\mathcal{E}}=0,\quad 1\leq i,j\leq m.
\end{equation}
\textcolor{black}{Under the assumption of small damping, we have small real parts for the eigenvalues of lower-frequency modes. In the case of internal resonance, this results in (near) resonances among the imaginary parts of the eigenvalues corresponding to the internally resonant modes. To this end,} we allow for the following type of (near) \emph{inner} resonances
\begin{equation}
\label{eq:res-inner}
    \lambda_i^\mathcal{E}\approx\boldsymbol{l}\cdot\boldsymbol{\lambda}^\mathcal{E}+\boldsymbol{j}\cdot\bar{\boldsymbol{\lambda}}^\mathcal{E},\quad \bar{\lambda}_i^\mathcal{E}\approx\boldsymbol{j}\cdot\boldsymbol{\lambda}^\mathcal{E}+\boldsymbol{l}\cdot\bar{\boldsymbol{\lambda}}^\mathcal{E}
\end{equation}
for some $i\in\{1,\cdots,m\}$, where $\boldsymbol{l},\boldsymbol{j}\in\mathbb{N}_0^m,\,\, |\boldsymbol{l}+\boldsymbol{j}|:=\sum_{k=1}^m (l_k+j_k)\geq2$, and
\begin{equation}
\boldsymbol{\lambda}^\mathcal{E}=(\lambda^\mathcal{E}_1,\cdots,\lambda^\mathcal{E}_m).
\end{equation}

Following Haller and Ponsioen~\cite{haller2016nonlinear}, we define a \emph{periodic spectral submanifold} (SSM) with period ${2\pi}/{\Omega}$, $\mathcal{W}(\mathcal{E},\Omega t)$, corresponding to the master spectral subspace $\mathcal{E}$ as a $2m$-dimensional \emph{invariant} manifold to the nonlinear system~\eqref{eq:full-first} such that $\mathcal{W}(\mathcal{E},\Omega t)$
\begin{enumerate}[label=(\roman*)]
\item perturbs smoothly from $\mathcal{E}$ at the trivial equilibrium point $\boldsymbol{z}=0$ under the addition of nonlinear terms and external excitation in~\eqref{eq:full-first}, and
\item is strictly smoother than any other periodic invariant manifolds with period ${2\pi}/{\Omega}$ that satisfies (i).
\end{enumerate}

The existence and uniqueness of such SSMs have been investigated in~\cite{haller2016nonlinear}. We summarize the main results in the following theorem.

\begin{theorem}
\label{th:SSM-existence-uniqueness}
Assume the \emph{non-resonance} condition
\begin{gather}
\boldsymbol{a}\cdot\mathrm{Re}(\boldsymbol{\lambda}^\mathcal{E})+\boldsymbol{b}\cdot\mathrm{Re}(\bar{\boldsymbol{\lambda}}^\mathcal{E})\neq \mathrm{Re}(\lambda_k),\nonumber\\
\forall\,\,\lambda_k\in\spect(\boldsymbol{\Lambda})\setminus\spect(\mathcal{E}),\nonumber\\
\forall\,\,\boldsymbol{a},\boldsymbol{b}\in\mathbb{N}_0^m,\,\,2\leq |\boldsymbol{a}+\boldsymbol{b}|\leq\Sigma(\mathcal{E}),
\end{gather}
where $|\boldsymbol{a}+\boldsymbol{b}|=\sum_{k=1}^m(a_k+b_k)$ and $\Sigma(\mathcal{E})$ is the \emph{absolute spectral quotient} of $\mathcal{E}$, defined as
\begin{equation}
\Sigma(\mathcal{E}) = \mathrm{Int}\left(\frac{\min_{\lambda\in\spect(\boldsymbol{\Lambda})}\mathrm{Re}\lambda}{\max_{\lambda\in\spect(\mathcal{E})}\mathrm{Re}\lambda}\right).
\end{equation}
Then, for \textcolor{black}{$\epsilon>0$} small enough, the following hold for system~\eqref{eq:full-first}:
\begin{enumerate}[label=(\roman*)]
\item \textcolor{black}{There exists a $2m$-dimensional, time-periodic SSM $\mathcal{W}(\mathcal{E},\Omega t)$ that depends smoothly on $\epsilon$,}\\
\item \textcolor{black}{The SSM $\mathcal{W}(\mathcal{E},\Omega t)$ is unique among all $C^{\Sigma(\mathcal{E})+1}$ invariant manifolds satisfying (i)}
\item $\mathcal{W}(\mathcal{E},\Omega t)$ can be viewed as an embedding \textcolor{black}{of an open set in the reduced coordinates $(\boldsymbol{p},\phi)$} into the phase space of system~\eqref{eq:full-first} via the map
\begin{equation}
\boldsymbol{W}_{\epsilon}(\boldsymbol{p},\phi):\textcolor{black}{\mathbb{C}^{2m}}\times{S}^1\to\mathbb{R}^{2n}\quad .
\end{equation}
\item There exists a polynomial function \textcolor{black}{$\boldsymbol{R}_{\epsilon}(\boldsymbol{p},\phi):\mathbb{C}^{2m}\times{S}^1\to \mathbb{C}^{2m}$} satisfying the invariance equation
\begin{align}
\label{eq:invariance}
& \boldsymbol{B}\left({D}_{\boldsymbol{p}}\boldsymbol{W}_{\epsilon}(\boldsymbol{p},\phi) \boldsymbol{R}_{\epsilon}(\boldsymbol{p},\phi)+{D}_{\phi}\boldsymbol{W}_{\epsilon}(\boldsymbol{p},\phi) \Omega\right)\nonumber\\
&=\boldsymbol{A}\boldsymbol{W}_{\epsilon}(\boldsymbol{p},\phi)+\boldsymbol{F}( \boldsymbol{W}_{\epsilon}(\boldsymbol{p},\phi))+\epsilon\boldsymbol{F}^{\mathrm{ext}}({\phi}),
\end{align}
such that the reduced dynamics on the SSM can be expressed as
\begin{equation}
\label{eq:ssm-red-full}
\dot{\boldsymbol{p}} = \boldsymbol{R}_{\epsilon}(\boldsymbol{p},\phi),\quad \dot{\phi}=\Omega.
\end{equation}
\end{enumerate}
\end{theorem}
\begin{proof}
This theorem is simply a restatement of Theorem 4 by Haller and Ponsioen~\cite{haller2016nonlinear}, which is based on more abstract results by Cabr\'e et al.~\cite{cabre2003parameterization-i,cabre2003parameterization-ii,cabre2005parameterization-iii} and Haro and de la Llave~\cite{haro2006parameterization,haro2006parameterization-num}.
\end{proof}

\begin{remark}
To check the \emph{non-resonance} condition in the above theorem, we need to know all eigenvalues, which are not available in general for high-dimensional systems. Indeed, the computation of \emph{all} natural frequencies of a high-dimensional system is computationally expensive and challenging. In practice, we only calculate a subset of eigenvalues in SSM analysis. For instance, we may calculate the first $n_\mathrm{s}$ modes with lowest natural frequencies and then find the \emph{inner} resonance among a subset of these $n_\mathrm{s}$ modes to determine the \emph{master} subspace. Then the \emph{non-resonance} condition is checked for the $n_\mathrm{s}$ modes.
\end{remark}
\begin{remark}
The parameterization coordinates $\boldsymbol{p}$ are $m$ pairs of complex conjugate coordinates, namely,
\begin{equation}
\boldsymbol{p}=(q_1,\bar{q}_1,\cdots,q_m,\bar{q}_m),
\end{equation}
where $q_i$ and $\bar{q}_i$ denote the parameterization coordinates corresponding to $\boldsymbol{v}_i^{\mathcal{E}}$ and $\bar{\boldsymbol{v}}_i^{\mathcal{E}}$, respectively. In this paper, we refer to such coordinates as \emph{normal} coordinates as well because they characterize the reduced dynamics on SSM.
\end{remark}

\section{Computation of SSM}
\label{sec:comp-ssm}
In this section, we briefly review the computation procedure developed by Jain \& Haller~\cite{SHOBHIT}, which enables computation of SSMs in physical coordinates using only the eigenvectors associated to the master modal subspace $\mathcal{E}$.~\textcolor{black}{The procedure in~\cite{SHOBHIT} is based on the parameterization method for invariant manifolds (see Haro et al.~\cite{haro2016parameterization} for an overview).}

We seek the unknown parametrizations $\boldsymbol{W}_{\epsilon}(\boldsymbol{p},\phi)$ and $\boldsymbol{R}_{\epsilon}(\boldsymbol{p},\phi)$ as an asymptotic series in $\epsilon$ given their smooth dependence on $\epsilon$. It follows that
\begin{gather}
    \boldsymbol{W}_{\epsilon}(\boldsymbol{p},\phi)=\boldsymbol{W}(\boldsymbol{p})+\epsilon \boldsymbol{X}(\boldsymbol{p},\phi)+\mathcal{O}(\epsilon^2),\label{eq:ssm-exp-eps}\\
    \boldsymbol{R}_{\epsilon}(\boldsymbol{p},\phi)=\boldsymbol{R}(\boldsymbol{p})+\epsilon \boldsymbol{S}(\boldsymbol{p},\phi)+\mathcal{O}(\epsilon^2)\label{eq:red-exp-eps}.
\end{gather}
Substituting the above expansions into the invariance equation~\eqref{eq:invariance} and collecting terms according to the order of $\epsilon$ yields
\begin{equation}
\label{eq:SSM-auto-eq}
    \mathcal{O}(\epsilon^0):\,\, \boldsymbol{B}{D}_{\boldsymbol{p}}\boldsymbol{W}(\boldsymbol{p})\boldsymbol{R}(\boldsymbol{p})=\boldsymbol{A}\boldsymbol{W}(\boldsymbol{p})+\boldsymbol{F}(\boldsymbol{W}(\boldsymbol{p})),
\end{equation}
which \textcolor{black}{turns out} the same as the invariance equation for the \textcolor{black}{autonomous SSM in the $\epsilon=0$ (unforced) limit of system~\eqref{eq:eom-second-full}}. Furthermore, we obtain
\begin{align}
\label{eq:SSM-nonauto-eq}
    \mathcal{O}(\epsilon):\,\, & \boldsymbol{B}{D}_{\boldsymbol{p}}\boldsymbol{W}(\boldsymbol{p})\boldsymbol{S}(\boldsymbol{p},\phi)+\boldsymbol{B}D_{\boldsymbol{p}}\boldsymbol{X}(\boldsymbol{p},\phi)\boldsymbol{R}(\boldsymbol{p})\nonumber\\
    & +\boldsymbol{B}D_{\phi}\boldsymbol{X}(\boldsymbol{p},\phi)\Omega =\boldsymbol{A}\boldsymbol{X}(\boldsymbol{p},\phi)\nonumber\\
    & +D\boldsymbol{F}(\boldsymbol{W}(\boldsymbol{p}))\boldsymbol{X}(\boldsymbol{p},\phi)+\boldsymbol{F}^{\mathrm{ext}}(\phi).
\end{align}

\subsection{Autonomous part}
We first solve~\eqref{eq:SSM-auto-eq} to obtain a Taylor expansion for the autonomous SSM $\boldsymbol{W}(\boldsymbol{p})$ and its reduced dynamics $\boldsymbol{R}(\boldsymbol{p})$ on it. The basic idea of solving~\eqref{eq:SSM-auto-eq} is summarized here but we refer to~\cite{SHOBHIT} for more details. Specifically, a Taylor series is used to expand $\boldsymbol{W}(\boldsymbol{p})$ and $\boldsymbol{R}(\boldsymbol{p})$ in the normal coordinates $\boldsymbol{p}$
\begin{equation}
\label{eq:auto-expansion}
    \boldsymbol{W}(\boldsymbol{p})=\sum_{\boldsymbol{k}} \boldsymbol{w}_{\boldsymbol{k}}\boldsymbol{p}^{\boldsymbol{k}},\,\,
    \boldsymbol{R}(\boldsymbol{p})=\sum_{\boldsymbol{k}} \boldsymbol{r}_{\boldsymbol{k}}\boldsymbol{p}^{\boldsymbol{k}},\,\, |\boldsymbol{k}|\geq1,
\end{equation}
where $\boldsymbol{p}^{\boldsymbol{k}}=p_1^{k_1}\cdot${\dots}$\cdot p_{2m}^{k_{2m}}$ and $|\boldsymbol{k}|=k_1+\cdots+k_{2m}$. We have omitted the leading order $(|\boldsymbol{k}|=0)$ terms in the expansions because $\boldsymbol{W}(\boldsymbol{0})=\boldsymbol{0}$ and $\boldsymbol{R}(\boldsymbol{0})=\boldsymbol{0}$. Substituting~\eqref{eq:auto-expansion} into~\eqref{eq:SSM-auto-eq} and balancing the terms of $\boldsymbol{p}^{\boldsymbol{k}}$ for $\boldsymbol{k}$ satisfying $|\boldsymbol{k}|=j$ yields a set of linear equations of the form
\begin{equation}
\label{eq:ssm-auto-it}
\mathcal{A}_{\boldsymbol{k}}\boldsymbol{w}_{\boldsymbol{k}}=\mathcal{B}_{\boldsymbol{k}}\boldsymbol{r}_{\boldsymbol{k}}-\mathcal{C}_{\boldsymbol{k}},\quad |\boldsymbol{k}|=j,
\end{equation}
where $\mathcal{A}_{\boldsymbol{k}}$, $\mathcal{B}_{\boldsymbol{k}}$ and $\mathcal{C}_{\boldsymbol{k}}$ depend on the expansion coefficients at lower order if $j\geq2$. When $j=1$, the expansion coefficients are related to the master subspace $\mathcal{E}$ and can be solved for directly. Subsequently, we can solve the linear equations~\eqref{eq:ssm-auto-it} recursively to obtain the expansion coefficients at higher orders.

As a demonstration of the above procedure, we consider the case $j=1$. Let ${\boldsymbol{e}}_i\in\mathbb{R}^{2m}$ be the unit vector aligned along the $i$-th coordinate axis of. It follows that $|{\boldsymbol{e}}_i|=1$ for $1\leq i\leq 2m$ and we have
\begin{equation}
\label{eq:ssm-auto-first-order}
\boldsymbol{B}\sum_{{\boldsymbol{e}}_i}\boldsymbol{w}_{{\boldsymbol{e}}_i}\sum_{{\boldsymbol{e}}_j}(\boldsymbol{r}_{{\boldsymbol{e}}_j})_{i}\boldsymbol{p}^{{\boldsymbol{e}}_j}=\boldsymbol{A}\sum_{{\boldsymbol{e}}_j}\boldsymbol{w}_{{\boldsymbol{e}}_j}\boldsymbol{p}^{{\boldsymbol{e}}_j}.
\end{equation}
With the notation
\begin{equation}
\boldsymbol{W}_{\mathbf{I}}=(\boldsymbol{w}_{{\boldsymbol{e}}_1},\cdots,\boldsymbol{w}_{{\boldsymbol{e}}_{2m}}),\quad 
\boldsymbol{R}_{\mathbf{I}}=(\boldsymbol{r}_{{\boldsymbol{e}}_1},\cdots,\boldsymbol{r}_{{\boldsymbol{e}}_{2m}}),
\end{equation}
balancing the two sides of~\eqref{eq:ssm-auto-first-order} yields
\begin{equation}
\boldsymbol{B}\boldsymbol{W}_{\mathbf{I}}\boldsymbol{R}_{\mathbf{I}}=\boldsymbol{A}\boldsymbol{W}_{\mathbf{I}},
\end{equation}
from which we obtain
\begin{gather}
\label{eq:auto-ssm-first-order}
\boldsymbol{W}_{\mathbf{I}} = (\boldsymbol{v}^\mathcal{E}_1,\bar{\boldsymbol{v}}^\mathcal{E}_1,\cdots,\boldsymbol{v}^\mathcal{E}_m,\bar{\boldsymbol{v}}^\mathcal{E}_m),\\
\boldsymbol{R}_{\mathbf{I}}=\diag(\lambda^\mathcal{E}_1,\bar{\lambda}^\mathcal{E}_1,\cdots,\lambda^\mathcal{E}_m,\bar{\lambda}^\mathcal{E}_m).
\end{gather}
Hence, the eigenvectors and eigenvalues associated to the master spectral subspace $\mathcal{E}$ solve the autonomous invariance equations~\eqref{eq:SSM-auto-eq} at the leading order, $j=1$. Using this solution at the leading order, the linear equations~\eqref{eq:ssm-auto-it} can be recursively solved to approximate the autonomous SSM \textcolor{black}{up to} arbitrarily high orders ($j\ge2$) of accuracy. We refer to~\cite{SHOBHIT} for details on the higher-order case.

\begin{sloppypar}
Now, let the autonomous part of the vector field of the reduced dynamics be arranged in complex conjugate blocks as follows
\begin{equation}
\label{eq:red-auto-block}
    \boldsymbol{R}(\boldsymbol{p})=\begin{pmatrix}\boldsymbol{R}_{1}(\boldsymbol{p})\\\vdots\\\boldsymbol{R}_{m}(\boldsymbol{p})\end{pmatrix},
\end{equation}
where $\boldsymbol{R}_{i}(\boldsymbol{p})\in\mathbb{C}^2$ contains the complex conjugate components of the autonomous part of the vector field associated to the $i$-th pair of master mode $(\boldsymbol{v}_i^{\mathcal{E}},\bar{\boldsymbol{v}}_i^{\mathcal{E}})$. Under the (near) inner resonances given by~\eqref{eq:res-inner}, we define
a set containing the corresponding monomial multi-indices as
\begin{equation}
\label{eq:def-Ri}
\mathcal{R}_i=\{(\boldsymbol{l},\boldsymbol{j}): \lambda_i^\mathcal{E}\approx\boldsymbol{l}\cdot\boldsymbol{\lambda}^\mathcal{E}+\boldsymbol{j}\cdot\bar{\boldsymbol{\lambda}}^\mathcal{E}\}.
\end{equation}
Then it follows from the result of~\cite{SHOBHIT} that the normal-form-style parameterization of the autonomous reduced dynamics is given by
\begin{equation}
\label{thm:ssm-auto}
    \boldsymbol{R}_{i}(\boldsymbol{p})=\begin{pmatrix}\lambda_i^{\mathcal{E}}q_i\\\bar{\lambda}_i^{\mathcal{E}}\bar{q}_i\end{pmatrix}+\sum_{(\boldsymbol{l},\boldsymbol{j})\in\mathcal{R}_i}\begin{pmatrix}\gamma(\boldsymbol{l},\boldsymbol{j})\boldsymbol{q}^{\boldsymbol{l}}\bar{\boldsymbol{q}}^{\boldsymbol{j}}\\\bar{\gamma}(\boldsymbol{l},\boldsymbol{j})\boldsymbol{q}^{\boldsymbol{j}}\bar{\boldsymbol{q}}^{\boldsymbol{l}}\end{pmatrix},
\end{equation}
where the normal form coefficients $\gamma(\boldsymbol{l},\boldsymbol{j})$  along with the expansion coefficients of $\boldsymbol{W}(\boldsymbol{p})$ are obtained using the computation method in~\cite{SHOBHIT}.
\end{sloppypar}

\begin{remark}
The computational cost for formulating and solving \eqref{eq:ssm-auto-it} is significant for large $j$. In practice, the expansion is truncated at some order $j_{\max}$. It follows that $j\leq j_{\max}\leq r$ in \eqref{eq:ssm-auto-it} and $j_{\max}$ is referred to as the \emph{expansion order} of SSM. In this paper, we determine the necessary expansion order based on the convergence of the FRC under increasing order, given that the computed approximate SSM will converge to the unique $C^{\Sigma(\mathcal{E})+1}$-smooth SSM as the order of approximation, $j$, increases.
\end{remark}

\subsection{Non-autonomous part}
With $\boldsymbol{W}(\boldsymbol{p})$ and $\boldsymbol{R}(\boldsymbol{p})$ at hand, we solve~\eqref{eq:SSM-nonauto-eq} to obtain $\boldsymbol{X}(\boldsymbol{p},\phi)$ and $\boldsymbol{S}(\boldsymbol{p},\phi)$. Likewise, Taylor expansion in $\boldsymbol{p}$ is used to approximate $\boldsymbol{X}$ and $\boldsymbol{S}$. The expansion coefficients here are not constant but functions of $\phi$ and hence periodic. \textcolor{black}{In this work, we restrict ourselves to a leading-order approximation in $\boldsymbol{p}$ for $\boldsymbol{X}$ and $\boldsymbol{S}$~\cite{SHOBHIT,breunung2018explicit}, i.e.,}
\begin{equation}
\label{eq:leading-nonauto}
\textcolor{black}{
\begin{gathered}
    \boldsymbol{X}(\boldsymbol{p},\phi)=\boldsymbol{X}_{\boldsymbol{0}}(\phi) + \mathcal{O}(|\boldsymbol{p}|),\\
    \boldsymbol{S}(\boldsymbol{p},\phi)=\boldsymbol{S}_{\boldsymbol{0}}(\phi) + \mathcal{O}(|\boldsymbol{p}|).
\end{gathered}}    
\end{equation}
Then, the reduced dynamics~\eqref{eq:ssm-red-full} takes the form
\begin{equation}
\label{eq:red-nonauto-lead}
    \dot{\boldsymbol{p}}=\boldsymbol{R}(\boldsymbol{p})+\epsilon\boldsymbol{S}_{\boldsymbol{0}}(\phi)+\mathcal{O(\epsilon|\boldsymbol{p}|)}.
\end{equation}
Similar to~\eqref{eq:red-auto-block}, we arrange the non-autonomous part of the vector field of reduced dynamics in complex conjugate blocks as follows
\begin{equation}
\label{eq:red-nonaut-block}
    \boldsymbol{S}_{\boldsymbol{0}}(\phi)=\begin{pmatrix}\boldsymbol{S}_{\boldsymbol{0},1}(\phi)\\\vdots\\\boldsymbol{S}_{\boldsymbol{0},m}(\phi)\end{pmatrix},
\end{equation}
where $\boldsymbol{S}_{\boldsymbol{0},i}(\phi)\in\mathbb{C}^2$ contains the complex conjugate components of the leading-order non-autonomous part of the vector field associated with the $i$-th pair of master modes, $(\boldsymbol{v}_i^{\mathcal{E}},\bar{\boldsymbol{v}}_i^{\mathcal{E}})$. Let
\begin{equation}
\label{eq:forcing-conj}
\boldsymbol{F}^{\mathrm{ext}}(\phi)={\boldsymbol{F}}^\mathrm{a}e^{\mathrm{i}\phi}+{\boldsymbol{F}^\mathrm{a}}e^{-\mathrm{i}\phi},
\end{equation}
where the forcing amplitude vector ${\boldsymbol{F}}^{\mathrm{a}}\in\mathbb{R}^{2n}$ with superscript `a' stands for `amplitude'. It follows then from the derivation in Appendix~\ref{sec:proof-thm-ssm-nonauto} that
\begin{equation}
\label{thm:ssm-nonauto}
\boldsymbol{S}_{\boldsymbol{0},i}(\phi)=\begin{pmatrix}{{S}}_{\boldsymbol{0},i}e^{\mathrm{i}\phi}\\\bar{{S}}_{\boldsymbol{0},i}e^{-\mathrm{i}\phi}\end{pmatrix},\quad i=1,\cdots,m,
\end{equation}
with
\begin{equation}
\label{eq:ssm-nonauto}
{{S}}_{\boldsymbol{0},i}=\left\lbrace \begin{array}{cl}
      (\boldsymbol{u}_i^{\mathcal{E}})^\ast\boldsymbol{F}^\mathrm{a}   & \text{if} \hspace{2mm} \lambda_i^{\mathcal{E}}\approx\mathrm{i}\Omega \\
      {0}   & \text{otherwise}
    \end{array} \right..
\end{equation}
In addition, letting $\boldsymbol{S}_{\boldsymbol{0}}(\phi)=\boldsymbol{s}_{\boldsymbol{0}}^+e^{\mathrm{i}\phi} + {\boldsymbol{s}}_{\boldsymbol{0}}^-e^{-\mathrm{i}\phi}$, we obtain
\begin{equation}
\label{eq:nonautossm-exp}
\boldsymbol{X}_{\boldsymbol{0}}(\phi)=\boldsymbol{x}_{\boldsymbol{0}}e^{\mathrm{i}\phi}+\bar{\boldsymbol{x}}_{\boldsymbol{0}}e^{-\mathrm{i}\phi},
\end{equation}
where $\boldsymbol{x}_{\boldsymbol{0}}$ is the solution to the system of linear equations
\begin{equation}
\label{eq:expphi-}
(\boldsymbol{A}-\mathrm{i}\Omega\boldsymbol{B})\boldsymbol{x}_{\boldsymbol{0}}=\boldsymbol{B}\boldsymbol{W}_{\mathbf{I}}\boldsymbol{s}_{\boldsymbol{0}}^+-\boldsymbol{F}^{\mathrm{a}}.
\end{equation}

\section{Reduced dynamics on SSM}
\label{sec:redyn-ssm}
\begin{sloppypar}
\textcolor{black}{In this section}, we establish the form of the leading-order reduced dynamics on a multi-dimensional, time-periodic SSM with internal resonance. As the SSM is an attracting slow manifold, its reduced dynamics will serve as a reduced-order model for the evolution of all nearby initial conditions. \textcolor{black}{In the special case that $\mathrm{Re}(\lambda_{2n})=\mathrm{Re}(\lambda_{2n-1})=\cdots=\mathrm{Re}(\lambda_1)$, e.g., when the system has a purely mass-proportional damping, we do not have a slow SSM. However, as we will see in this section, we select master subspace based on external and internal resonance, and the slowness of SSM is not an essential ingredient. The attractiveness of the SSM is automatically ensured because the remaining modes will decay quickly due to damping.}
\end{sloppypar}

\subsection{Main theorems}
\label{sec:theos}

When the excitation frequency $\Omega$ is not close to any of the natural frequencies, i.e., the external excitation is not in (near-) resonance with the system's eigenvalues, then it follows from~\eqref{eq:ssm-nonauto} that the non-autonomous part of the reduced dynamics vanishes. Indeed, the reduced dynamics is autonomous in this setting as the normal form style of parametrization of the non-autonomous SSM removes the non-resonant terms from its reduced dynamics. Hence, the trivial fixed point of the reduced dynamics is a stable focus. Substituting the steady-state $\boldsymbol{p}(t)=\mathbf{0}$ into~\eqref{eq:ssm-exp-eps} and utilizing~\eqref{eq:leading-nonauto} and~\eqref{eq:nonautossm-exp}, we obtain the periodic response of the full system at steady state as follows
\begin{equation}
\label{eq:zt-linear}
 \boldsymbol{z}(t)=-2\epsilon\mathrm{Re}\left((\boldsymbol{A}-\mathrm{i}\Omega\boldsymbol{B})^{-1}\boldsymbol{F}^{\mathrm{a}}e^{\mathrm{i}\Omega t}\right). 
\end{equation}
\textcolor{black}{Substituting~\eqref{eq:zABF} into the above equation, letting
\begin{equation}
   \boldsymbol{f}^\mathrm{ext}(\Omega t)={\boldsymbol{f}}^\mathrm{a}e^{\mathrm{i}\Omega t}+{\boldsymbol{f}}^\mathrm{a}e^{-\mathrm{i}\Omega t} 
\end{equation}
and utilizing~\eqref{eq:forcing-conj}, we can rewrite~\eqref{eq:zt-linear} in a more familiar representation as
\begin{equation}
    \boldsymbol{x}(t)=2\epsilon\mathrm{Re}\left((-\Omega^2\boldsymbol{M}+\mathrm{i}\Omega\boldsymbol{C}+\boldsymbol{K})^{-1}\boldsymbol{f}^{\mathrm{a}}e^{\mathrm{i}\Omega t}\right).
\end{equation}
Therefore,} the system behaves as a linear system at leading order.

We are mainly concerned with the response of the system~\eqref{eq:full-first} near an external resonance with the forcing frequency. We assume that the excitation frequency $\Omega$ is resonant with the master eigenvalues in the following way:
\begin{equation}
\label{eq:res-forcing}
    \boldsymbol{\lambda}^{\mathcal{E}}-\mathrm{i}\boldsymbol{r}\Omega\approx0,\,\, \bar{\boldsymbol{\lambda}}^{\mathcal{E}}+\mathrm{i}\boldsymbol{r}\Omega\approx0,\,\, \boldsymbol{r}\in\mathbb{Q}^m.
\end{equation}
\textcolor{black}{As an example of the resonance relation~\eqref{eq:res-forcing}, we consider an internally resonant system such that the master subspace $\mathcal{E}$ has two pairs of modes that exhibit near 1:3 inner resonances, i.e., $\lambda_2^\mathcal{E}\approx3\lambda_1^\mathcal{E}$ and $\bar{\lambda}_2^\mathcal{E}\approx3\bar{\lambda}_1^\mathcal{E}$. Then, if the external forcing frequency $\Omega$ is nearly resonant with the first pair of modes, i.e., $\lambda_1^\mathcal{E}\approx\mathrm{i}\Omega,~\lambda_2^\mathcal{E}\approx \mathrm{i}3\Omega$, we have $\boldsymbol{r}=(1,3)$. However, if the external forcing resonates with the second pair of modes, i.e., $\lambda_1^\mathcal{E} \approx \frac{1}{3}\mathrm{i}\Omega, ~\lambda_2^\mathcal{E}\approx \mathrm{i}\Omega$, then we have $\boldsymbol{r}=(1/3,1)$.}

\begin{theorem}[Reduced dynamics in polar coordinates]
\label{theo:polar}
Under the \emph{inner} resonance condition~\eqref{eq:res-inner}, the \emph{external} resonance condition~\eqref{eq:res-forcing}, and with polar coordinates $(\rho_i,\theta_i)$ defined as
\begin{equation}
\label{eq:polar-form}
    q_i=\rho_ie^{\mathrm{i}(\theta_i+r_i\Omega t)},\,\,\bar{q}_i=\rho_ie^{-\mathrm{i}(\theta_i+r_i\Omega t)},
\end{equation}
for $i=1,\cdots,m$, the following statements hold for $\epsilon>0$ small enough:
\begin{enumerate}[label=(\roman*)]
\item Under the coordinate transformation~\eqref{eq:polar-form}, the reduced dynamics~\eqref{eq:ssm-red-full} on the $2m$-dimensional SSM can be simplified to yield a slow-fast dynamical system. In the rotating frame, the {slow-phase} reduced dynamics in \emph{polar} coordinates $(\boldsymbol{\rho},\boldsymbol{\theta})\in\mathbb{R}^m\times\mathbb{T}^m$ is given by
\begin{equation}
\begin{pmatrix}\dot{\rho}_i\\\dot{\theta}_i\end{pmatrix}=\boldsymbol{r}^{\mathrm{p}}_i(\boldsymbol{\rho},\boldsymbol{\theta},\Omega,\epsilon)+\mathcal{O}(\epsilon|\boldsymbol{\rho}|)\boldsymbol{g}_i^\mathrm{p}(\phi),\label{eq:ode-reduced-slow-polar}
\end{equation}
for $i=1,\cdots,m$. Here the superscript $\mathrm{p}$ stands for `polar', $\boldsymbol{g}_i^\mathrm{p}$ is a periodic function and
\begin{align}
\label{eq:ode-reduced-polar-blocki}
\boldsymbol{r}^{\mathrm{p}}_i&=\begin{pmatrix}\rho_i\mathrm{Re}(\lambda_i^{\mathcal{E}})\\\mathrm{Im}(\lambda_i^{\mathcal{E}})-r_i\Omega\end{pmatrix}\nonumber\\
& +\sum_{(\boldsymbol{l},\boldsymbol{j})\in\mathcal{R}_i}\boldsymbol{\rho}^{\boldsymbol{l}+\boldsymbol{j}}\boldsymbol{Q}(\rho_i,\varphi_i(\boldsymbol{l},\boldsymbol{j}))\begin{pmatrix}\mathrm{Re}(\gamma(\boldsymbol{l},\boldsymbol{j}))\\\mathrm{Im}(\gamma(\boldsymbol{l},\boldsymbol{j}))\end{pmatrix}\nonumber\\
& +\epsilon\boldsymbol{Q}(\rho_i,-\theta_i)\begin{pmatrix}\mathrm{Re}(f_i)\\\mathrm{Im}(f_i)\end{pmatrix}
\end{align}
with $\mathcal{R}_i$ defined in~\eqref{eq:def-Ri} and with $\boldsymbol{\varphi}_i$ and $\boldsymbol{Q}$ defined as
\begin{gather}
\varphi_i(\boldsymbol{l},\boldsymbol{j})=\langle \boldsymbol{l}-\boldsymbol{j}-\mathbf{e}_i,\boldsymbol{\theta} \rangle,\label{eq:varphi-ang}\\
\boldsymbol{Q}(\rho,\theta)= \begin{pmatrix}\cos \theta & -\sin \theta\\\frac{1}{\rho}\sin \theta&\frac{1}{\rho}\cos \theta\end{pmatrix},\\
f_i=\left\lbrace \begin{array}{cl}
      (\boldsymbol{u}_i^{\mathcal{E}})^\ast\boldsymbol{F}^a   & \text{if}\hspace{2mm}  r_i=1 \\
      {0}   & \text{otherwise}
    \end{array} \right. \label{eq:fi}.
\end{gather}
Here $\mathbf{e}_i\in\mathbb{R}^m$ is the unit vector aligned along the $i$-th axis.
\item Any hyperbolic fixed point of the leading-order truncation of~\eqref{eq:ode-reduced-slow-polar}, viz,
\begin{equation}
\begin{pmatrix}\dot{\rho}_i\\\dot{\theta}_i\end{pmatrix}=\boldsymbol{r}^{\mathrm{p}}_i(\boldsymbol{\rho},\boldsymbol{\theta},\Omega,\epsilon),\quad i=1,\cdots, m,\label{eq:ode-reduced-slow-polar-leading}
\end{equation}
\textcolor{black}{persists as} a periodic solution $\boldsymbol{p}(t)$ of the reduced dynamics~\eqref{eq:ssm-red-full} on the SSM $\mathcal{W}(\mathcal{E},\Omega t)$. For a given excitation amplitude $\epsilon_0$, the leading-order approximation to the FRC is given by the zero level set of the components of the function $\boldsymbol{\mathcal{F}}^{\mathrm{p}}_{\epsilon_0}:\mathbb{R}^m\times\mathbb{T}^m\times\mathbb{R}\to\mathbb{R}^{2m}$
\begin{equation}
\boldsymbol{\mathcal{F}}^{\mathrm{p}}_{\epsilon_0}(\boldsymbol{\rho},\boldsymbol{\theta},\Omega):=\begin{pmatrix}\boldsymbol{r}^{\mathrm{p}}_1(\boldsymbol{\rho},\boldsymbol{\theta},\Omega,\epsilon_0)\\\vdots\\\boldsymbol{r}^{\mathrm{p}}_m(\boldsymbol{\rho},\boldsymbol{\theta},\Omega,\epsilon_0)\end{pmatrix}.
\end{equation}
\item The stability type of a hyperbolic fixed point of~\eqref{eq:ode-reduced-slow-polar-leading} coincides with the stability type of the corresponding periodic solution on the SSM $\mathcal{W}(\mathcal{E},\Omega t)$.
\end{enumerate}
\end{theorem}

\begin{proof}
We present the proof of this theorem in Appendix~\ref{sec:proof-theo-polar}.
\end{proof}

We restrict ourselves to the leading-order approximation (see~\eqref{eq:leading-nonauto} and~\eqref{eq:red-nonauto-lead}) for the following three reasons: (i) the proof of the theorem implies the persistence of hyperbolic periodic orbits under the addition of terms at order $\mathcal{O}(\epsilon|\boldsymbol{p}|)$ or higher; (ii) numerical experiments show that the results with this approximation already have satisfied accuracy; (iii) we obtain a parametric reduced-order model~\eqref{eq:ode-reduced-slow-polar-leading} with the forcing frequency $\Omega$ and the amplitude $\epsilon$ as system parameters, enabling efficient parameter continuation (see section~\ref{sec:cont-fixed-points}). When the higher-order terms at $\mathcal{O}(\epsilon|\boldsymbol{p}|^k)$ with $k\geq1$ for the non-autonomous part are taken into consideration, the slow-phase reduced dynamics is still of the form~\eqref{eq:ode-reduced-slow-polar-leading}. However, the coefficients of these higher-order terms are implicit functions of $\Omega$ and one has to solve systems of linear equations to obtain the coefficients for each $\Omega$~\cite{ponsioen2019analytic}. \textcolor{black}{In ref.~\cite{jiang2005nonlinear}, a Galerkin-based method was used to solve the invariance equations and the resulting reduced dynamics is not parametric in $\Omega$. Thus, one needs to construct reduced-order models for a number of discrete excitation frequencies to approximate a forced response curve~\cite{jiang2005nonlinear}.}

We note that the reduced dynamics~\eqref{eq:ode-reduced-slow-polar} becomes singular at $\rho_i=0$ for any $i\in\{1,\cdots,m\}$ due to the blow-up of $\boldsymbol{Q}(\rho_i,-\theta_i)$ at $\rho_i=0$. 
Such a singularity always arises in the study of a 1:1 resonance between the higher-frequency master mode and external forcing frequency~\cite{nayfeh1989modal}.
For instance, if $\Omega\approx\omega_2$ with $\omega_2\approx3\omega_1$, we have a solution branch with vanishing $\rho_1$~\cite{nayfeh1974nonlinear}. One is tempted to simply ignore the corresponding component in the vector field~\eqref{eq:ode-reduced-slow-polar}, but this prevents us from determining the correct stability of the fixed point based on the simplified system~\cite{nayfeh1988undesirable}. For this reason, we also give the Cartesian coordinate representation of the reduced dynamics on the SSM in the following theorem.

\begin{theorem}[Reduced dynamics on Cartesian coordinates]
\label{theo:cartesian}
Under the \emph{inner} resonance condition~\eqref{eq:res-inner}, the \emph{external} resonance condition~\eqref{eq:res-forcing}, and with Cartesian coordinates $(q_{i,\mathrm{s}}^\mathrm{R},q_{i,\mathrm{s}}^\mathrm{I})$ defined as
\begin{gather}
    q_i=q_{i,\mathrm{s}}e^{\mathrm{i}r_i\Omega t}=(q_{i,\mathrm{s}}^\mathrm{R}+\mathrm{i}q_{i,\mathrm{s}}^\mathrm{I})e^{\mathrm{i}r_i\Omega t},\nonumber\\
    \label{eq:cartesian-form}
    \bar{q}_i=\textcolor{black}{\bar{q}_{i,\mathrm{s}}e^{-\mathrm{i}r_i\Omega t}}=(q_{i,\mathrm{s}}^\mathrm{R}-\mathrm{i}q_{i,\mathrm{s}}^\mathrm{I})e^{-\mathrm{i}r_i\Omega t},
\end{gather}
for $i=1,\cdots,m$, where $q_{i,\mathrm{s}}^\mathrm{R}=\mathrm{Re}(q_{i,\mathrm{s}})$ and $q_{i,\mathrm{s}}^\mathrm{I}=\mathrm{Im}(q_{i,\mathrm{s}})$, the following statements hold for $\epsilon>0$ small enough:
\begin{enumerate}[label=(\roman*)]
\item Under the coordinate transformation~\eqref{eq:cartesian-form}, the reduced dynamics~\eqref{eq:ssm-red-full} on the $2m$-dimensional SSM, can be simplified to yield a slow-fast dynamical system with the coordinate transformation~\eqref{eq:cartesian-form}. In the rotating frame, the {slow-phase} reduced dynamics in \emph{Cartesian} coordinates $(\boldsymbol{q}_{\mathrm{s}}^\mathrm{R},\boldsymbol{q}_{\mathrm{s}}^\mathrm{I})\in\mathbb{R}^m\times\mathbb{R}^m$ is given by
\begin{equation}
\label{eq:ode-reduced-slow-cartesian}
\begin{pmatrix}\dot{q}_{i,\mathrm{s}}^\mathrm{R}\\\dot{q}_{i,\mathrm{s}}^\mathrm{I}\end{pmatrix}=\boldsymbol{r}^{\mathrm{c}}_i(\boldsymbol{q}_{\mathrm{s}},\Omega,\epsilon)+\mathcal{O}(\epsilon|\boldsymbol{q}_\mathrm{s}|)\boldsymbol{g}_i^\mathrm{c}(\phi),
\end{equation}
for $i=1,\cdots,m$. Here the superscript $\mathrm{c}$ stands for `Cartesian', $\boldsymbol{g}_i^\mathrm{c}$ is a periodic function, and
\begin{align}
\boldsymbol{r}^{\mathrm{c}}_i & =\begin{pmatrix}\mathrm{Re}(\lambda_i^{\mathcal{E}}) & r_i\Omega-\mathrm{Im}(\lambda_i^{\mathcal{E}})\\
\mathrm{Im}(\lambda_i^{\mathcal{E}})-r_i\Omega & \mathrm{Re}(\lambda_i^{\mathcal{E}})\end{pmatrix}\begin{pmatrix}q_{i,\mathrm{s}}^{\mathrm{R}}\\q_{i,\mathrm{s}}^{\mathrm{I}}\end{pmatrix}
\nonumber\\
& +\sum_{(\boldsymbol{l},\boldsymbol{j})\in\mathcal{R}_i}\begin{pmatrix}\mathrm{Re}\left(\gamma(\boldsymbol{l},\boldsymbol{j})\boldsymbol{q}_s^{\boldsymbol{l}}\bar{\boldsymbol{q}}_s^{\boldsymbol{j}}\right)\\\mathrm{Im}\left(\gamma(\boldsymbol{l},\boldsymbol{j})\boldsymbol{q}_s^{\boldsymbol{l}}\bar{\boldsymbol{q}}_s^{\boldsymbol{j}}\right)\end{pmatrix}+\epsilon\begin{pmatrix}\mathrm{Re}(f_i)\\\mathrm{Im}(f_i)\end{pmatrix}.
\end{align}
\item Any hyperbolic fixed point of the leading-order truncation of~\eqref{eq:ode-reduced-slow-cartesian}, viz,
\begin{equation}
\label{eq:ode-reduced-slow-cartesian-leading}
\begin{pmatrix}\dot{q}_{i,\mathrm{s}}^\mathrm{R}\\\dot{q}_{i,\mathrm{s}}^\mathrm{I}\end{pmatrix}=\boldsymbol{r}^{\mathrm{c}}_i(\boldsymbol{q}_{\mathrm{s}},\Omega,\epsilon),\quad i=1,\cdots,m,
\end{equation}
corresponds to a periodic solution $\boldsymbol{p}(t)$ of the reduced dynamics~\eqref{eq:ssm-red-full} on SSM $\mathcal{W}(\mathcal{E},\Omega t)$. For a given excitation amplitude $\epsilon_0$, the leading-order approximation to the FRC is given by the zero level set of the components of the function $\boldsymbol{\mathcal{F}}^c_{\epsilon_0}:\mathbb{C}^m\times\mathbb{R}\to\mathbb{R}^{2m}$
\begin{equation}
\boldsymbol{\mathcal{F}}^{\mathrm{c}}_{\epsilon_0}(\boldsymbol{q}_{\mathrm{s}},\Omega):=\begin{pmatrix}\boldsymbol{r}^{\mathrm{c}}_1(\boldsymbol{q}_{\mathrm{s}},\Omega,\epsilon_0)\\\vdots\\\boldsymbol{r}^{\mathrm{c}}_m(\boldsymbol{q}_{\mathrm{s}},\Omega,\epsilon_0)\end{pmatrix}.
\end{equation}
\item The stability type of a hyperbolic fixed point of~\eqref{eq:ode-reduced-slow-cartesian-leading} coincides with the stability type of the corresponding periodic solution on the SSM $\mathcal{W}(\mathcal{E},\Omega t)$.
\end{enumerate}
\end{theorem}

\begin{proof}
We present the proof of this theorem in Appendix~\ref{sec:proof-theo-cartesian}
\end{proof}

\begin{remark}
The \emph{Cartesian} coordinates and \emph{polar} coordinates featured in Theorems~\ref{theo:polar} and~\ref{theo:cartesian} are related by
\begin{gather}
\rho_i = ||q_{i,\mathrm{s}}||=\sqrt{\left(q_{i,\mathrm{s}}^{\mathrm{R}}\right)^2+\left(q_{i,\mathrm{s}}^{\mathrm{I}}\right)^2},\nonumber\\
\theta_i=\arg(q_{i,\mathrm{s}})=\textcolor{black}{\texttt{atan2}}(q_{i,\mathrm{s}}^{\mathrm{I}},q_{i,\mathrm{s}}^{\mathrm{R}})
\end{gather}
for $i=1,\cdots,m$. In this paper, we will plot the results of $\rho_i$ instead of $\left(q_{i,\mathrm{s}}^{\mathrm{R}}, q_{i,\mathrm{s}}^{\mathrm{I}}\right)$ for easier interpretation of the vibration amplitudes.
\end{remark}

\subsection{Continuation of fixed points}
\label{sec:cont-fixed-points}
The above theorems indicate that we can find periodic orbits by locating the fixed points of the reduced dynamics for $(\boldsymbol{\rho},\boldsymbol{\theta})$ in polar coordinate representation or $(\boldsymbol{q}_{\mathrm{s}}^{\mathrm{R}},\boldsymbol{q}_{\mathrm{s}}^{\mathrm{I}})$ in Cartesian coordinate representation. The solution manifold of the fixed points is two-dimensional and may be parameterized by the system parameters $(\Omega,\epsilon)$. For a given $\epsilon=\epsilon_0$, a one-dimensional solution manifold is obtained, corresponding to the FRC stated in the theorems.

For a two-dimensional SSM ($m=1$), we have $\boldsymbol{l}=\boldsymbol{j}+\boldsymbol{e}_1$~\cite{ponsioen2019analytic,ponsioen2020model} and hence in equation~\eqref{eq:varphi-ang} $\varphi_1(\boldsymbol{l},\boldsymbol{j})=0$ for all $(\boldsymbol{l},\boldsymbol{j})\in\mathcal{R}_1$. It follows that one can obtain FRC from the joint zero level set of $\boldsymbol{\mathcal{F}}^p_{\epsilon_o}(\rho_1,\theta_1,\Omega)$, which is the intersection of two two-dimensional surfaces in a three-dimensional space parameterized by $(\rho_1,\theta_1,\Omega)$. Following this approach, \emph{all} equilibrium points in a given computational domain for $(\rho_1,\theta_1,\Omega)$ can be found. Therefore, this level-set based method is able to find \emph{isolas}, namely, isolated solution branches of FRC. The reader may refer to~\cite{ponsioen2019analytic,ponsioen2020model,SHOBHIT} for more details about this level-set-based technique.

Under the internal resonance assumption~\eqref{eq:res-inner} with $m\geq2$, the level-set-based detection of fixed points becomes impracticable due to the increment of dimensions. Instead, we seek the fixed points by solving the set of nonlinear algebraic equations defining them numerically. \emph{Parameter continuation} provides a powerful tool to cover the solution manifold of fixed points. Several packages are available to perform such continuation, including \textsc{auto}~\cite{doedel2007auto}, \textsc{matcont}~\cite{dhooge2003matcont} and \textsc{coco}~\cite{dankowicz2013recipes}. The last one is distinguished from the first two because it uses a \textcolor{black}{staged construction paradigm where larger problems are assembled from smaller ones. More details about the staged construction and its applications can be found in~\cite{dankowicz2013recipes}}. 

In this paper, we use the \texttt{ep} toolbox in~\textsc{coco}~\cite{dankowicz2013recipes} to perform the continuation of fixed points of~\eqref{eq:ode-reduced-slow-polar-leading} or~\eqref{eq:ode-reduced-slow-cartesian-leading}. \textcolor{black}{The `\texttt{ep}' stands for \emph{equilibrium point}. Note that the implementation of our method does not necessarily rely on \textsc{coco}. One can use other toolboxes such as \textsc{auto}~\cite{doedel2007auto} and \textsc{matcont}~\cite{dhooge2003matcont} for the continuation of fixed points, or even manually solve for the fixed points of~\eqref{eq:ode-reduced-slow-polar-leading} or~\eqref{eq:ode-reduced-slow-cartesian-leading}.}

Along with the computation of fixed points, \texttt{ep} also calculates the eigenvalues of the Jacobian of the reduced vector field and hence provides information about the stability and bifurcation of the fixed points. Leveraging this capability, we have built a toolbox \texttt{SSM-ep}\footnote{\texttt{SSM-ep} toolbox is included in SSMTool 2.1~\cite{ssmtool21}}, based on the \texttt{ep} toolbox in \textsc{coco}. The \texttt{SSM-ep} toolbox performs one-dimensional continuation of fixed points with respect to changes in $\Omega$ \emph{or} $\epsilon$. For each fixed point obtained in this fashion, the corresponding periodic solution in the SSM $\mathcal{W}(\mathcal{E},\Omega t)$ in normal coordinates $\boldsymbol{p}(t)$ is mapped back to physical coordinates $\boldsymbol{z}(t)$. We provide more details on this inverse mapping in next subsection.

As an starting point of continuation, an initial fixed point is needed. \texttt{SSM-ep} provides two options for finding such an initial fixed point:
\begin{sloppypar}
\begin{itemize}
\item \texttt{fsolve}: The \textsc{matlab} nonlinear equation solver \texttt{fsolve} is called to locate the zeros of the vector field. This solver finds zeros by optimization techniques.
\item \texttt{forward}: A long-time forward simulation is performed and a fixed point is sought based on the fact that the initial condition is now in the basin of attraction of the assumed fixed point.
\end{itemize}
\end{sloppypar}
The above two options ask for an initial guess for the initial point in the optimization or the initial condition in the forward simulation. By default, we set $\boldsymbol{\rho}=\boldsymbol{\theta}=0.1$ in the case of the polar representation (Theorem~\ref{theo:polar}) and $\boldsymbol{q}_{\mathrm{s}}=\boldsymbol{0}$ in the case of the Cartesian representation (Theorem~\ref{theo:cartesian}) as the initial guess. Numerical experiments suggest that these choices are robust in general.

\subsection{FRC in physical coordinates}
\begin{sloppypar}
With the fixed points of the reduced dynamics on the SSM computed, the corresponding periodic orbits on the SSM can be computed from the transformation~\eqref{eq:polar-form} or~\eqref{eq:cartesian-form}. We then need to map the periodic orbits in normal coordinates back to physical coordinates. If $\boldsymbol{p}(t)$ is a trajectory in normal coordinates, we obtain the corresponding trajectory, $\boldsymbol{z}(t)$, in physical coordinates, namely, $\boldsymbol{z}(t)$, by substituting $\boldsymbol{p}(t)$ into~\eqref{eq:ssm-exp-eps}. With the leading order approximation of non-autonomous SSM, we have
\begin{equation}
\boldsymbol{z}(t)=\boldsymbol{W}(\boldsymbol{p}(t))+\epsilon\left(\boldsymbol{x}_{\boldsymbol{0}}e^{\mathrm{i}\Omega t}+\bar{\boldsymbol{x}}_{\boldsymbol{0}}e^{-\mathrm{i}\Omega t}\right)
\end{equation}
where $\boldsymbol{x}_{\boldsymbol{0}}$ is the $\Omega$-dependent solution of the system of linear equations (cf.~\eqref{eq:expphi-}). The stability type of the periodic orbit, $\boldsymbol{z}(t)$, is the same as that of the $\boldsymbol{p}(t)$, given that the SSM is invariant and attracting.
\end{sloppypar}

When a FRC is obtained from a numerical method, it is represented as a set of periodic solutions, $\{\boldsymbol{p}(t,\Omega_i)\}$, for a set of sampled excitation frequencies, $\{\Omega_i\}$. For each sampled $\Omega_i$, the corresponding $\boldsymbol{x}_{\boldsymbol{0}}$ is obtained by solving the system of linear equations~\eqref{eq:expphi-}. \textcolor{black}{All numerical results reported in this paper have been obtained with a nonuniform sampling for $\Omega$, which is automatically determined by \texttt{atlas} algorithms in \textsc{coco}~\cite{dankowicz2013recipes,dankowicz2020multidimensional}. Specifically, we perform \texttt{ep} continuation in a given frequency span, allowing an adaptive change of the continuation step size by the \texttt{atlas} algorithms. This enables continuation along complex paths and results in a non-uniform sampling for $\Omega$. The \texttt{SSM-ep} toolbox supports uniform sampling and the  \textsc{coco}-based nonuniform sampling for $\Omega$. Note that the sampling strategy for $\Omega$ does not necessarily rely on \textsc{coco}. One can simply use uniform sampling or adopt other suitable nonuniform sampling methods that capture complicated geometry of the FRC.}


\subsection{Computational cost}
\label{sec:compLoad}
The main computational cost of FRC from SSM analysis is composed of three factors:
\begin{itemize}
\item A one-time computation of the autonomous SSM,
\item Parameter continuation of the fixed points of the reduced dynamics,
\item $N_\Omega$ times computation of the non-autonomous SSM, where $N_\Omega$ is the number of sampled frequencies in $\{\Omega_i\}$.
\end{itemize}
The second factor is the smallest among the three because 1) the reduced dynamical system on the SSM is $2m$-dimensional and $m$ is equal to two or three in most practical applications; 2) we perform a continuation of \emph{fixed points} instead of \emph{periodic orbits}. In contrast, the computational cost of the first factor increases significantly with the increment of the expansion order of the SSM, as discussed in~\cite{SHOBHIT}. For the third factor, we need to solve a system of linear equations with size $2n$ for each sampled excitation frequency $\Omega$. This process is computationally intensive if the number of samples is large and the system is high dimensional. \emph{Parallel computing} can be utilized to speed up this part of the computation. As an alternative, we may simply ignore the contribution of $\boldsymbol{x}_{\boldsymbol{0}}$, given that $\epsilon$ is a small parameter. Such a simplification has been adopted in the method of normal forms~\cite{touze2006nonlinear,vizzaccaro2020direct}.~\textcolor{black}{Unless otherwise stated, the reported computational time of FRC using SSM in this paper includes all the three factors.}

\section{Examples}
\label{sec:examp}
\textcolor{black}{In this section,} we illustrate our computational algorithm for resonant SSMs in examples of increasing complexity. The numerical package used in these computations is available from~\cite{ssmtool21}.

\subsection{A chain of oscillators}
Consider the chain of nonlinear oscillators shown in Fig.~\ref{fig:three_oscillators_fixed_two_ends} with their equations of motion given by
\begin{gather}
\ddot{x}_1+x_1+ c_1\dot{x}_1+ K(\textcolor{black}{x_1}-\textcolor{black}{x_2})^3=\epsilon f_1\cos\Omega t,\nonumber\\
\ddot{x}_2+x_2+ c_2\dot{x}_2+ K[(\textcolor{black}{x_2}-\textcolor{black}{x_1})^3+(\textcolor{black}{x_2}-\textcolor{black}{x_3})^3]=0,\nonumber\\
\ddot{x}_3+x_3+ c_3\dot{x}_3+ K(\textcolor{black}{x_3}-\textcolor{black}{x_2})^3=0\label{eq:chain-eom}.
\end{gather}
The unforced linearized system around the origin has eigenvalues 
\begin{gather}
\lambda_{1,2}=-\frac{ c_1}{2}\pm\mathrm{i}\sqrt{1-0.25c_1^2}\approx\pm\mathrm{i},\nonumber\\
\lambda_{3,4}=-\frac{ c_2}{2}\pm\mathrm{i}\sqrt{1-0.25c_2^2}\approx\pm\mathrm{i},\nonumber\\
\lambda_{5,6}=-\frac{ c_3}{2}\pm\mathrm{i}\sqrt{1-0.25c_3^2}\approx\pm\mathrm{i},
\end{gather}
\textcolor{black}{provided that $0<c_{1,2,3}\ll1$}. Hence the system has a 1:1:1 internal resonance, yielding $\boldsymbol{r}=(1,1,1)$ in~\eqref{eq:res-forcing} for $\Omega=1$.

\begin{figure}[!ht]
\centering
\includegraphics[width=0.4\textwidth]{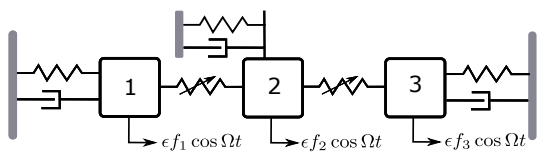}
\caption{A chain of three oscillators with identical natural frequencies.}
\label{fig:three_oscillators_fixed_two_ends}
\end{figure}

With $c_1=5\times10^{-4}$ N.s/m, $c_2=1\times10^{-3}$ N.s/m, $c_3=1.5\times10^{-3}$ N.s/m, $K=1\times10^{-3}$ N/$\mathrm{m}^3$, $f_1=1$ N and $\epsilon=0.005$, we obtain the FRC in the normal coordinates $(\rho_1,\rho_2,\rho_3)$ and in the physical coordinates $(||x_1||_{\infty},||x_2||_{\infty},||x_3||_{\infty})$ in Fig.~\ref{fig:oneToneTone}. Here and in the upcoming examples, $||\bullet||_{\infty}:=\max_{t\in[0,T]}||\bullet(t)||$ denotes the amplitude of the periodic response.

The FRC in Fig.~\ref{fig:oneToneTone} displays rich dynamic behavior due to modal interactions, including stable and unstable periodic orbits, as well as saddle-node and Hopf bifurcations. Recall that only the first degrees-of-freedom (DOF) is excited. However, we observe nontrivial dynamics in the second and third DOF, resulting from modal interactions due to the 1:1:1 internal resonance.


We now use the \texttt{po}-toolbox of \textsc{coco} to illustrate the accuracy and efficiency of the SSM-based FRC analysis. In \texttt{po}, a periodic orbit is sought as the solution to a boundary-value problem with periodic boundary condition and an appropriate phase condition if the system is autonomous. Then the collocation method is used to discretize the boundary-value problem and parameter continuation is performed to obtain a solution manifold of periodic orbits representing the FRC. In the continuation with \texttt{po}, a variational problem is solved for each periodic orbit and then the stability of the periodic orbit is obtained. As seen in Fig.~\ref{fig:oneToneTone}, the results from SSM match closely with the reference solutions from \texttt{po} (labelled as \emph{Collocation}). The computation here was performed on an Intel(R) Core(TM) i7-6700HQ processor (2.60 GHz) of a laptop. The computational times for the SSM analysis and the \texttt{po} toolbox were 27.5 seconds and 56.4 seconds, respectively. This speed-up gain will become substantially more significant in higher dimensional problems, as will see in later examples. Indeed, the dimension of the continuation problem of fixed  point is $2m$. For most practical applications with internal resonance, we have $m\in\{2,3\}$, independently of $n$. In contrast, the dimension of the continuation problem of periodic orbits is $2nk$, which increases linearly with respect to $n$. We typically have $k\sim\mathcal{O}(100)$ in the collocation discretization. Such a significant difference between the dimensions of the two continuation problems results in a major speed-up gain.

\begin{figure}[!ht]
\centering
\includegraphics[width=0.45\textwidth]{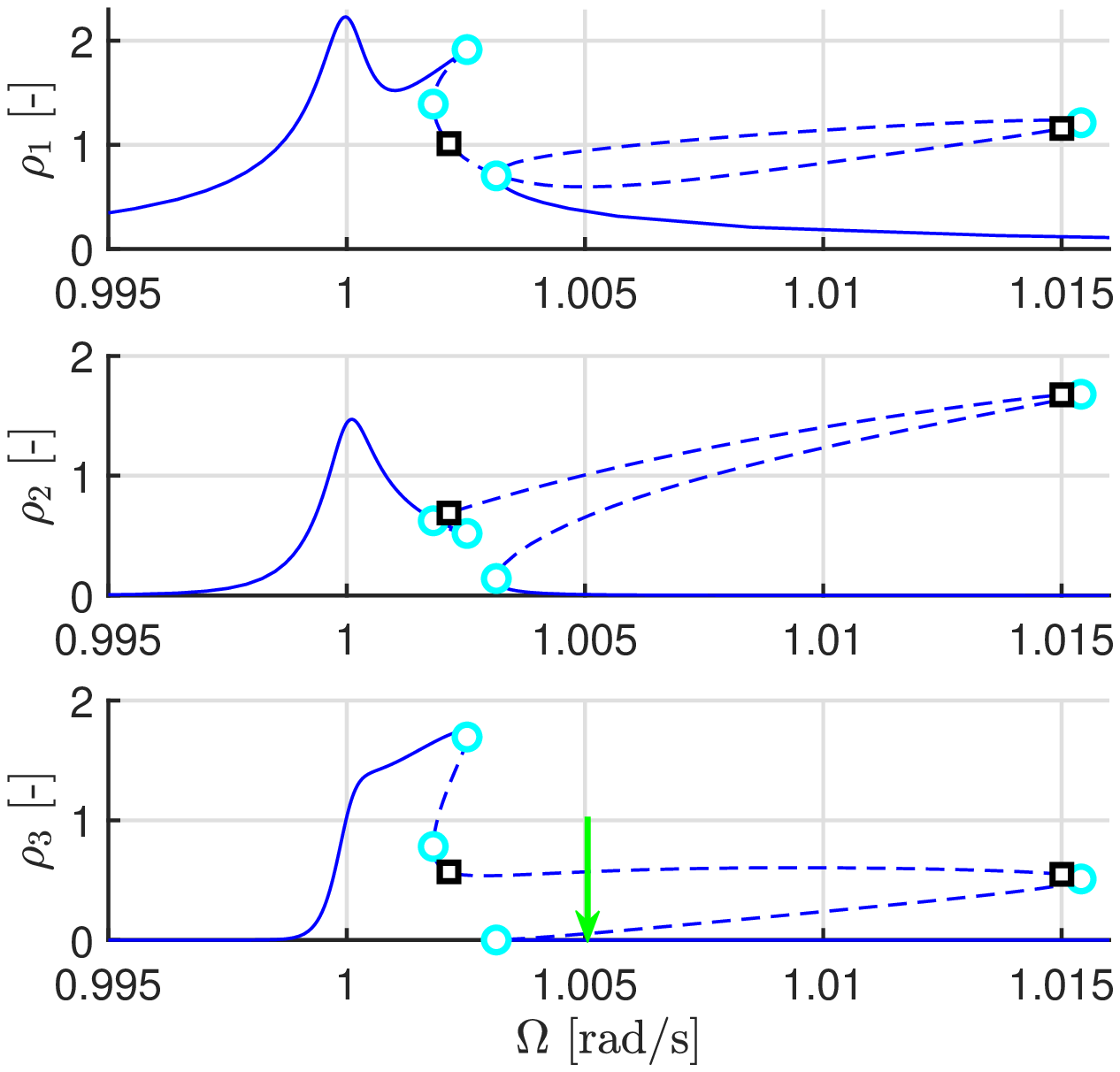}
\includegraphics[width=0.45\textwidth]{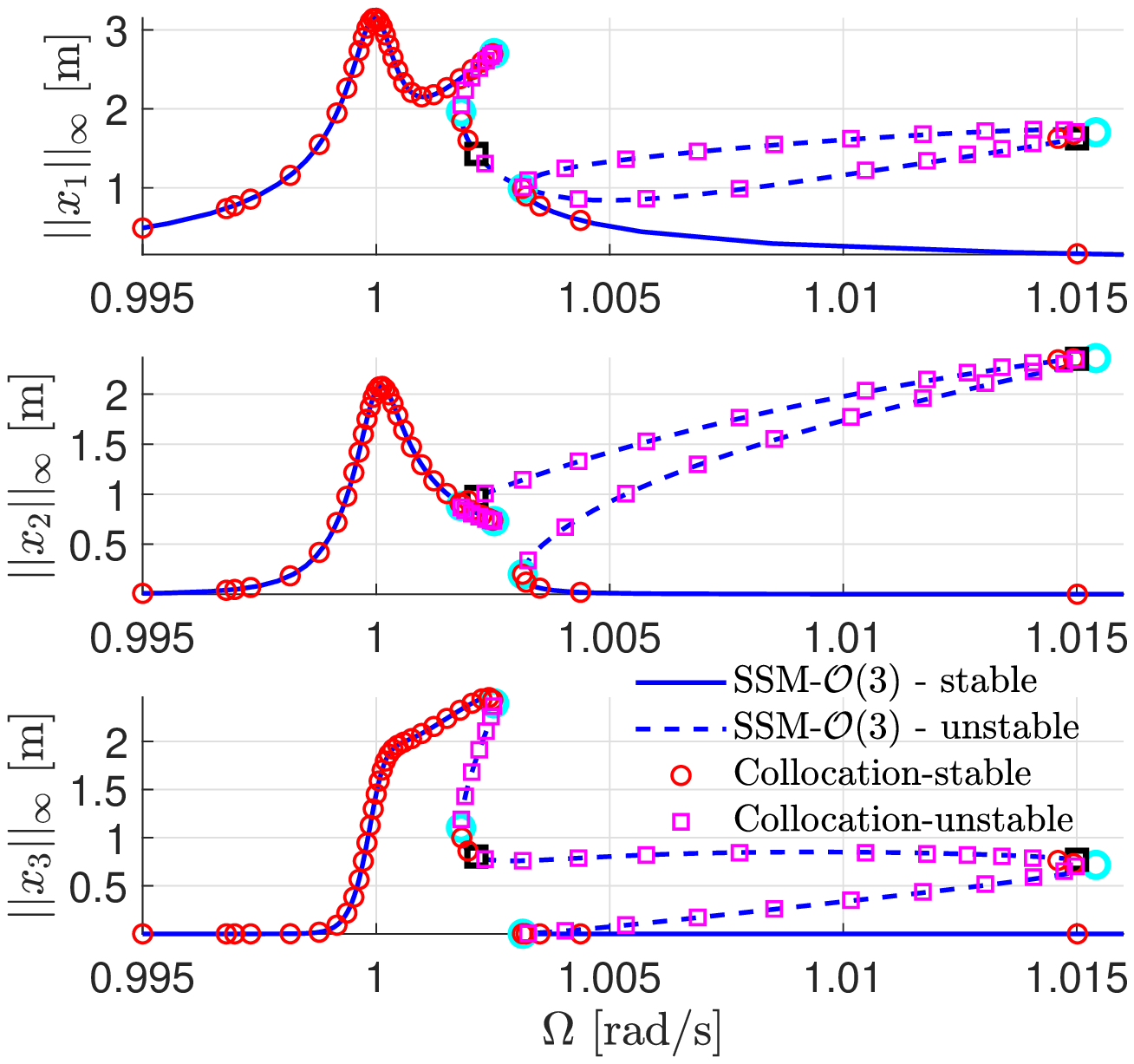}
\caption{The FRCs of the nonlinear oscillator chain~\eqref{eq:chain-eom} in normal coordinates ($\rho_1,\rho_2,\rho_3$) and physical coordiantes ($x_1,x_2,x_3$). Here, and throughout the paper, the solid lines indicate stable solution branches, while dashed lines mark unstable solution branches. The cyan circles denote saddle-node bifurcation points and black squares denote Hopf bifurcation points. The label SSM-$\mathcal{O}(k)$ suggests that the expansion order of SSM is $k$. In the panels for FRC in physical coordinates, the results obtained by continuation of periodic orbits with the collocation method are presented as well to illustrate the accuracy of the SSM-based results. The SSM results plotted here are obtained with {Cartesian} coordinate representation. The continuation path in {polar} coordinates terminates at $\Omega\approx1.0054$ with $\rho_3\approx2.08\times10^{-8}$ (see the green arrow in the third panel), which triggers near singularity and then the failure of the Newton iteration.}
\label{fig:oneToneTone}
\end{figure}

In this example, SSM computations were conducted in both polar and Cartesian coordinates. The two representations generate consistent results whenever results can be obtained. As we discussed in section~\ref{sec:theos}, however, the \emph{polar} coordinate representation can have the singularity issue. Indeed, the continuation of fixed points in the vector field with \emph{polar} representation terminates at $\Omega\approx1.0054$ rad/s where $\rho_3\approx2.08\times10^{-8}$, as indicated by the green arrow in Fig.~\ref{fig:oneToneTone}. Such a termination results from the failure of Newton iteration at a nearly singular point where $\rho_3\approx0$. By contrast, the continuation of fixed points in the vector field with \emph{Cartesian} representation is successfully performed in the given range of $\Omega$ with no singularity encountered.

No reduction is involved in this example, namely, $m=n$. It follows that the SSM analysis here is equivalent to the application of the method of normal form~\cite{neild2011applying}. Unlike the approach in~\cite{neild2011applying}, however, no assumptions are made here on the smallness of the nonlinearity in the SSM analysis. In the remaining examples, we will have $m\ll n$ to demonstrate the effectiveness and efficiency of SSM-based model reduction.

\subsection{A prismatic beam with axial stretching}
\label{sec:prismatic_beam}
Next we consider a forced hinged-clamped beam of the type treated in~\cite{nayfeh1974nonlinear}. Let $E$ be the elastic modulus, $r$, $A$ and $I$ be the radius of gyration, area and moment of inertia of the cross section, $L$ be the characteristic length and $\rho$ be the density of the beam. The governing equation for the transverse deflection $w(x,t)$ of the beam in dimensionless form is given by~\cite{nayfeh1974nonlinear}
\begin{gather}
\frac{\partial^4{w}}{\partial {x}^4}+\frac{\partial^2{w}}{\partial{t}^2}=\epsilon\left(H\frac{\partial^2{w}}{\partial{x}^2}+p-2c\frac{\partial{w}}{\partial{t}}\right),\nonumber\\
w(0)=w{''}(0)=w(l)=w'(l)=0\label{eq:beam-pde}.
\end{gather}
Here $H$ represents \textcolor{black}{the nonlinear} axial stretching force due to large deformation
\begin{equation}
    H = \frac{1}{\textcolor{black}{2l}}\int_0^l\left(\frac{\partial{w}}{\partial{x}}\right)^2\mathrm{d}{x},
\end{equation}
${x}$, ${t}$ are dimensionless length and time; ${p}$ and ${c}$ are dimensionless distributed loading and damping coefficients; and $\epsilon$ characterizes the slenderness ratio of the beam. These dimensionless quantities are defined as follows in~\cite{nayfeh1974nonlinear}:
\begin{gather}
x=\frac{\hat{x}}{L},\quad
t=\sqrt{\frac{Er^2}{\rho L^4}}\hat{t},\quad
w=\frac{\hat{w}L}{r^2},\nonumber\\
p=\frac{\hat{p}L^7}{r^6 EA},\quad
c=\frac{\hat{c}L^4}{2r^3A\sqrt{\rho E}},\quad
\epsilon=\frac{r^2}{L^2},
\end{gather}
where $\hat{x}$, $\hat{t}$, $\hat{w}$ are the length, time and transverse deflection with units; $\hat{p}$ and $\hat{c}$ are distributed loading and damping coefficient. Here we have $\hat{x}\in[0,lL]$ and then $x\in[0,l]$. 

\begin{sloppypar}
With a modal expansion
\begin{equation}
w(x,t)=\sum_{i=1}^n\psi_i(x)u_i(t)
\end{equation}
followed by a Galerkin projection, the governing partial-differential equation~\eqref{eq:beam-pde} is transferred into a set of ordinary differential equations
\begin{align}
\label{eq:beam-odes}
& \ddot{u}_i+\omega_i^2u_i\nonumber\\
& =\epsilon\left(-2c\dot{u}_i+f_i\cos\Omega t+\frac{1}{2l}\sum_{j,k,s}\alpha_{ijks}u_ju_ku_s\right),
\end{align}
for $i=1,\cdots,n$, where
\begin{gather}
f_i=\int_0^l \psi_i(x)p(x)\mathrm{d}x,\\ \alpha_{ijks}=\left(\int_0^l\psi_i(x)\psi_s^{''}(x)\mathrm{d}x\right)\left(\int_0^l\psi_j^{\prime}(x)\psi_k^{\prime}(x)\mathrm{d}x\right).
\end{gather}
Here the eigenfunction $\psi_i(x)$ and the corresponding natural frequency $\omega_i$ are the solutions of the eigenvalue problem
\begin{gather}
\frac{d^4\psi_i}{dx^4}-\omega_i^2\psi_i=0,\nonumber\\ \psi_i(0)=\psi_i^{''}(0)=\psi_i(l)=\psi_i^{'}(l)=0,
\end{gather}
whose solutions have been documented in~\cite{nayfeh1974nonlinear}
\end{sloppypar}

For $l=2$, the first two modes have a near 1:3 internal resonance, \textcolor{black}{i.e., $\omega_2\approx3\omega_1$,}
where $\omega_1=3.8533$ and $\omega_2=12.4927$ .
The forced response of this system under external harmonic response has been investigated in~\cite{nayfeh1974nonlinear} with the method of multiple scale (MMS) at $\Omega\approx\omega_1$ and $\Omega\approx\omega_2$. Here we use reduction to the 1:3 resonant SSM to study such a system and compare the results obtained by the two methods. With $n=10$, we take the first two pairs of modes as the master spectral space, namely, $\mathcal{E}=\Span\{v_1,\bar{v}_1,v_2,\bar{v}_2\}$. Consequently, the dimension of the phase space for the full system is 20 while the reduced system on the resonant SSM will be four-dimensional. The physical coordinates $\boldsymbol{x}$ in~\eqref{eq:eom-second-full} are actually modal coordinates $\boldsymbol{u}$ in this example.~\textcolor{black}{Note that in order to apply SSM reduction on this problem, we do not require the nonlinear and damping terms to be scaled by $\epsilon$, in contrast to MMS.}

\subsubsection{Primary resonance of the first mode}
Let $\epsilon=1\times10^{-4}$, $c=100$, $\epsilon f_1=5$ and $f_2=\cdots=f_{10}=0$, we are interested in the forced response for $\Omega\approx\omega_1$. The first mode is excited and hence $\rho_1\neq0$. Due to the internal resonance, $\rho_2\neq0$ as well. This allows the use of {polar} coordinates with $\boldsymbol{r}=(1,3)$. The obtained FRCs in the coordinates $(\rho_1,\rho_2)$ and $(u_1,u_2)$ for $\Omega\in[3.7782,4.0867]$ are presented in Fig.~\ref{fig:FirstMode}. Although nonzero, $\rho_2$ is still small compared to $\rho_1$, and  hence the response of the system is mainly contributed by the first mode, as seen in the first two panels of Fig.~\ref{fig:FirstMode}. An excellent match between the results of SSM analysis and MMS is observed for $||u_1||_{\infty}$ while discrepancies occur in the FRC of $||u_2||_{\infty}$. We also use the \texttt{po} toolbox in \textsc{coco} to extract the FRC of the full system as the reference solution to compare the accuracy of solutions obtained by the two methods. The results from \texttt{po} are labelled as \emph{Collocation}. As can be seen in the last panel of Fig.~\ref{fig:FirstMode}, SSM reduction yields more accurate results than MMS.

\begin{figure}[!ht]
\centering
\includegraphics[width=0.45\textwidth]{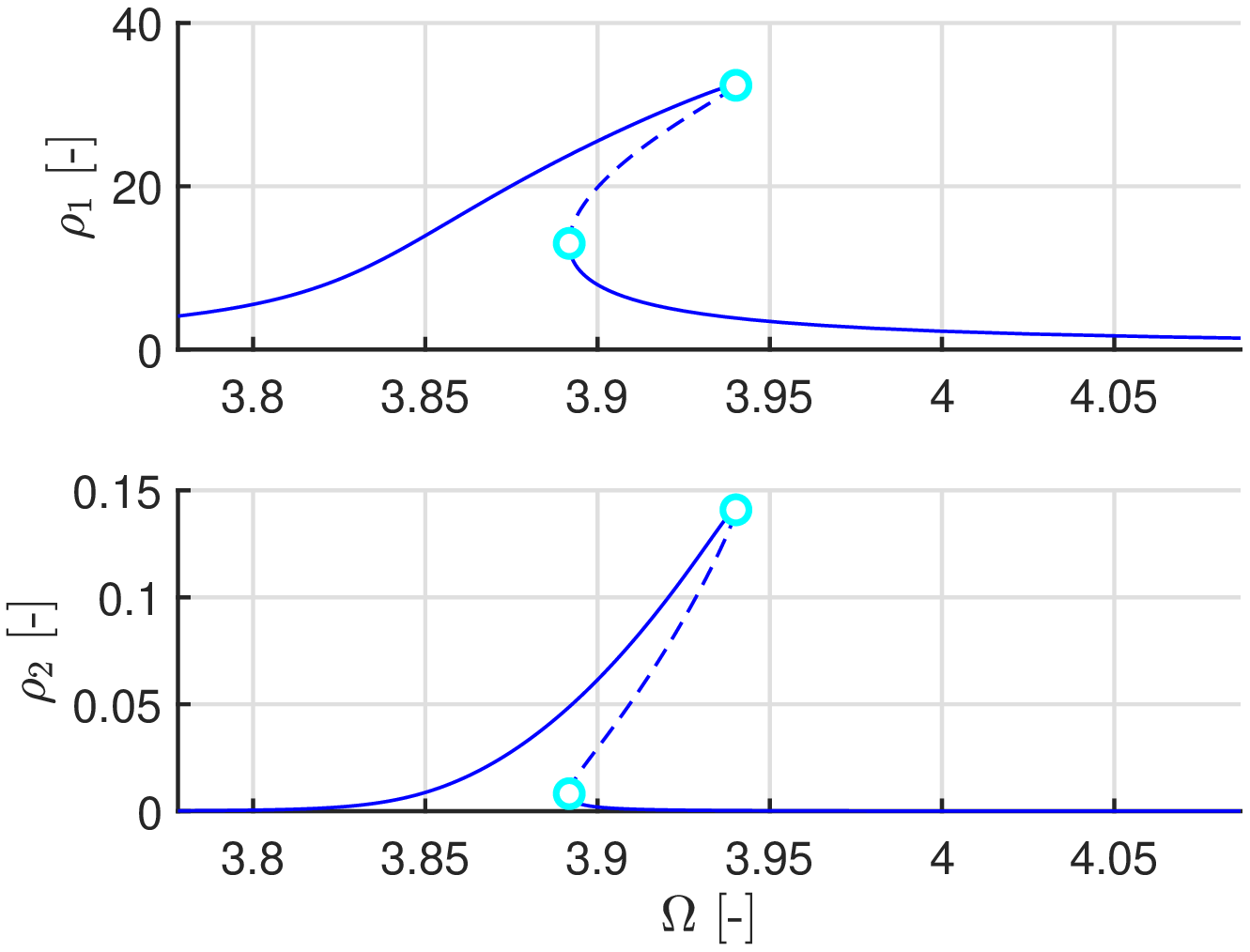}
\includegraphics[width=0.45\textwidth]{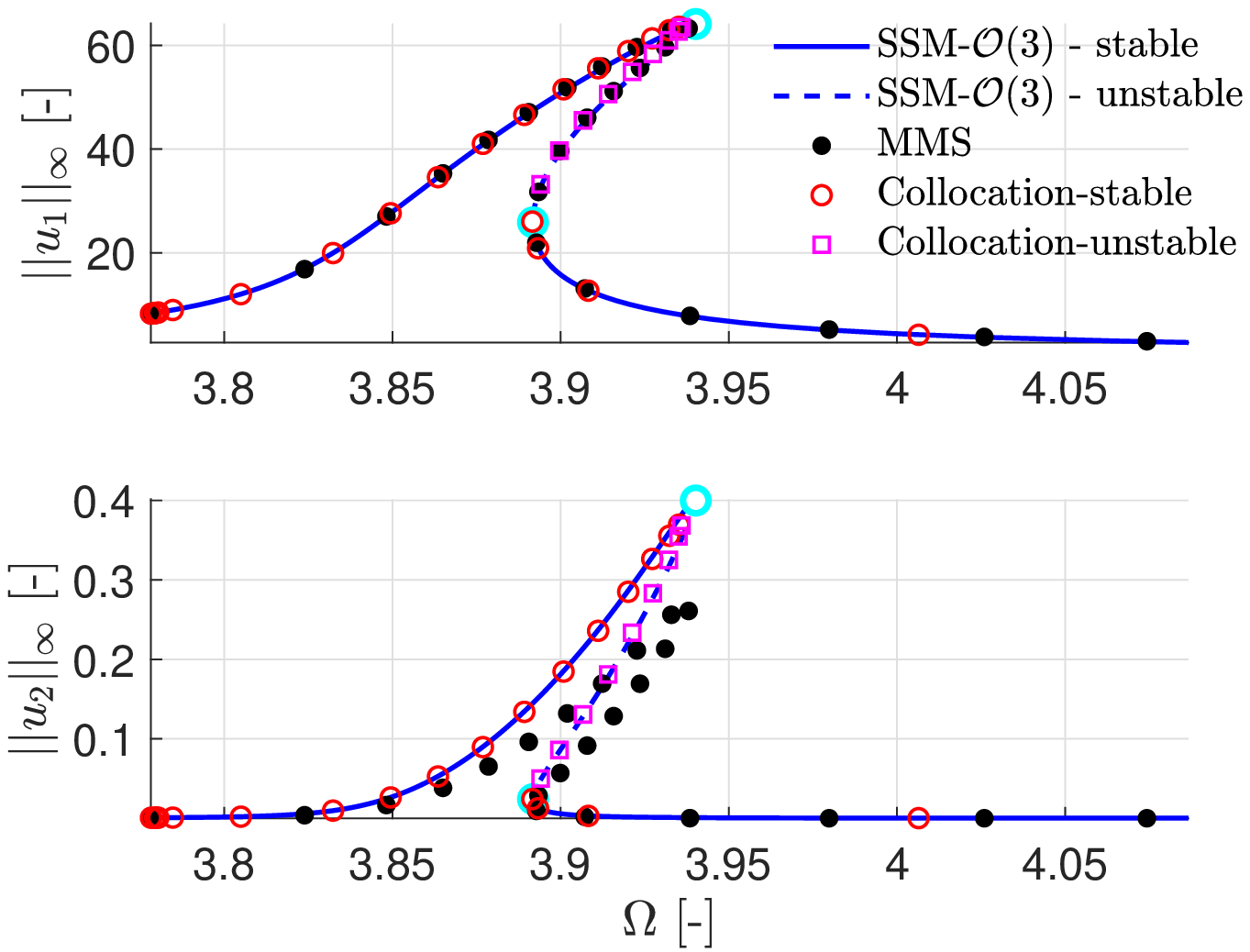}
\caption{The FRC for the forced beam equations~\eqref{eq:beam-odes} in normal coordinates ($\rho_1,\rho_2$) and modal coordinates ($u_1,u_2$) with $\Omega\approx\omega_1=3.8533$. The results obtained by the method of multiple scales (MMS), as well as the continuation of periodic orbits with the collocation method, are presented for comparison and validation.}
\label{fig:FirstMode}
\end{figure}

In MMS, the response of $u_2$ at steady state is not affected by $f_2$ because $f_2$ is not involved in the corresponding {secular} equation when $\Omega\approx\omega_1$~\cite{nayfeh1974nonlinear}. In addition, MMS predicts $u_3=\cdots=u_{10}=0$, independently of $f_{1,\cdots,10}$~\cite{nayfeh1974nonlinear}. In contrast, the results of SSM depend on $f_{1,\cdots,10}$ because the \textcolor{black}{non-autonomous} SSM depends on the external forcing, as can be seen in equation~\eqref{eq:expphi-}. Therefore, SSM reduction yields more accurate results than MMS. To further demonstrate this advantage, we consider the case $\epsilon f_1=\cdots=\epsilon f_{10}=5$, in which, all 10 modes are excited. In this case, MMS returns the same results as in previous loading case, while the results obtained by SSM reduction change, as seen in Fig.~\ref{fig:FirstMode_AllF}. Indeed, the amplitude $||u_2||_{\infty}$ at $\Omega\leq3.85$ and $\Omega\geq3.95$ increases due to the non-vanishing $f_2$. In addition, SSM reduction correctly predicts the nontrivial response of $u_3$, whereas MMS predicts zero response in $u_3$.

\begin{figure}[!ht]
\centering
\includegraphics[width=0.45\textwidth]{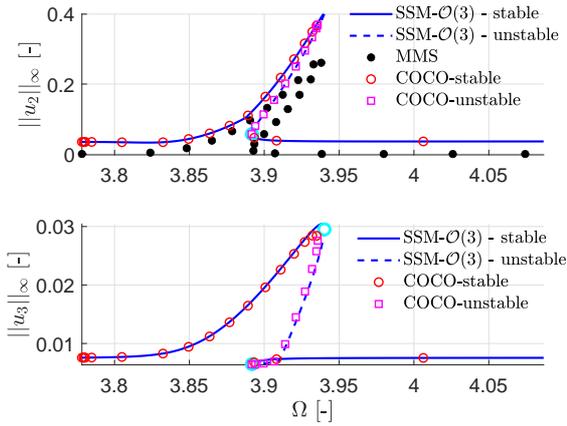}
\caption{The FRC for the forced beam equations~\eqref{eq:beam-odes} in modal coordinates $(u_2,u_3)$ for $\Omega\approx\omega_1=3.8533$ and $\epsilon f_1=\cdots=\epsilon f_{10}=5$. The MMS incorrectly predicts $||u_3||_{\infty}\equiv0$ (not shown).}
\label{fig:FirstMode_AllF}
\end{figure}

\begin{sloppypar}
A further advantage of SSM analysis over MMS is that the reduced dynamics on the SSM is four dimensional while the MMS has to be applied to the full system. Therefore, the computational cost of SSM reduction is smaller than that of MMS when it comes to the size of problems. In addition, MMS is a symbolic method that requires significant efforts in symbolic computation and derivation. The SSM computation, in contrast, is a fully automated, recursive numerical procedure~\cite{ponsioen2018automated,SHOBHIT}.
\end{sloppypar}

\subsubsection{Primary resonance of the second mode}
Letting $\epsilon=1\times10^{-4}$, $c=10$, $f_1=0$, $\epsilon f_2=40$ and $f_3=\cdots=f_{10}=0$, we are interested in the forced response for $\Omega\approx\omega_2$. In this setting, the second mode is excited and hence $\rho_2\neq0$. The first mode, however, can be either excited or inactive. As a consequence, there are two solution branches where $\rho_1=0$ and $\rho_1\neq0$, respectively~\cite{nayfeh1974nonlinear}. Given the possibility that $\rho_1=0$, we use {Cartesian} coordinates here with $\boldsymbol{r}=(1/3,1)$.

\begin{sloppypar}
We first consider the solution branch with non-vanishing $\rho_1$. Providing an initial solution on such a branch to parameter continuation is a challenging task because this branch is an isola: it is isolated from the branch with vanishing $\rho_1$~\cite{ponsioen2019analytic}. Here we provide an initial guess for parameter continuation based on the solution from the MMS. The FRCs obtained in this way in the coordinates $(\rho_1,\rho_2)$ and $(u_1,u_2)$ for $\Omega\in[11.7431, 13.9918]$ are shown in Fig.~\ref{fig:SecondModeNonZero}. From the first two panels, we have $\mathcal{O}(\rho_1)\sim\mathcal{O}(\rho_2)$ for $\Omega\geq 13$ and $\rho_1\gg\rho_2$ for $\Omega\leq 12.5$. Therefore, the system response can be dominated by the first mode although the external forcing is applied to the second mode ($f_1=0, f_2\neq0$). This intriguing phenomenon is a result of the modal interaction arising from the internal resonance. As can be seen in the last two panels, the results of the two methods match well.
\end{sloppypar}

\begin{figure}[!ht]
\centering
\includegraphics[width=0.45\textwidth]{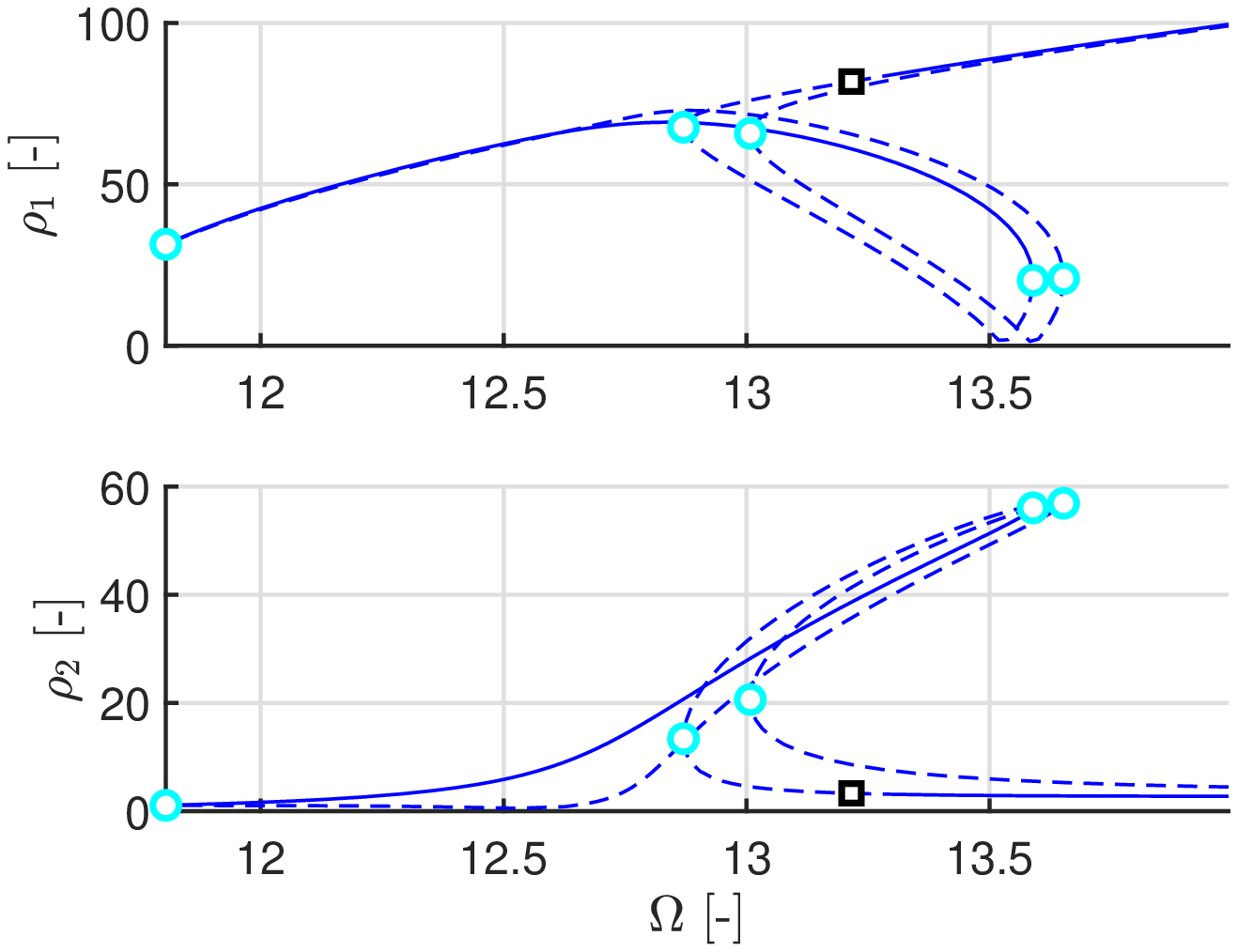}
\includegraphics[width=0.45\textwidth]{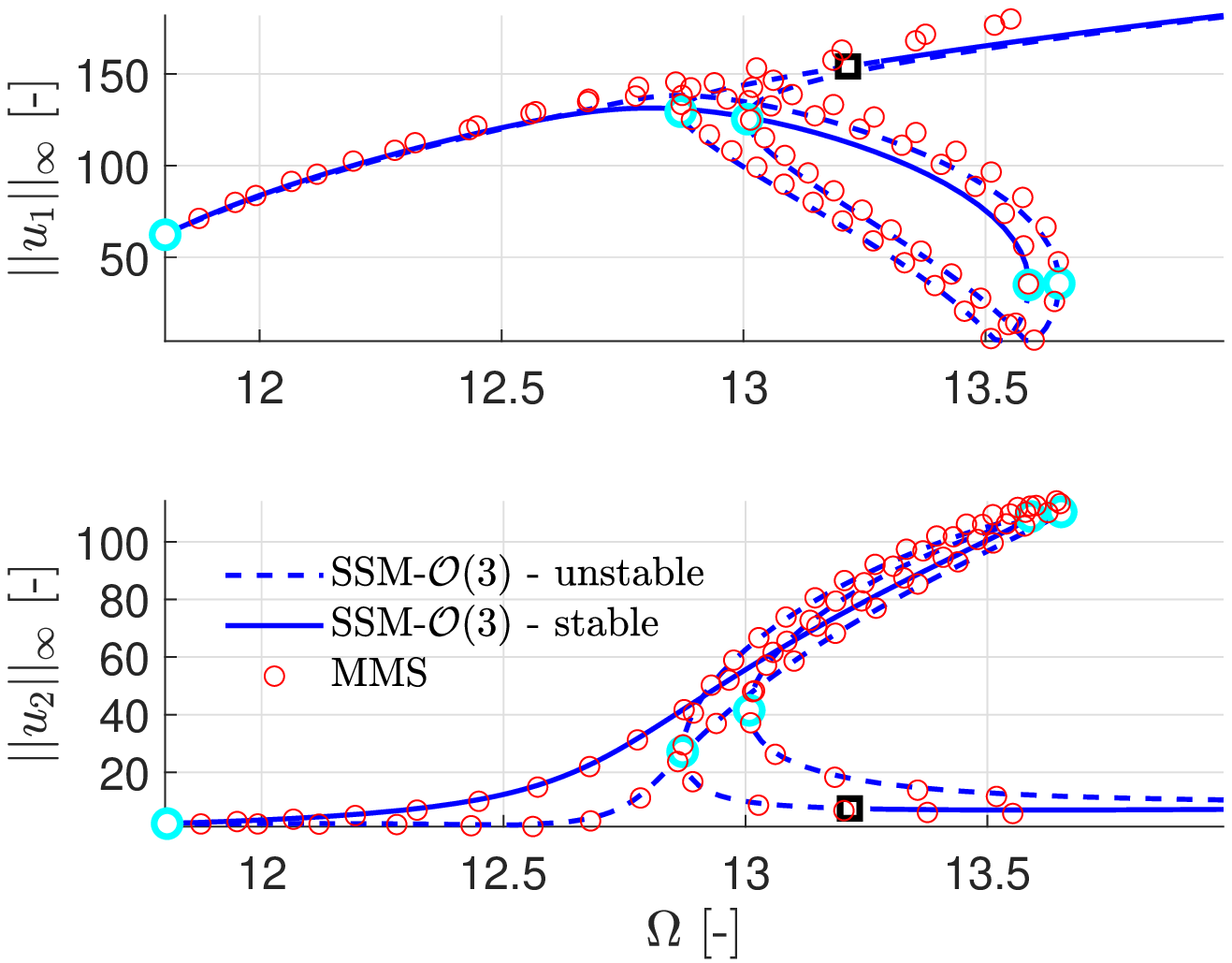}
\caption{The FRC for the forced beam equations~\eqref{eq:beam-odes} in normal coordinates ($\rho_1,\rho_2$) and modal coordinates ($u_1,u_2$) for $\Omega\approx\omega_2=12.4927$ and $\rho_1\neq0$. The results obtained by the method of multiple scale (MMS) are also shown for comparison.}
\label{fig:SecondModeNonZero}
\end{figure}

We then move to the solution branch with vanishing $\rho_1$. The FRCs in the coordinates $(\rho_1,\rho_2)$ and $(u_1,u_2)$ are shown in Fig.~\ref{fig:SecondModeZero}. From the first two panels, we have $\rho_1\equiv0$ and the FRC of $\rho_2$ is similar to that of forced Duffing oscillator. Here the {upper} and {lower} branches are computed separately because their connecting point, namely, the other saddle-node (SN) point, is outside the computational domain of $\Omega$. In fact the other SN point is not detected for $\Omega\leq\omega_3$. In the last panel, we observe a good match between the results of $||u_2||_{\infty}$ obtained by the two methods. Again, MMS predicts vanishing $u_1$. In contrast, SSM-based analysis is more accurate, predicting non-vanishing $u_1$ even though $\rho_1\equiv0$, as can be seen in the third panel of Fig.~\ref{fig:SecondModeZero}.

\begin{figure}[!ht]
\centering
\includegraphics[width=0.45\textwidth]{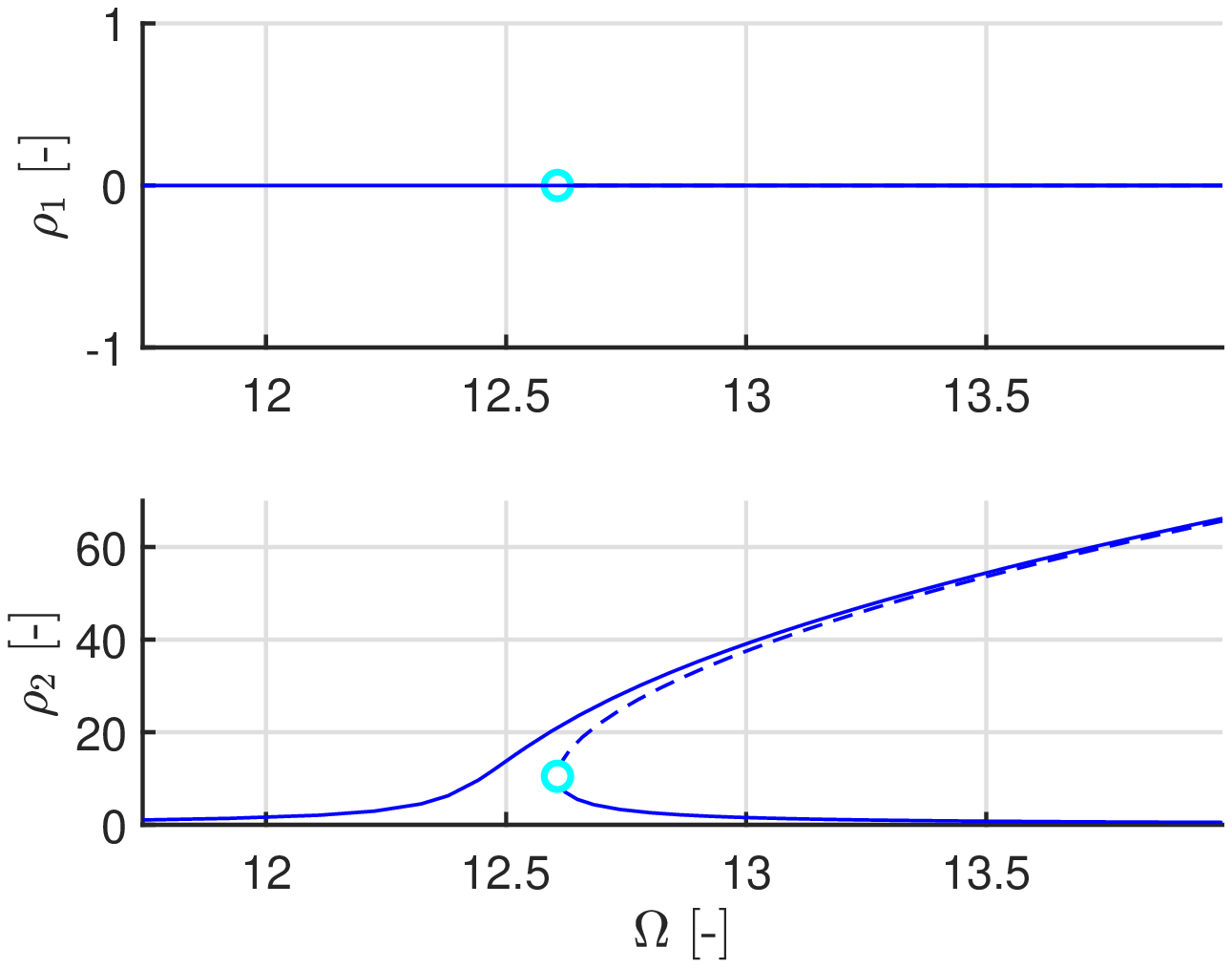}
\includegraphics[width=0.45\textwidth]{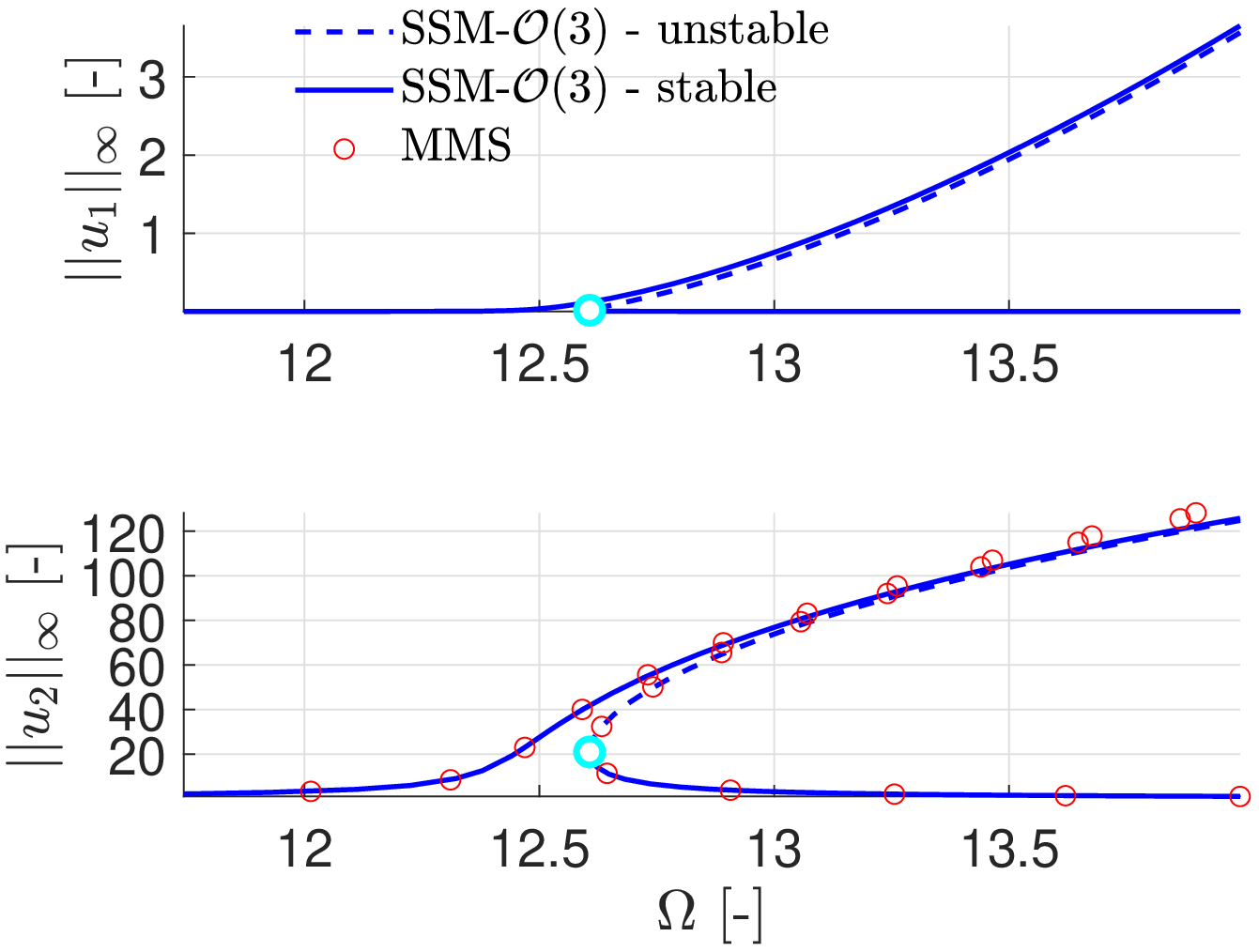}
\caption{The FRC for the forced beam equations~\eqref{eq:beam-odes} in normal coordinates ($\rho_1,\rho_2$) and modal coordinates ($u_1$,$u_2$) for $\Omega\approx\omega_2=12.4927$ and $\rho_1\equiv0$. The MMS predicts $||u_1||_{\infty}\equiv0$ (not shown).}
\label{fig:SecondModeZero}
\end{figure}

\textcolor{black}{\subsection{A viscoelastic beam with gyroscopic force}}
\textcolor{black}{Next, to demonstrate the capability of our SSM reduction for systems with \emph{gyroscopic} and \emph{nonlinear damping} forces, we consider a viscoelastic axially moving beam subject to harmonic base excitation, illustrated in Fig.~\ref{fig:axial-moving-beam}}.

\begin{figure}[!ht]
\centering
\includegraphics[width=0.45\textwidth]{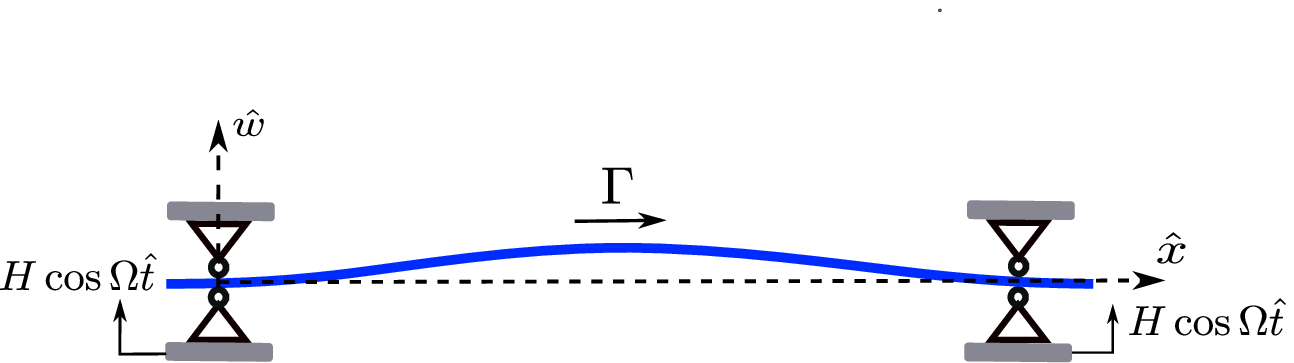}
\caption{\textcolor{black}{A pinned-pinned axially moving beam subject to harmonic base excitation.}}
\label{fig:axial-moving-beam}
\end{figure}

\textcolor{black}{The mechanics of axially moving slender beams and strings have received much attention in the past several decades in connection with power transmission belts, tramways, and band saw blades etc~\cite{pellicano2000nonlinear}. Consider a uniform axially moving viscoelastic beam, with density $\rho$, cross-section area $A$, moment of inertia  $I$ and initial tension $P$, travelling at an axial speed $\Gamma$ between two simple supports that are distance $L$ apart. The support foundation is subject to a harmonic oscillation $H\cos\Omega\hat{t}$. Let the transverse displacement of the beam observed in a frame attached to the oscillating foundation be $\hat{w}(\hat{x},\hat{t})$, which is a function of time $\hat{t}$ and axial coordinate $\hat{x}$. With viscoelastic Kelvin constitutive law 
\begin{equation}
\label{eq:viscoelastic}
    \hat{\sigma}=E\hat{\epsilon}+\eta\frac{\partial\hat{\epsilon}}{\partial\hat{t}},
\end{equation}
where $\hat{\sigma}$ and $\hat{\epsilon}$ denote stress and strain, and $E$ and $\eta$ are the Young’s modulus and viscosity of the beam material, the equation of motion is given by~\cite{xiaodong2006non}
\begin{align}
\label{eq:eom-movingbeam}
    & \rho A \left(\frac{\partial^2\hat{w}}{\partial\hat{t}^2}+2\Gamma\frac{\partial^2\hat{w}}{\partial\hat{x}\partial\hat{t}}+\Gamma^2\frac{\partial^2\hat{w}}{\partial\hat{x}^2}\right)-P\frac{\partial^2\hat{w}}{\partial\hat{x}^2}+EI\frac{\partial^4\hat{w}}{\partial\hat{x}^4}\nonumber\\
    & +\eta I\frac{\partial^5\hat{w}}{\partial\hat{t}\partial\hat{x}^4}=\frac{A}{L}\int_0^L\left[\frac{E}{2}\left(\frac{\partial\hat{w}}{\partial\hat{x}}\right)^2+\eta\frac{\partial\hat{w}}{\partial\hat{x}}\frac{\partial^2\hat{w}}{\partial\hat{x}\partial\hat{t}}\right]\mathrm{d}\hat{x}\frac{\partial^2\hat{w}}{\partial\hat{x}^2}\nonumber\\
    & +\rho AH\Omega^2\cos\Omega\hat{t}
\end{align}
and boundary conditions
\begin{equation}
    \hat{w}(0,\hat{t})=\hat{w}(L,\hat{t})=0,\,\, \frac{\partial^2\hat{w}}{\partial\hat{x}^2}(0,t)=\frac{\partial^2\hat{w}}{\partial\hat{x}^2}(L,t)=0.
\end{equation}}

\textcolor{black}{Similarly to ~\cite{xiaodong2006non}, we introduce the following dimensionless variables and parameters
\begin{gather}
    w=\frac{\hat{w}}{L},\quad x=\frac{\hat{x}}{L},\quad t=\hat{t}\sqrt{\frac{P}{\rho AL^2}},\nonumber\\
    \gamma=\Gamma\sqrt{\frac{\rho A}{P}},\quad k_f^2=\frac{EI}{PL^2},\quad
    \alpha=\frac{I\eta}{L^3\sqrt{\rho AP}},\nonumber\\ k_1=\sqrt{\frac{EA}{P}},\quad\omega=\Omega\sqrt{\frac{\rho AL^2}{P}},\quad \epsilon=\frac{H}{L},
\end{gather}
to obtain the nondimensionalized form of~\eqref{eq:eom-movingbeam} as
\begin{align}
\label{eq:eom-movingbeam-dimless}
    & \frac{\partial^2{w}}{\partial{t}^2}+2\gamma\frac{\partial^2{w}}{\partial{x}\partial{t}}+(\gamma^2-1)\frac{\partial^2{w}}{\partial{x}^2}+k_f^2\frac{\partial^4{w}}{\partial{x}^4}+\alpha\frac{\partial^5{w}}{\partial{t}\partial{x}^4}\nonumber\\
    & =\int_0^1\left[\frac{1}{2}k_1^2\left(\frac{\partial{w}}{\partial{x}}\right)^2+\alpha\frac{k_1^2}{k_f^2}\frac{\partial{w}}{\partial{x}}\frac{\partial^2{w}}{\partial{x}\partial{t}}\right]\mathrm{d}{x}\frac{\partial^2{w}}{\partial{x}^2}\nonumber\\ &\quad +\epsilon\omega^2\cos\omega {t}.
\end{align}
The equation above is consistent with the literature (equivalent to equation (15) in~\cite{xiaodong2006non} when the nonlinear damping effects are ignored; equivalent to equation (6) in~\cite{pellicano2000nonlinear} when both damping and forcing terms are ignored).}

\textcolor{black}{Similar to the previous example, we apply the Galerkin approach to discretize the equation of motion. With a modal expansion
\begin{equation}
    w(x,t)=\sum_{j=1}^n \sin(j\pi x)u_j(t),
\end{equation}
the Galerkin projection yields a system of ODEs as follows
\begin{equation}
\label{eq:eom-movingbeam-ode}
    \ddot{\boldsymbol{u}}+(\boldsymbol{C}+\boldsymbol{G})\dot{\boldsymbol{u}}+\boldsymbol{K}\boldsymbol{u}+\boldsymbol{f}(\boldsymbol{u},\dot{\boldsymbol{u}})=\epsilon\omega^2\boldsymbol{g}\cos\omega t,
\end{equation}
where $\boldsymbol{u}=(u_1,\cdots,u_n)$ and similarly to~\cite{pellicano2000nonlinear}, we have
\begin{gather}
    C_{ij}=\alpha(i\pi)^4\delta_{ij},\quad G_{ij}=4\gamma\frac{ ij}{i^2-j^2}\left(1-(-1)^{i+j}\right),\nonumber\\ K_{ij}=\left(k_f^2(i\pi)^4-(\gamma^2-1)(i\pi)^2\right)\delta_{ij}\nonumber,\\
    f_i = \frac{1}{4} k_1^2\pi^4i^2\sum_j\left({j^2}u_j^2\right)u_i+\frac{\alpha}{2}\frac{k_1^2}{k_f^2}\pi^4i^2\sum_j\left({j^2}u_j\dot{u}_j\right)u_i,\nonumber\\  g_i= \frac{1-(-1)^i}{i\pi}
\end{gather}
for $i,j=1,\cdots, n$. The $\delta_{ij}$ above is Kronecker delta and $G_{ii}=0$. Note that $\boldsymbol{G}^\mathrm{T}=-\boldsymbol{G}$ is a gyroscopic matrix and we have cubic nonlinear damping due to the viscoelastic constitutive law~\eqref{eq:viscoelastic}.}

\textcolor{black}{Following~\cite{tang2019nonlinear}, the parameters of the model are chosen as $A=1.2\times10^{-3}\,\mathrm{m}^2$, $I=9\times10^{-8}\,\mathrm{m}^3$, $\rho=7680\,\mathrm{kg/{m}^3}$, $E=30\times10^9\,\mathrm{Pa}$, $L=1\,\mathrm{m}$ and $P=6.75\times10^4\mathrm{N}$. The dimensionless parameters are obtained as $k_f=0.2$ and $k_1=23.0940$. With $\gamma=0.5128$, the first two natural frequencies of the linear, unforced part of~\eqref{eq:eom-movingbeam-ode} are given by $\omega_1\approx3.1954$ and $\omega_2\approx9.5862\approx3\omega_1$. In the following computation, we select the viscoelastic parameter $\eta=1\times10^{-4}E$ and $n=10$.}

\textcolor{black}{Similarly to the previous example, we take the first two pairs of modes as the master spectral subspace to account for the near 1:3 internal resonance, resulting in a four-dimensional reduced-order model. For the base excitation amplitude $\epsilon=1.5\times10^{-4}$, we found that the forced response curve converges well at $\mathcal{O}(5)$. The FRC is plotted in Fig.~\ref{fig:FRC-axial-moving-beam}, where we also observe agreement between the results of SSM reduction and collocation method applied to the full system using the \texttt{po} toolbox in \textsc{coco}~\cite{dankowicz2013recipes}. To explore the effects of nonlinear damping, we also calculate the FRC for system~\eqref{eq:eom-movingbeam-ode} with the nonlinear damping ignored. We observe that the nonlinear damping effects become significant as the response amplitudes increase.}

\begin{figure}[!ht]
\centering
\includegraphics[width=0.45\textwidth]{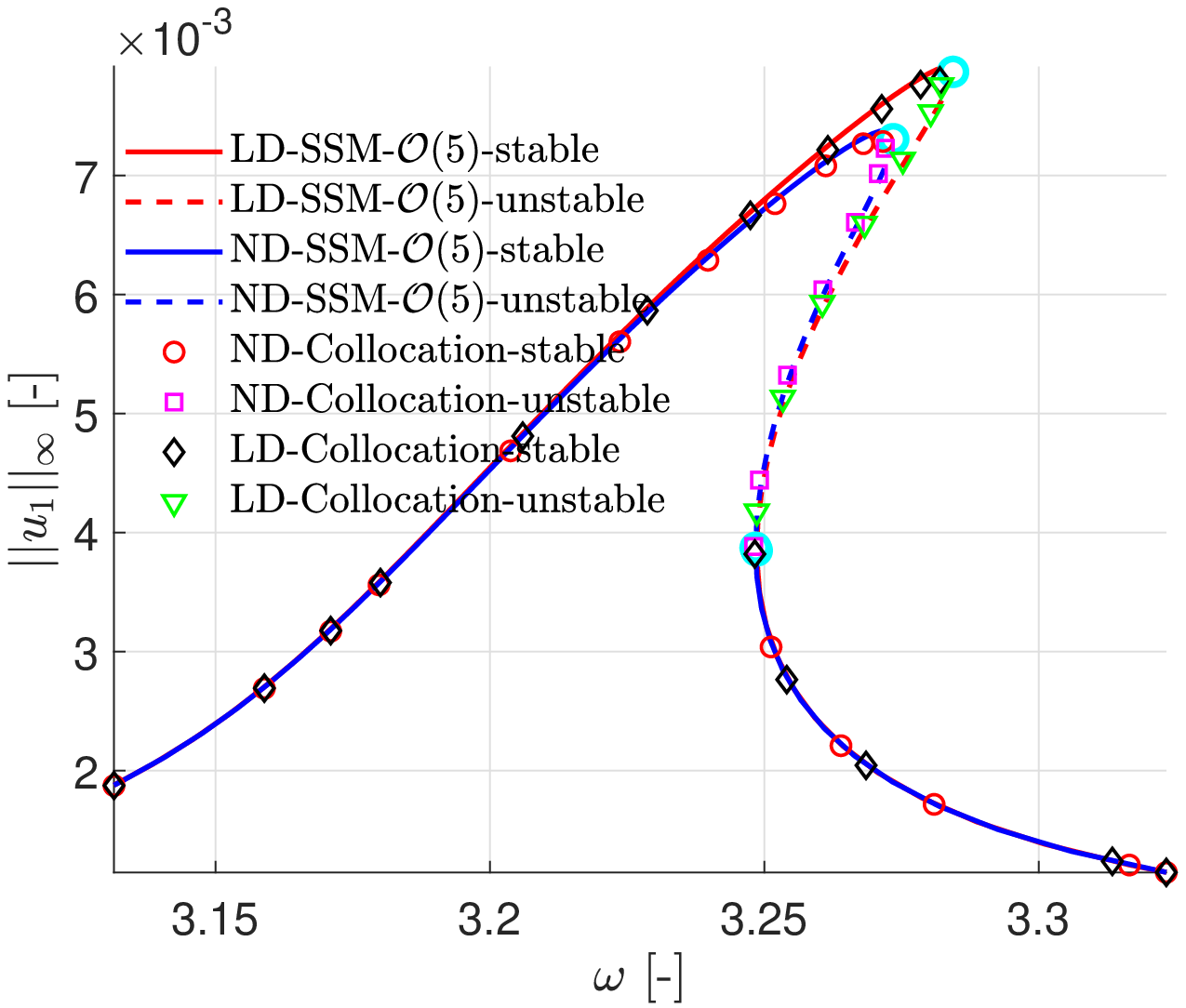}\\
\includegraphics[width=0.45\textwidth]{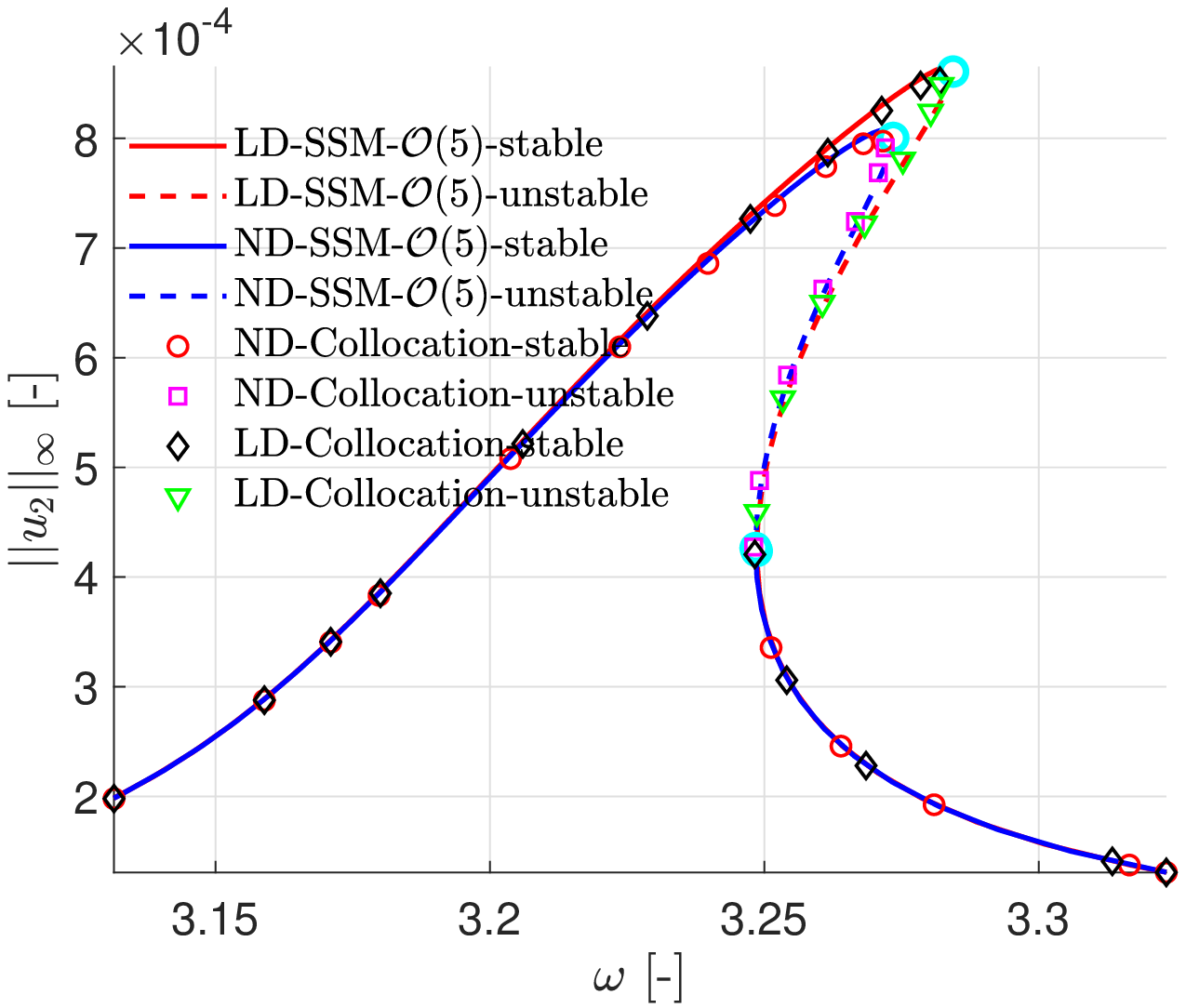}
\caption{\textcolor{black}{The FRC for the axially moving beam~\eqref{eq:eom-movingbeam-ode} in modal coordinates $(u_1,u_2)$ for $\omega\approx\omega_1=3.1954$. The `LD' and `ND' in the legend represent linear and nonlinear damping respectively. The results obtained by the continuation of periodic orbits with the collocation method agree with those obtain via SSM-based reduced-order models.}}
\label{fig:FRC-axial-moving-beam}
\end{figure}

\subsection{A von K\'arm\'an beam with support spring}
\label{sec:vonKarmanBeam}
\textcolor{black}{To demonstrate the computational efficiency of our SSM-based reduction procedure, we shift our focus to higher-dimensional finite element models. First, we} consider a clamped-pinned von K\'arm\'an beam with a support linear spring at its midspan, as shown in Fig.~\ref{fig:clamped_pinned_force}. This example is distinct from the \textcolor{black}{example~\ref{sec:prismatic_beam}} in the following aspects:
\begin{itemize}
\item A linear spring is attached at the midspan of the beam and the stiffness of the spring is tuned to trigger an exact 1:3 internal resonance, $\omega_2=3\omega_1$, such that the modal interaction in the primary resonance of the first mode is highlighted;
\item The beam structure is modeled using the von K\'arm\'an beam theory~\cite{reddy2015introduction} and hence both {axial} and {transverse} displacements are included as unknowns. Thus, the axial stretching effect is taken into account automatically.
\item \begin{sloppypar}The governing equation is discretized using the finite element method instead of a modal expansion. With an increasing number of elements, ranging from 8 to 10,000, we demonstrate the remarkable computational efficiency of SSM reduction relative to the harmonic balance method and collocation schemes applied to the full systems.\end{sloppypar}
\end{itemize}

\begin{figure}[!ht]
\centering
\includegraphics[width=0.45\textwidth]{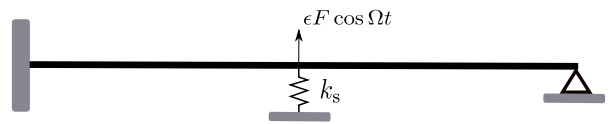}
\caption{A clamped-pinned von K\'arman beam with a spring support and a harmonic excitation at its midspan.}
\label{fig:clamped_pinned_force}
\end{figure}

\textcolor{black}{We set the width and height of the cross section to be $10\,\mathrm{mm}$ and the length of the beam to be $2700\,\mathrm{mm}$. Material properties are specified with the density $ 1780\times10^{-9}\,\mathrm{kg/{mm}^3}$ and the Young's modulus $45\times10^6\,\mathrm{kPa}$.} Following a finite element discretization, three DOF are introduced at each node: the axial and transverse displacements, and the rotation angle. The equations of motion of the discrete model can be written as
\begin{equation}
\boldsymbol{M}\ddot{\boldsymbol{x}}+\boldsymbol{C}\dot{\boldsymbol{x}}+\boldsymbol{K}\boldsymbol{x}+\boldsymbol{N}(\boldsymbol{x})=\epsilon\boldsymbol{f}\cos\Omega t
\end{equation}
where $\boldsymbol{x}\in\mathbb{R}^{3N_{\mathrm{e}}-2}$ is the assembly of all DOF, and $N_{\mathrm{e}}$ is the number of elements in the discretization. We use Rayleigh damping matrix of the form\textcolor{black}{~\eqref{eq:RayleighDamping}. In this example, we set $\alpha=0$ and $\beta=\frac{2}{9}\times10^{-4}\,\mathrm{s}^{-1}$ such that the system is weakly damped and from eq.~\eqref{eq:weakDampFreq}, we have $\lambda_{2i-1,2i}\approx \omega_i$  for $i\leq 2$}. More details about the formulation of $\boldsymbol{M}$, $\boldsymbol{K}$ and $\boldsymbol{N}$ can be found at~\cite{jain2018exact,FEcode}.

We first tune the stiffness of the support spring, $k_{\mathrm{s}}$, such that $\omega_2=3\omega_1$ holds and hence an 1:3 internal resonance occurs. As can be seen in Fig.~\ref{fig:freq1to3beam}, such an internal resonance arises at $k_\mathrm{s}\approx37\,\mathrm{kg/s^2}$. In the following computations, we set $k_\mathrm{s}=37\,\mathrm{kg/s^2}$ which gives $\omega_1=33.20\,\mathrm{rad/s}$ and $\omega_2=99.59\,\mathrm{rad/s}$.

\begin{figure}[!ht]
\centering
\includegraphics[width=0.4\textwidth]{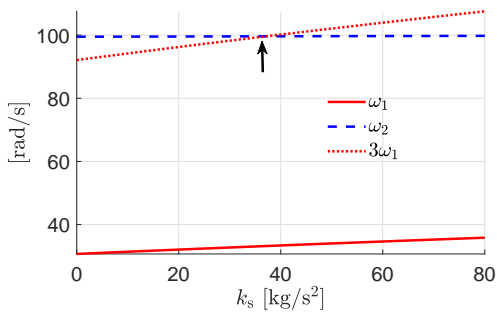}
\caption{Natural frequencies of the clamped-pinned beam with a support spring at its midspan, as functions of the stiffness of the support spring $k_\mathrm{s}$. At the intersection pointed by the arrow, $\omega_2= 3\omega_1$. The beam here is discretized with 100 elements resulting in 298 DOF. Numerical experiments suggest that the position of such \textcolor{black}{an} intersection is robust with respect to the number of elements used in the discretization.}
\label{fig:freq1to3beam}
\end{figure}

Now we consider the forced response of the discretized beam with a transverse load applied at its midspan. Let $F=1000\textcolor{black}{\,\mathrm{mN}}$ and $\epsilon=0.02$, we calculate the FRC for $\Omega$ over the interval $[0.96,1.05]\omega_1$ using SSM reduction, the harmonic balance method with \textsc{nlvib} tool~\cite{krack2019harmonic}, and the collocation method with the \texttt{po} toolbox of \textsc{coco}~\cite{dankowicz2013recipes}. These three methods will be applied to the same discretized beam with an increasing number of beam elements. Notably, when the number of elements is large enough, the mesh is artificially over-refined and the round-off errors are known to accumulate~\cite{liu2021balancing}. Indeed, when the number of elements is 30,000, the first natural frequency significantly deviates from the correct value. For this reason, here we set the upper bound for the number of elements to be 10,000, even though we could handle orders of magnitude more.

The following computations are all performed on a remote Intel Xeon E3-1585Lv5 processor (3.0-3.7 GHz) on the ETH Euler cluster. In the SSM reduction method, we take the first two pairs of complex conjugate modes as the master subspace to account for the 1:3 internal resonance, the same resonance considered in the previous example. This time, however, we use {polar} coordinates because we are interested in the primary resonance of the first mode for which no singularity occurs. It follows that the phase space for the full system is $6N_{\mathrm{e}}-4$ dimensional while the one for the reduced dynamical system is only four dimensional. \textcolor{black}{We observed that cubic approximation of SSM is not able to produce convergent FRC, as seen in Fig.~\ref{fig:FRCs-vonBeam-orders}. Instead, we use $\mathcal{O}(7)$ expansion in this example given the curve converges well at this order.} The \textsc{nlvib} tool and the \texttt{po} toolbox of \textsc{coco} are used to extract the FRC of the full system directly. We have carefully tuned the setting of \textsc{coco} such that the computational time of the collocation method using \texttt{po} is reasonable. Such tuning efforts include disabling some advanced feature of \texttt{po} and increasing maximal continuation step size. More details about the tuning are presented in Appendix~\ref{sec:apdix-coco}. As for the setting of \textsc{nlvib}, we set the number of harmonics to be 10 and the nominal step size to be 2. Note that stability analysis of periodic orbits is not provided in \textsc{nlvib}.

\begin{figure}[!ht]
\centering
\includegraphics[width=0.45\textwidth]{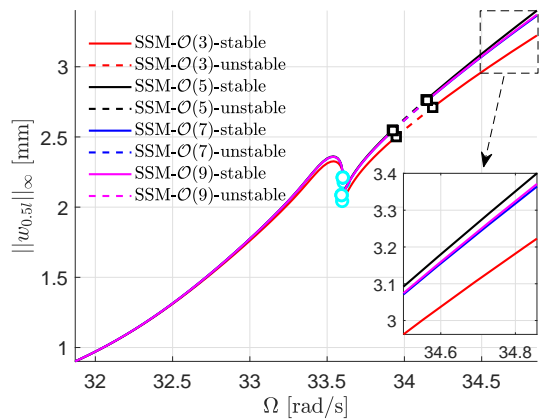}
\caption{\textcolor{black}{The FRCs in the amplitude of transverse displacement at the midspan of the clamped-pinned von K\'arm\'an beam discretized with 8 elements. These FRCs are obtained using SSM computations at different orders.}}
\label{fig:FRCs-vonBeam-orders}
\end{figure}

The computational times of FRC using the three methods with various number of elements are summarized in Fig.~\ref{fig:time_comparison_v4}. In the case of 40 elements, the system has 118 DOF, giving a 236-dimensional phase space. The computation times of FRC using SSM reduction, the harmonic balance method with \textsc{nlvib}, and the collocation method with \textsc{coco} are 14 seconds, 12.5 hours, and 58.5 hours, respectively. Therefore, the SSM reduction produces a significant speed-up gain relative to the other two methods applied to the full system. When the number of elements is further increased, the FRC computations with the harmonic balance method and the collocation method were no longer feasible. On the other hand, the SSM reduction only took about one hour to obtain the FRC in the case of 10,000 elements with 29,998 DOF.

\begin{figure}[!ht]
\centering
\includegraphics[width=0.48\textwidth]{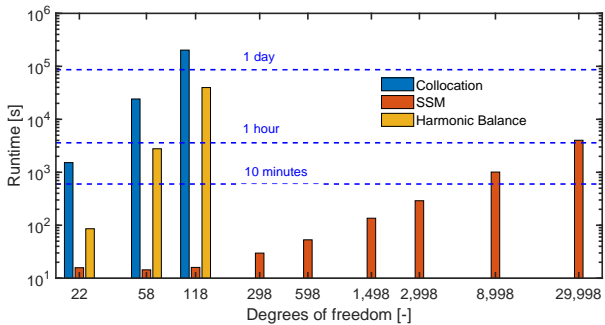}
\caption{Computational times for the FRC of the clamped-pinned von K\'arm\'an beam discretized with different number of elements. The number of DOF is given by $3N_{\mathrm{e}}-2$ when the beam is discretized with $N_{\mathrm{e}}$ elements. Here we have $N_{\mathrm{e}}\in\{$8, 20, 40, 100, 200, 500, 1,000, 3,000, 10,000$\}$. The upper bound of $N_{\mathrm{e}}$ is set to be 10,000 to avoid the accumulation of round-off errors induced by over-refined meshes.}
\label{fig:time_comparison_v4}
\end{figure}

The FRC obtained for the transverse vibration at the midspan and $1/4$ of the beam are plotted in Fig.~\ref{fig:FRCs-vonBeam-physics-v4}. The results obtained by the above three methods match well in the case of 8, 20 and 40 elements. We also use numerical integration to validate the results obtained by SSM reduction for the beam discretized with larger number of elements, where the harmonic balance method and the collocation method become impractical. Specifically, Newmark-beta integration is applied to the full systems and the responses at the \textcolor{black}{Poincar\'e} section $\{t:\mathrm{mod}(t,T)=0\}$, namely, $\boldsymbol{z}(0),\boldsymbol{z}(T),\boldsymbol{z}(2T),\cdots$ are recorded, where $T={2\pi}/{\Omega}$ is the period of harmonic excitation. The numerical integration terminates once the following periodicity condition is satisfied:
\begin{equation}
\frac{||\boldsymbol{z}(iT)-\boldsymbol{z}((i-1)T)||}{||\boldsymbol{z}((i-1)T)||}<\mathrm{Tol}.
\end{equation}
In this paper, we set $\mathrm{Tol}=0.001$. To speed up the convergence to steady state in numerical integration, a point on the trajectory obtained by SSM reduction has been chosen as $\boldsymbol{z}(0)$. As can be seen in the last two panels of Fig.~\ref{fig:FRCs-vonBeam-physics-v4}, the results obtained by SSM reduction match well with the ones from direct numerical integration. Results for $N_{\mathrm{e}}\in$$\{$500, 1,000, 3,000, 10,000$\}$ are not plotted here because the results at $N_{\mathrm{e}}=200$ already converge with respect to the increment of the number of elements.

\begin{figure*}[!ht]
\centering
\includegraphics[width=3.0in]{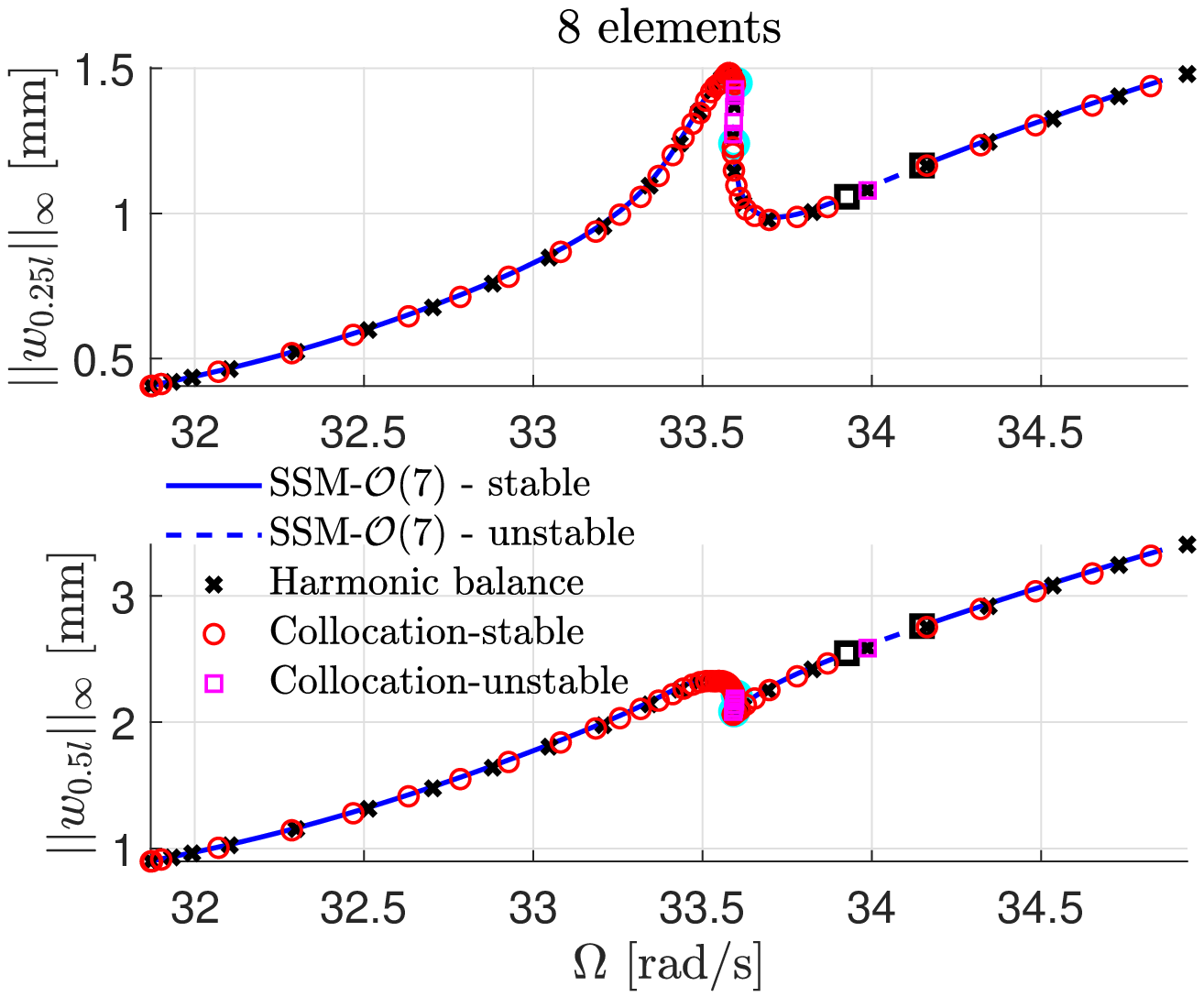}
\includegraphics[width=3.0in]{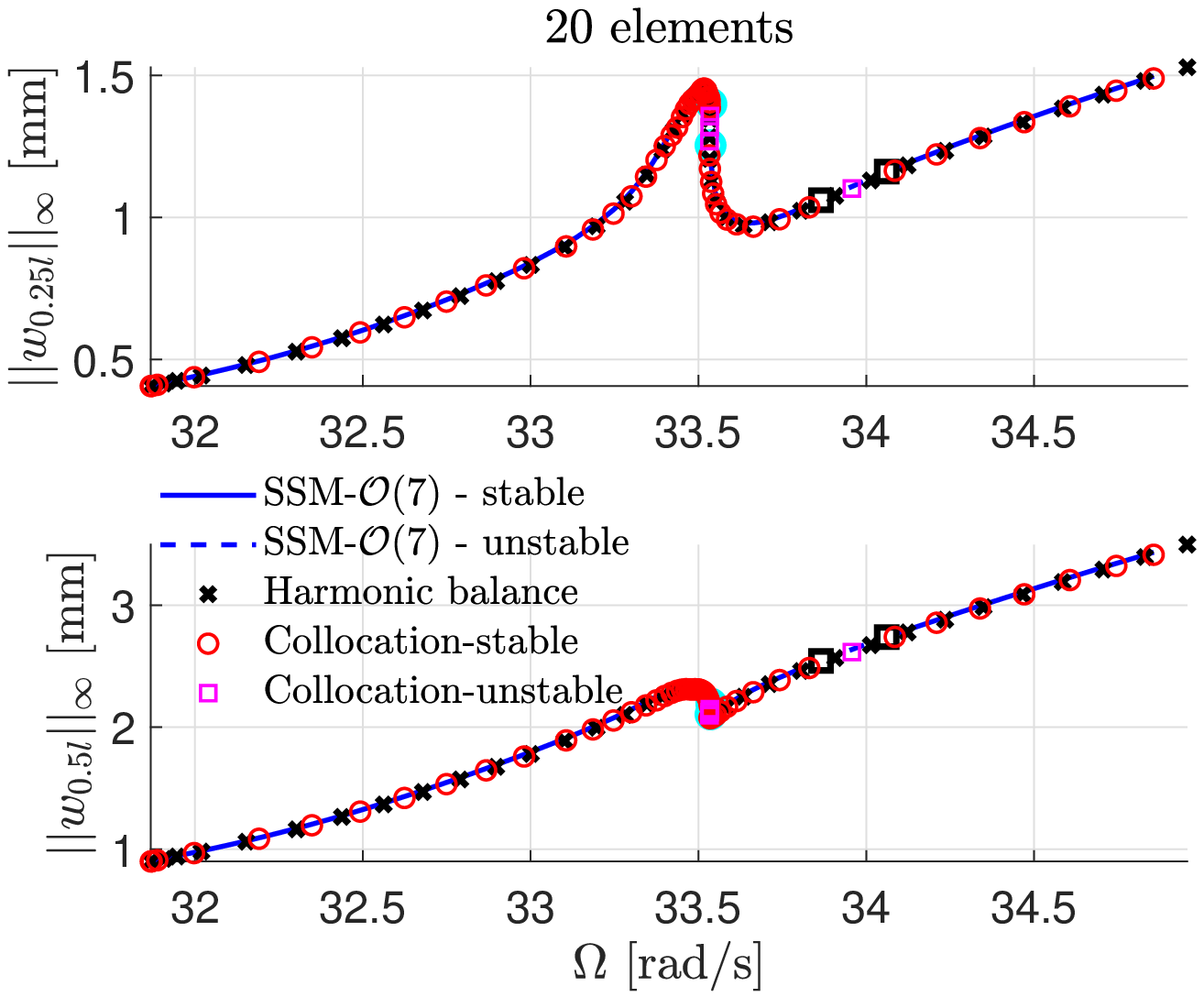}\\
\includegraphics[width=3.0in]{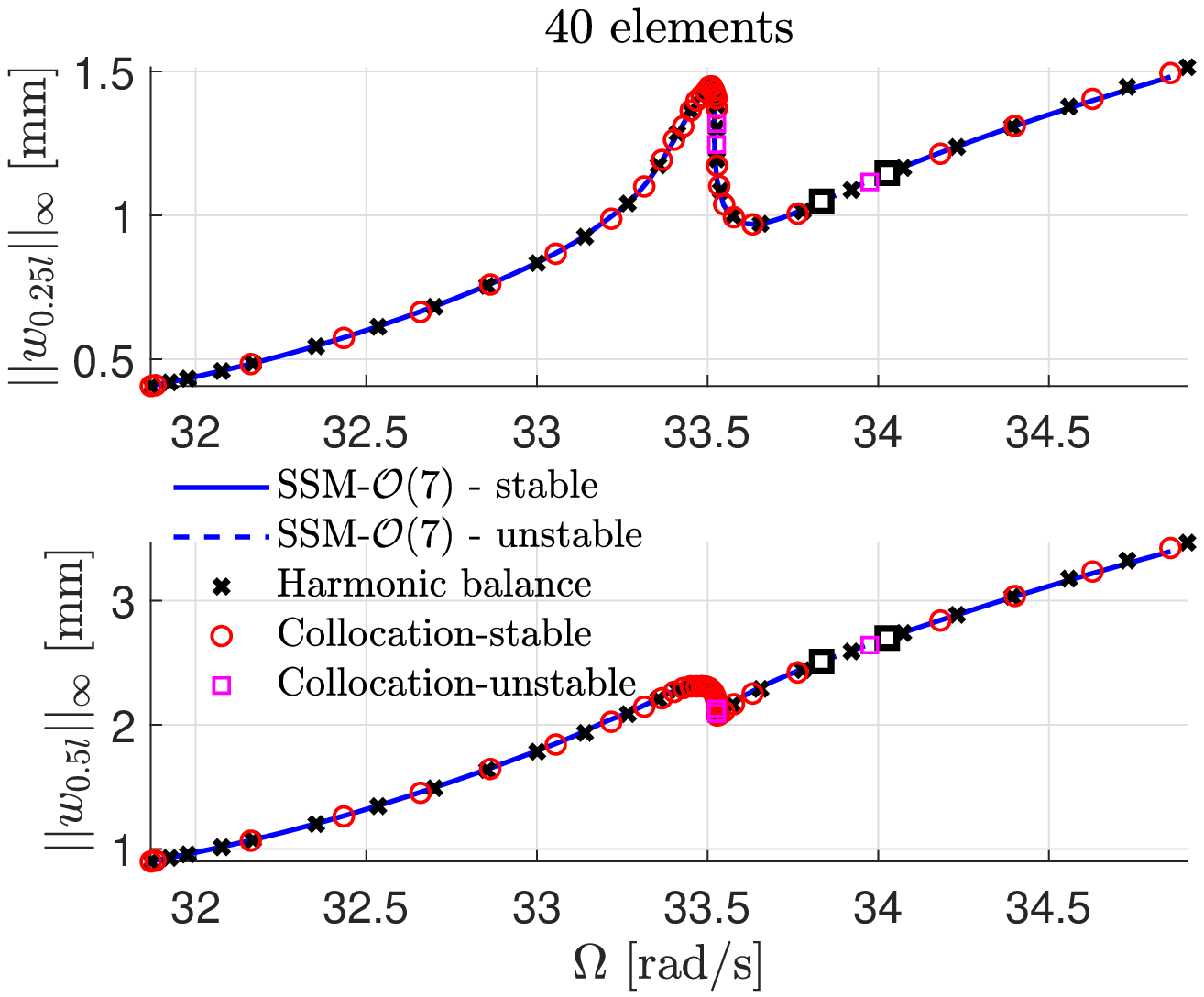}
\includegraphics[width=3.0in]{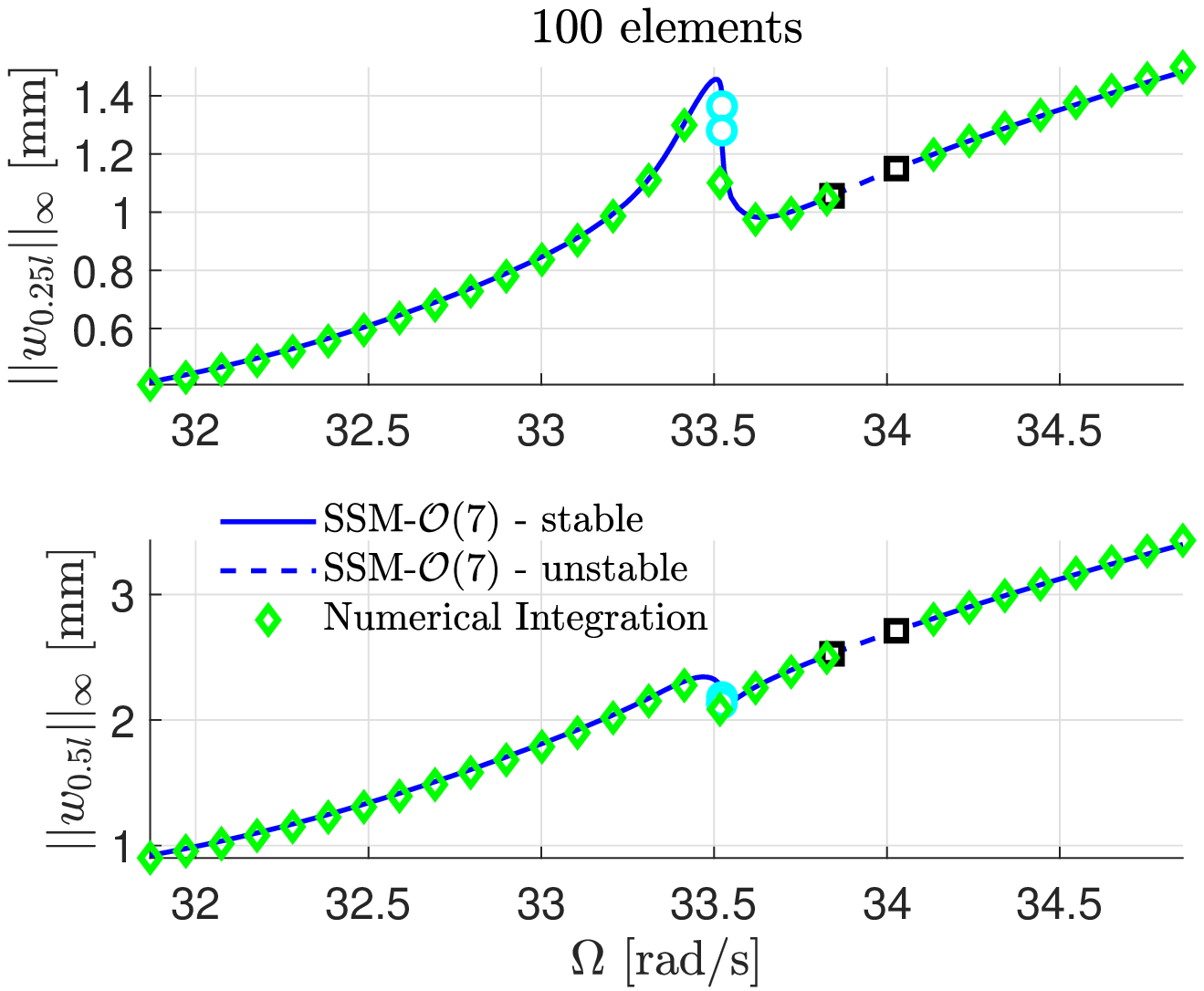}\\
\includegraphics[width=3.0in]{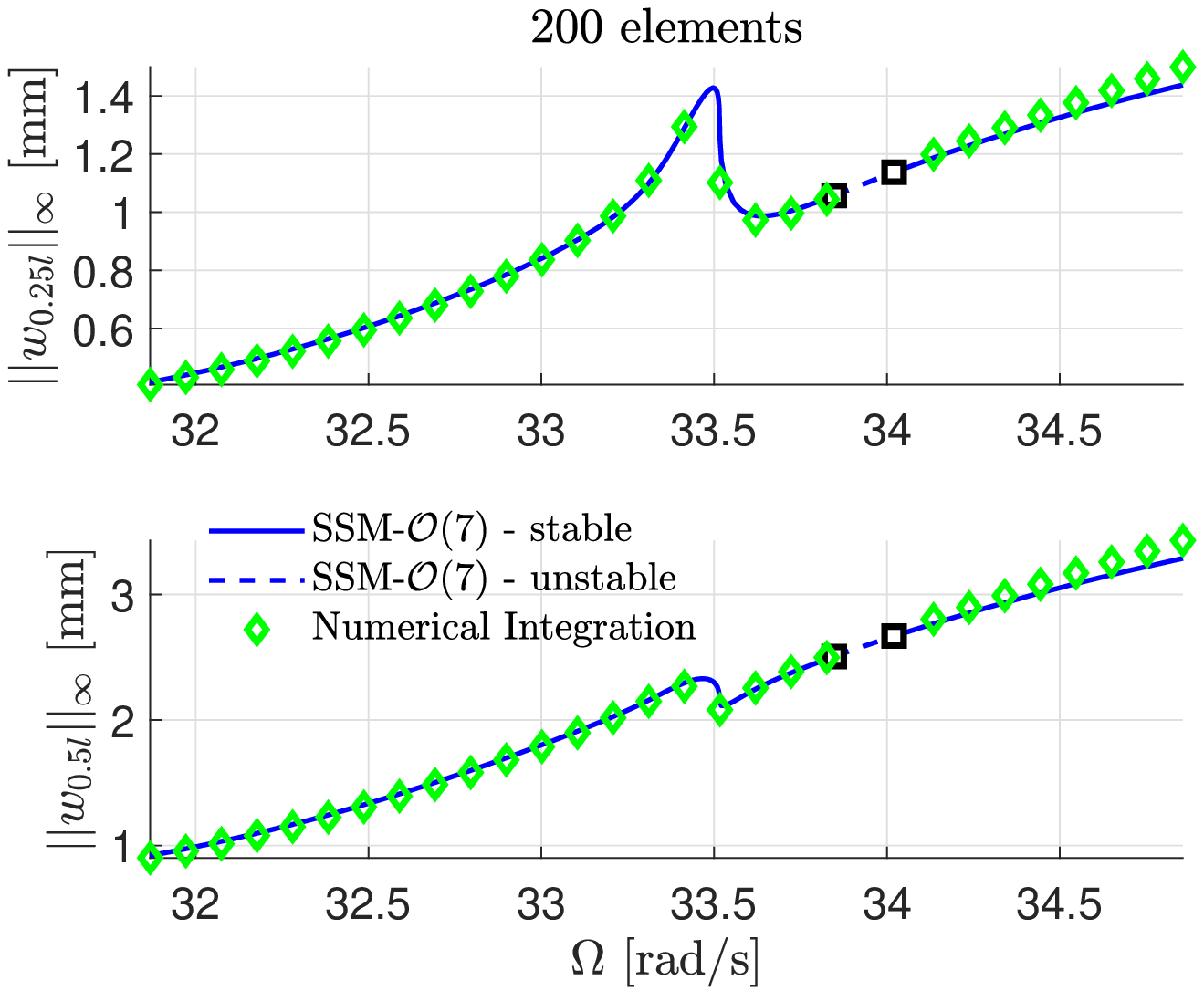}
\caption{The FRC in physical coordinates (the amplitude of transverse displacement $w$ at $0.25l$ and $0.5l$) of the clamped-pinned von K\'arm\'an beam discretized with different numbers of elements.}
\label{fig:FRCs-vonBeam-physics-v4}
\end{figure*}

Energy transfer due to modal interaction is observed in the FRCs discussed above. In particular, when the transverse vibration amplitude at 1/4 of the beam's length arrives its peak around $\Omega=\omega_1$, the transverse vibration amplitude at the midspan drops, as seen in Fig.~\ref{fig:FRCs-vonBeam-physics-v4}. This phenomenon results from the energy transfer between the first and the second bending modes due to the 1:3 internal resonance. Indeed, as can be seen Fig.~\ref{fig:FRCs-vonBeam-modal-v4}, the amplitude of the second mode $\rho_2$ has a peak at $\Omega\approx\omega_1$. In other words, the vibration amplitude of the \emph{second} mode approaches a maximum when $\Omega$ is around the natural frequency of the \emph{first} mode. Meanwhile, the amplitude of the first mode $\rho_1$ drops slightly when $\rho_2$ approaches its maximum. Therefore, the energy of the first mode is transferred to the second mode due to the internal resonance. From the mode shapes of the first and second modes, one can infer that the transverse vibrations at the mid span and at the 1/4 of the beam are representatives of the vibration of the first and the second modes, respectively. Therefore, the FRC of $||w_{0.25l}||_{\infty}$ and $||w_{0.5l}||_{\infty}$ are qualitatively similar to that of $\rho_2$ and $\rho_1$, respectively.

\begin{figure}[!ht]
\centering
\includegraphics[width=0.45\textwidth]{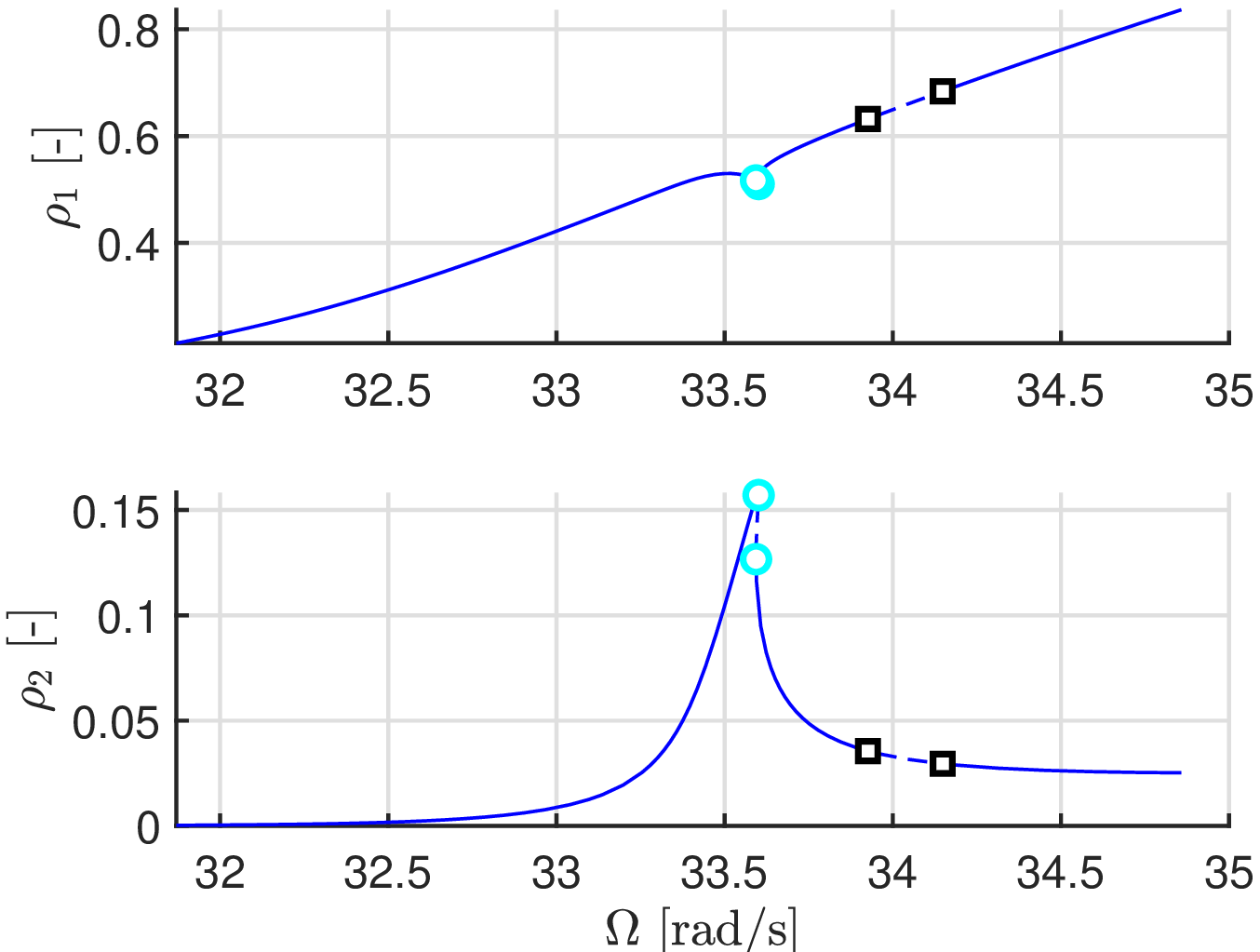}
\caption{FRC in $(\rho_1,\rho_2)$ of the clamped-pinned von K\'arm\'an beam discretized with 8 elements. The corresponding FRC in $||w_{0.25l}||_{\infty}$ and $||w_{0.5l}||_{\infty}$ is presented in the first panel of Fig.~\ref{fig:FRCs-vonBeam-physics-v4}.}
\label{fig:FRCs-vonBeam-modal-v4}
\end{figure}

\textcolor{black}{\subsection{A Timoshenko beam carrying a lumped mass}}
\textcolor{black}{In this section, we consider a finite element model of a geometrically nonlinear Timoshenko beam with an attached mass, as shown in Fig.~\ref{fig:clamped-disk}, to demonstrate the capability of SSM reduction for systems undergoing large deformations.}

\begin{figure}[!ht]
\centering
\includegraphics[width=0.4\textwidth]{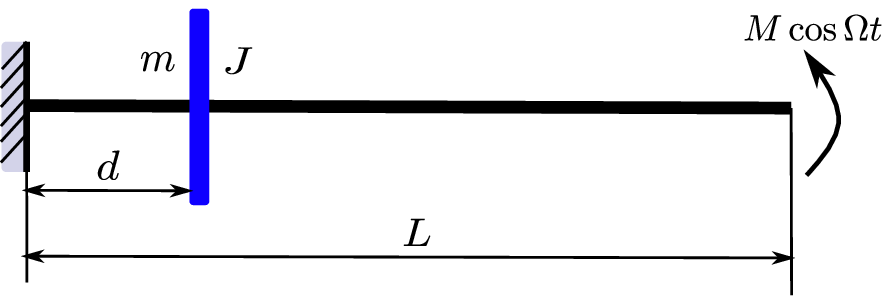}
\caption{\textcolor{black}{The schematic of a cantilever beam carrying a lumped mass $m$ with mass moment of inertia $J$. The beam is subject to an external harmonic moment at the free end.}}
\label{fig:clamped-disk}
\end{figure}

\textcolor{black}{The beam model here is the same as the one in section 7 of~\cite{ponsioen2020model}. The length, width and height of the beam are 1200 mm, 40 mm and 40 mm, respectively. We choose the following values for the material parameters: Density is $7850\,\mathrm{kg/mm^3}$, Young’s modulus is 90 MPa, shear modulus is 34.6 MPa, axial material damping constant is 13.4 Pa-s, and shear material damping constant is 8.3 Pa-s. Inspired by~\cite{zavodney1989non}, we add a lumped mass $m$ with mass moment of inertia $J$ at a position $d$ (cf.~Fig.~\ref{fig:clamped-disk}) from the fixed end and choose appropriate values of $m,J$ and $d$ to introduce internal resonances.}

\textcolor{black}{The cantilever beam is discretized in the same way as the one in section 7 of~\cite{ponsioen2020model}, resulting in a 21 degrees of freedom system (also see Example 7.3 in~\cite{ponsioen2018automated}). For the lumped mass, we choose $d=300\,\mathrm{mm}$, $m=80\,\mathrm{kg}$ and $J=5\times10^6\,\mathrm{kg\cdot mm^3}$, which results in a near 1:3 internal resonance among the first two natural frequencies of the discretized system as $\omega_1=2.2562\,\mathrm{rad/s}$ and $\omega_2=7.2301\,\mathrm{rad/s}\approx 3\omega_1$.}

\textcolor{black}{We apply a harmonic external moment $M\cos\Omega t$ at the free end of the beam and calculate the forced response curve of the system for $\Omega\approx\omega_1$. In particular, we are interested in the vibration amplitude of the transverse deflection of the beam at the free end. Since the system has near 1:3 internal resonance, we again take the first two pairs of complex conjugate modes as the master subspace in SSM reduction, reducing the dimension of phase space from 42 (of the full system) to four.}

\textcolor{black}{We set the moment amplitude $M=0.84\,\mathrm{N\cdot m}$ to obtain the FRCs in the frequency range $\Omega\in[2.1,2.7]\,\mathrm{rad/s}$ via SSM reduction at various orders, as shown in Fig.~\ref{fig:Timoshenko_beam_FRC_orders}. We observe that the FRC converges well at $\mathcal{O}(9)$ expansion. Remarkably, the peak vibration amplitude of the FRC reaches 415 mm, which is more than 10 times of the thickness of the beam and more than one third of the length of the beam. It is not surprising that we need a high-order expansion of SSM to capture such a large deformation. In contrast, for smaller excitation amplitude $M=0.24\,\mathrm{N\cdot m}$, we found that the peak response amplitude of the FRC over the same frequency interval is reduced to 283 mm and $\mathcal{O}(5)$ expansion of SSM is already able to produce converged FRC, as seen in Fig.~\ref{fig:Timoshenko_beam_FRC_orders}.}

\begin{figure}[!ht]
\centering
\includegraphics[width=0.45\textwidth]{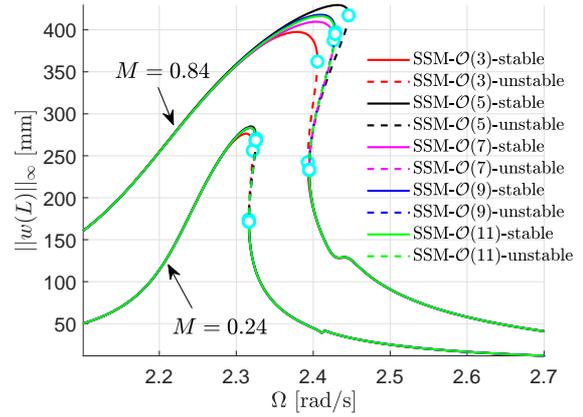}
\caption{\textcolor{black}{The FRCs in the amplitude of deflection at the free end for Timoshenko cantilever beam carrying a lumped mass. Two families of FRCs corresponding to moment amplitude $M=0.84\,\mathrm{N\cdot m}$ and $M=0.24\,\mathrm{N\cdot m}$ are obtained using SSM computations at different orders.}}
\label{fig:Timoshenko_beam_FRC_orders}
\end{figure}

\textcolor{black}{Similarly to the previous example, we also use the collocation method with the \texttt{po} toolbox of \textsc{coco} and the harmonic balance method with \textsc{nlvib} tool to extract the FRC of the full system to validate the results obtained from SSM reduction. The setting of algorithm parameters of \textsc{coco} are given in Appendix~\ref{sec:apdix-coco}. As seen in Fig.~\ref{fig:FRC-Timoshenko}, the results from SSM reduction match closely with the reference solution from \texttt{po} (labelled as Collocation). In the large amplitude response of $M=0.84\,\mathrm{N\cdot m}$, we notice small discrepancies relative to the full solution. As these discrepancies are not observed for the lower excitation amplitude of $M=0.24\,\mathrm{N\cdot m}$, they may be attributed to our restriction of the non-autonomous part of the SSM to its leading-order approximation~\eqref{eq:leading-nonauto}.}

\begin{figure}[!ht]
\centering
\includegraphics[width=0.45\textwidth]{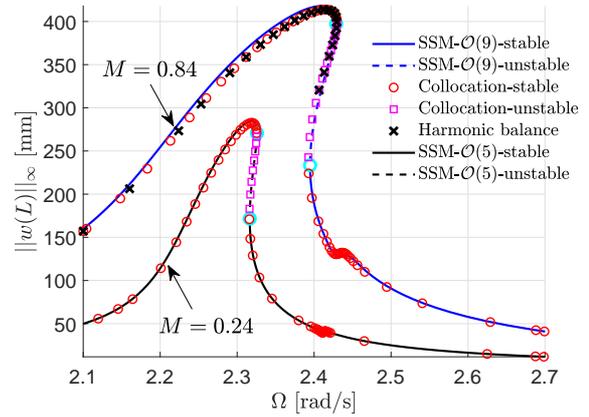}
\caption{\textcolor{black}{The FRCs in the amplitude of deflection at the free end for Timoshenko cantilever beam carrying a lumped mass subject to a harmonic moment $M\cos\Omega t$ at the free end. Here the continuation of haromnic balance method with \textsc{nlvib} for $M=0.84\,\mathrm{N\cdot m}$ terminates after its runtime reaches the one-day time-threshold.}}
\label{fig:FRC-Timoshenko}
\end{figure}

\textcolor{black}{All computations of this example are performed on a remote Intel Xeon E3-1585Lv5 processor (3.0-3.7 GHz) on the ETH Euler cluster. The computation times of FRC for $M=0.84$ using SSM reduction at $\mathcal{O}(9)$ and the collocation method with \textsc{coco} are 29 seconds and 3.8 hours, respectively. When we set the time-threshold of the harmonic balance method with \textsc{nlvib} to be one day, the number of harmonics to be 10 and the nominal step size to be 20, the continuation run in \textsc{nlvib} was not able to cover the full FRC (see Fig.~\ref{fig:FRC-Timoshenko}) because the adaptation of continuation step sizes in \textsc{nlvib} does not work well. Thus, the SSM reduction again produces a significant speed-up relative to the other two methods applied to the full system.}

\textcolor{black}{One can find a small bump at $\Omega\approx2.44$ in the FRC for $M=0.84$ shown in Fig.~\ref{fig:FRC-Timoshenko}. This bump results from the modal interaction between the first and the second bending modes. As seen in Fig.~\ref{fig:Timoshenko_beam_FRC_modal}, two peaks are observed in the FRC for the second mode $\rho_2$ because of the internal resonance. The second peak results in the small bump in the FRC shown in Fig.~\ref{fig:FRC-Timoshenko}.}

\begin{figure}[!ht]
\centering
\includegraphics[width=0.45\textwidth]{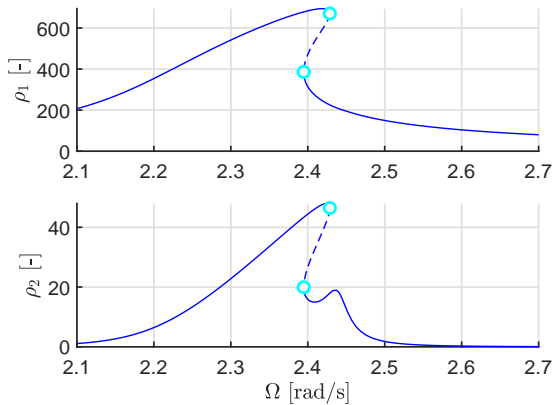}
\caption{\textcolor{black}{FRC in $(\rho_1,\rho_2)$ of the Timoshenko cantilever beam carrying a lumped mass subject to a harmonic moment $M\cos\Omega t$ at the free end with $M=0.84$. The corresponding FRC in the deflection at the free end is presented in Fig.~\ref{fig:FRC-Timoshenko}. The second peak near $\Omega=2.44$ for the FRC in $\rho_2$ explains the small bump for $\Omega\approx2.44$ in the FRC shown in Fig.~\ref{fig:FRC-Timoshenko}}}
\label{fig:Timoshenko_beam_FRC_modal}
\end{figure}

\subsection{A simply supported von K\'arm\'an plate}
\label{sec:vonKarmanPlate}
\begin{sloppypar}
We now consider a two-dimensional structure to demonstrate the effectiveness of SSM reduction in the case of high-dimensional systems. Specifically, we study the forced vibration of a simply supported plate (see the first panel of Fig.~\ref{fig:plate_mesh}). Let the length, width and thickness of this plate be $a$, $b$ and $h$, it follows from {classical linear} plate theory that its natural frequency is given by~\cite{geradin2014mechanical}
\begin{equation}
\omega_{(i,j)}=\left(\frac{i^2}{a^2}+\frac{j^2}{b^2}\right)\pi^2\sqrt{\frac{D}{\rho h}},
\end{equation}
where $i,j$ are positive integers, $\rho$ and $D$ are the density and bending stiffness of the plate, respectively. $D$ is given as follows
\begin{equation}
D = \frac{Eh^3}{12(1-\nu^2)},
\end{equation}
where $E$ and $\nu$ are Young's modulus and Poisson's ratio, respectively. In the case of {square} plate, we have $a=b=l$ and
\begin{equation}
\label{eq:exactFreqPlate}
\omega_{(1,2)}=\omega_{(2,1)}=\frac{5\pi^2}{l^2}\sqrt{\frac{D}{\rho h}}.
\end{equation}
We conclude that there exists 1:1 internal resonance between the second and third bending modes of the simply supported square plate.
\end{sloppypar}

\begin{figure}[!ht]
\centering
\includegraphics[width=0.4\textwidth]{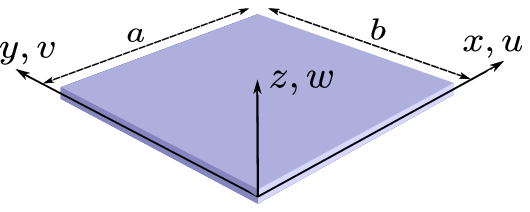}
\includegraphics[width=0.4\textwidth]{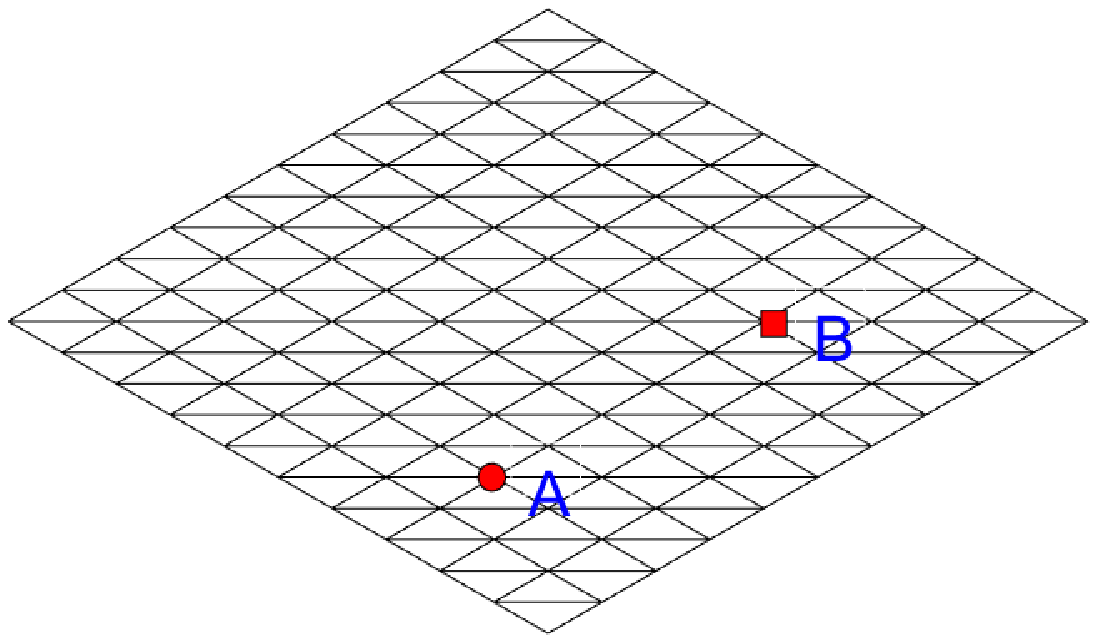}
\caption{A simply supported rectangular plate and a mesh for a square plate ($a=b=l$).}
\label{fig:plate_mesh}
\end{figure}

The square plate studied here is modeled using the von K\'arm\'an theory\textcolor{black}{, where both the \emph{in-plane} displacements $(u,v)$ and the \emph{out-of-plane} displacement $w$ are modeled as unknowns in the governing equations and the nonlinear strain due to large transverse deformation is considered.} The reader may refer to~\cite{reddy2015introduction} for the nonlinear governing equation of the plate. 

We apply the finite element method to discretize the governing equation. Triangular elements are used to perform such a discretization following the paradigm presented in the second panel of Fig.~\ref{fig:plate_mesh}. With the length of the plate uniformly divided into $n_{\mathrm{p}}$ subintervals, the number of elements and the number of DOF of the discretized plate are given by
\begin{equation}
N_{\mathrm{e}}=2n_{\mathrm{p}}^2,\quad n =6(n_{\mathrm{p}}^2+1)=3N_{\mathrm{e}}+6.
\end{equation}
For the mesh in Fig.~\ref{fig:plate_mesh}, we have $n_{\mathrm{p}}=10$, $N_{\mathrm{e}}=200$ and $n =606$. We use flat facet shell finite elements to discretize the displacement field~\cite{allman1976simple,allman1996implementation}. This is a plate element but can be used to model shell structures with small curvature. Each node in the element has six DOF, namely, $(u,v,w,w_x,w_y,u_y-v_x)$.
The reader may refer to~\cite{allman1976simple,allman1996implementation,FEcode} for the derivation of the mass and stiffness matrcies, and the coefficients of nonlinear internal forces. We also use Rayleigh damping in this example (cf.~\eqref{eq:RayleighDamping}).

\begin{sloppypar}
In following computations, we set $l=1\,\textrm{m}$, $h=0.01\,\textrm{m}$, $E=70\times 10^9$ Pa, $\nu=0.33$ and $\rho=2700\,\mathrm{kg}/\textrm{m}^3$. With the mesh in Fig.~\ref{fig:plate_mesh}, the natural frequencies of the discrete undamped linear plate are computed and compared with the analytical solutions to validate the correctness of $\boldsymbol{M}$ and $\boldsymbol{K}$ of the finite element model. It follows from~\eqref{eq:exactFreqPlate} that
\begin{equation}
\omega_{(1,2)}=\omega_{(2,1)}=768.4\,\mathrm{rad/s}.
\end{equation}
Meanwhile, the computation of natural frequencies using the finite element model gives
\begin{equation}
\omega_2 = 763.6\,\mathrm{rad/s},\quad\omega_3=767.7\,\mathrm{rad/s},
\end{equation} 
which are close to the reference solutions: their relative errors are $0.62\%$ and $0.09\%$, respectively. The vibration modal shapes of these two modes are plotted in Fig.~\ref{fig:plate_mode}. Given the mesh breaks symmetry between the two modes, the obtained $\omega_2$ is not exactly as $\omega_3$. This discrepancy will become smaller when the mesh size decreases. \textcolor{black}{Once again, we choose Rayleigh damping (see eq.~\eqref{eq:RayleighDamping}) with $\alpha=1$ and $\beta=4\times10^{-6}$ such that the eigenvalues of the damped linear plate are approximated} according to~\eqref{eq:weakDampFreq} and we have
\begin{equation}
\lambda_3=-1.7 +\mathrm{i}763.6\approx\mathrm{i}\omega_2, \lambda_5=-1.7 +\mathrm{i}767.7\approx\mathrm{i}\omega_3.
\end{equation}
We also considered a static nonlinear problem to further validate the correctness of nonlinear force $\boldsymbol{N}(\boldsymbol{x})$ of the finite element model. Specifically, we have studied Example 7.9.3 in~\cite{reddy2015introduction} using our finite element model. In the example, the transverse displacement $w$ of a simply supported square plate under uniformly distributed transverse load is calculated. We have solved the same problem and our results match well with the reference results in~\cite{reddy2015introduction}.
\end{sloppypar}

\begin{figure}[!ht]
\centering
\includegraphics[width=0.45\textwidth]{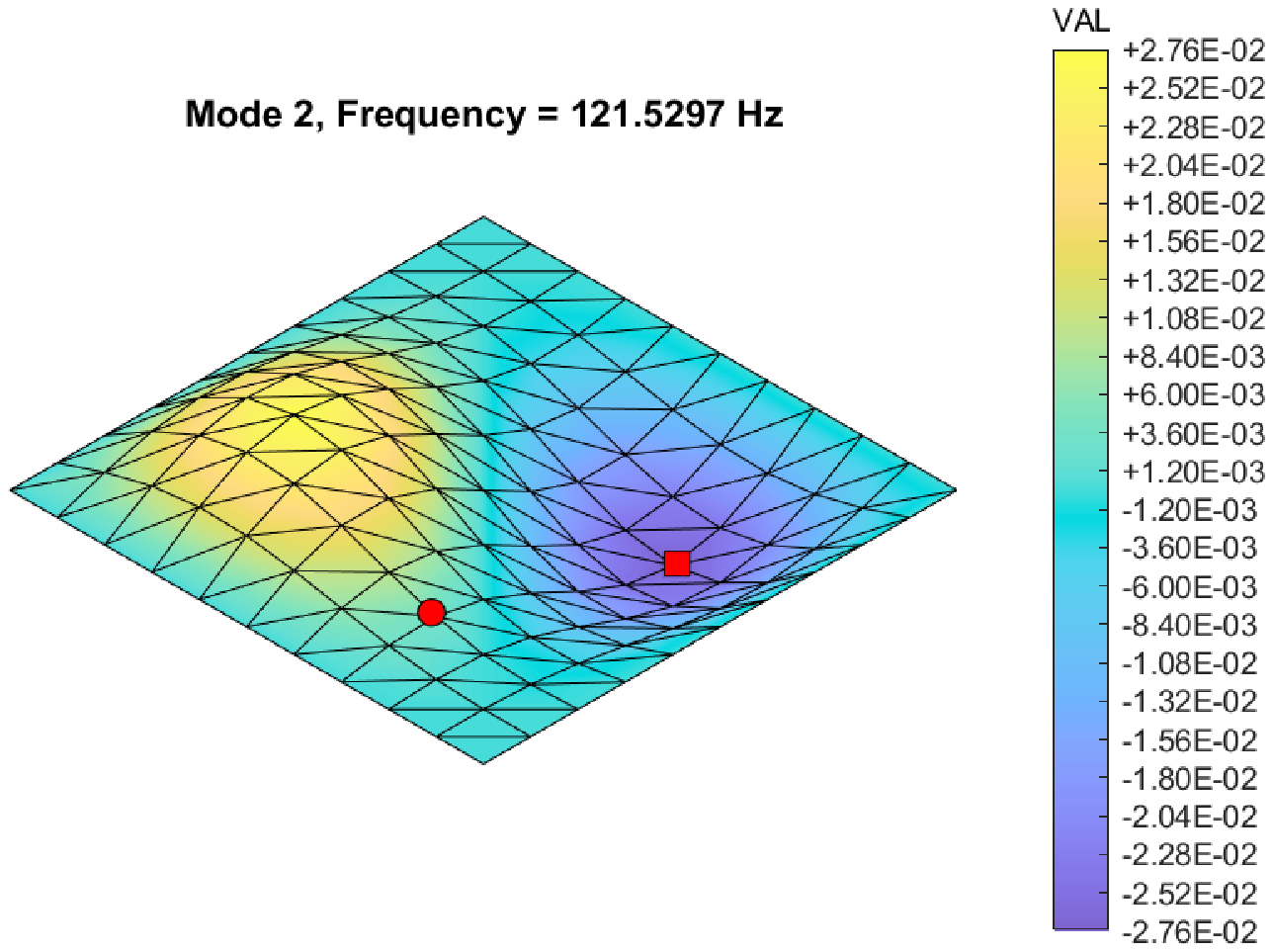}
\includegraphics[width=0.45\textwidth]{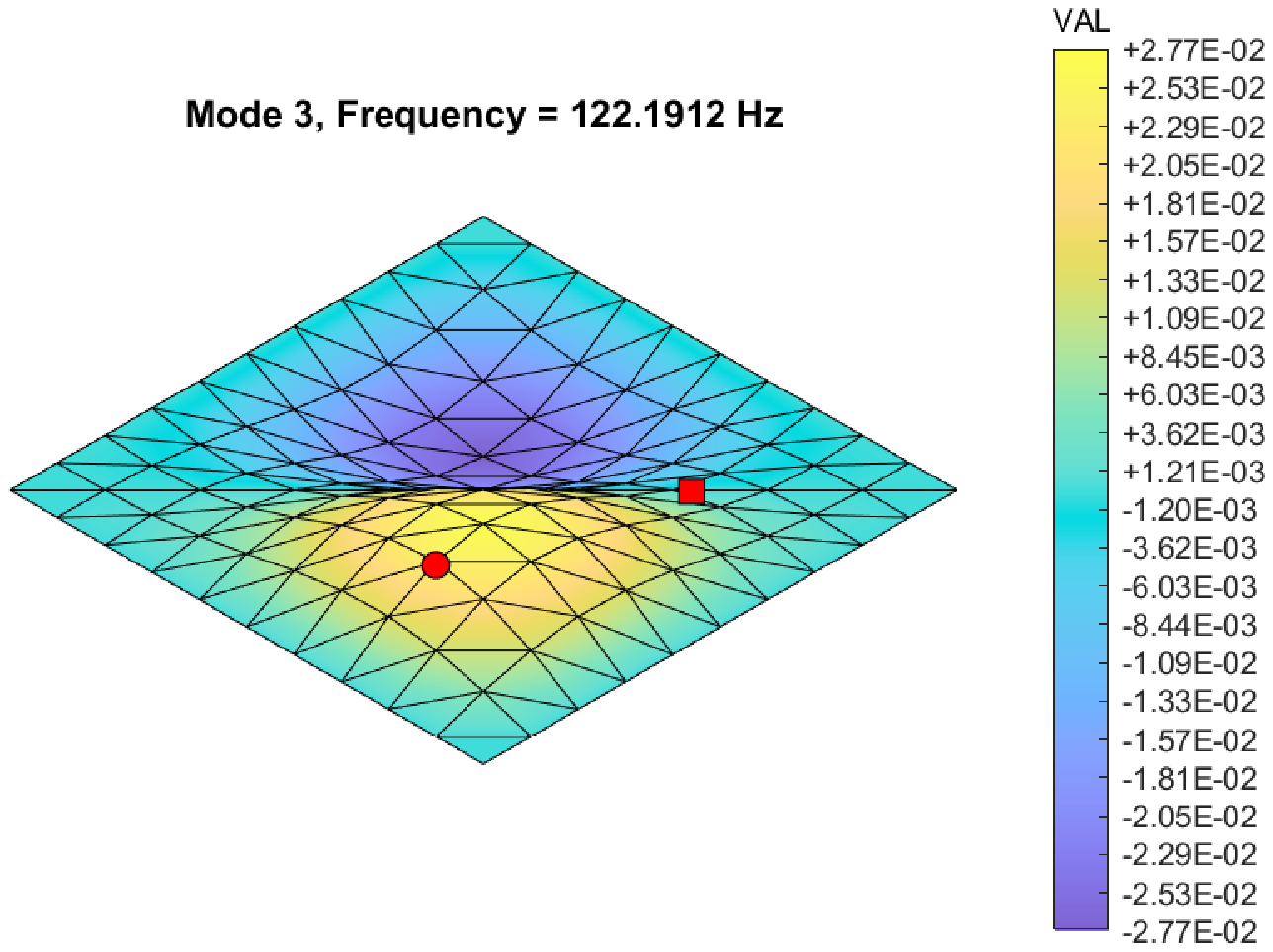}
\caption{Mode shapes of the second and third linear bending modes of the simply supported square plate.}
\label{fig:plate_mode}
\end{figure}

We seek to determine the FRC of this plate subject to a concentric load $50\cos\Omega t$ at point A with coordinate $(0.2l,0.3l)$ (cf.~Fig.~\ref{fig:plate_mesh}). It follows from the mode shapes of the plate (see Fig.~\ref{fig:plate_mode}) that point A is close to the nodal line of the second mode and then the modal force for the third mode is larger than that of the second mode. Here we choose the two pairs of complex conjugate modes corresponding to the second and third bending modes as the master spectral subspace to account for the 1:1 internal resonance. We again use {polar} coordinate representation because both modes are activated. The computation of the FRC in this example was performed on a remote node on the ETH Euler cluster with two Intel Xeon Gold 6150 processors (2.7-3.7 GHz).

\textcolor{black}{As seen in Fig.~\ref{fig:plate_FRC_A_orders}, the FRC under the concentrated load of $50\cos\Omega t$ converges well at $\mathcal{O}(5)$ expansion of the SSM. The peak vibration amplitude of the FRC is 2.4 mm. We generally observe that higher-order expansions of the SSM are required to accurately approximate larger response amplitudes. Indeed, when the load amplitude is doubled, numerical experiments show that the peak amplitude reaches 5.1 mm at the coordinate $(0.3l,0.3l)$, and an $\mathcal{O}(11)$-expansion is needed to yield a converged FRC.  Furthermore, upon tripling the forcing amplitude, we observe that the FRC obtained by SSM reduction does not converge. This observation is in agreement with the SSM theory which is applicable for limited forcing amplitudes. In the rest of this example, we study the original forcing case of $50\cos\Omega t$ using an $\mathcal{O}(5)$ SSM reduction.}

\begin{figure}[!ht]
\centering
\includegraphics[width=0.45\textwidth]{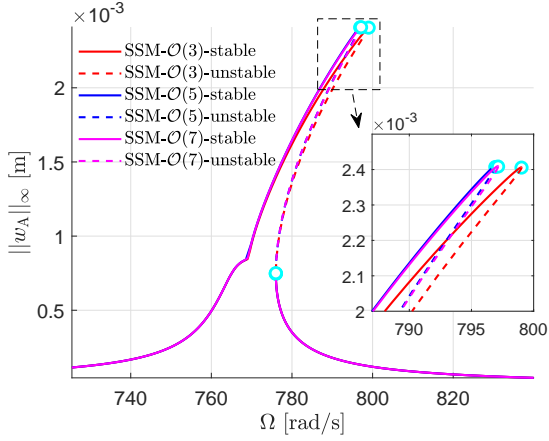}
\caption{\textcolor{black}{The FRCs in the amplitude of deflection at point A for von K\'arm\'an plate discretized with 200 elements and 606 DOF. These FRCs are obtained using SSM computations at different orders.}}
\label{fig:plate_FRC_A_orders}
\end{figure}

The FRC obtained by SSM reduction is plotted in Fig.~\ref{fig:plate_FRC_physical}, where the upper and lower panels present the amplitudes of transverse vibration at node A and B, respectively (cf.~Fig.~\ref{fig:plate_mesh}). 
To validate the effectiveness of SSM reduction, one may apply the collocation method or harmonic balance technique to the full system as we did in the previous example. However, these two methods are impractical due to the high dimensionality of the problem. For the same mechanical system with 606 DOF, trial \textcolor{black}{computational} experiments show that the harmonic balance method with \textsc{nlvib} performed only one continuation step and the collocation method with \textsc{coco} performed only four continuation steps in ten days of computational time. We consider an alternative method, namely, the shooting method combined with parameter continuation (cf.~\cite{peeters2009nonlinearII}), to extract the FRC of the full nonlinear system. In particular, the computation was performed using a \textsc{coco}-based shooting toolbox~\cite{coco-shoot} with the Newmark integrator and the atlas algorithm of \textsc{coco}. With 1,000 integration steps per excitation period and a maximum continuation step size $h_{\max}=50$, we obtain the FRC of full system. As can be seen in the figure, the results of the two techniques match closely.

\begin{figure}[!ht]
\centering
\includegraphics[width=3.0in]{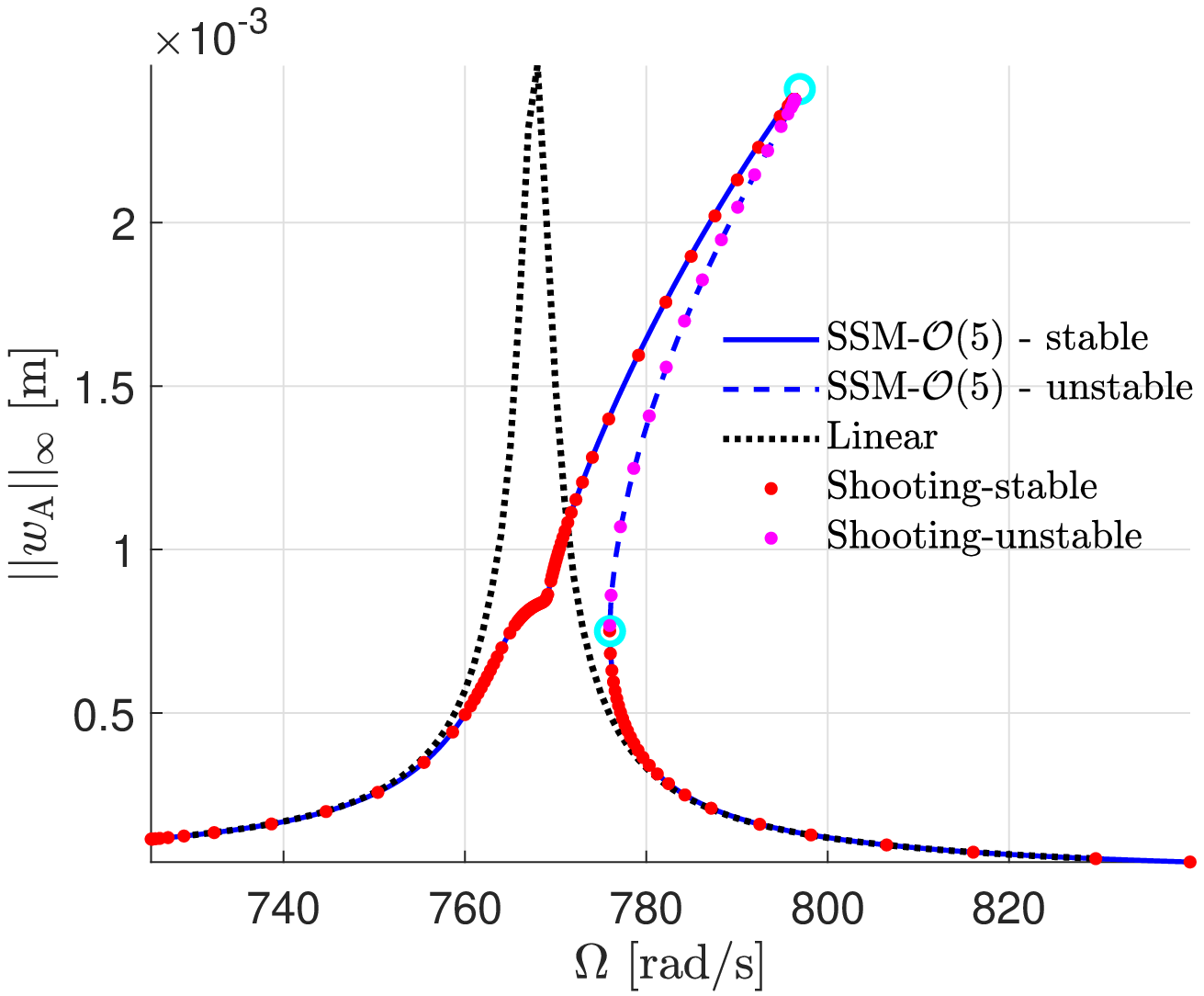}
\includegraphics[width=3.0in]{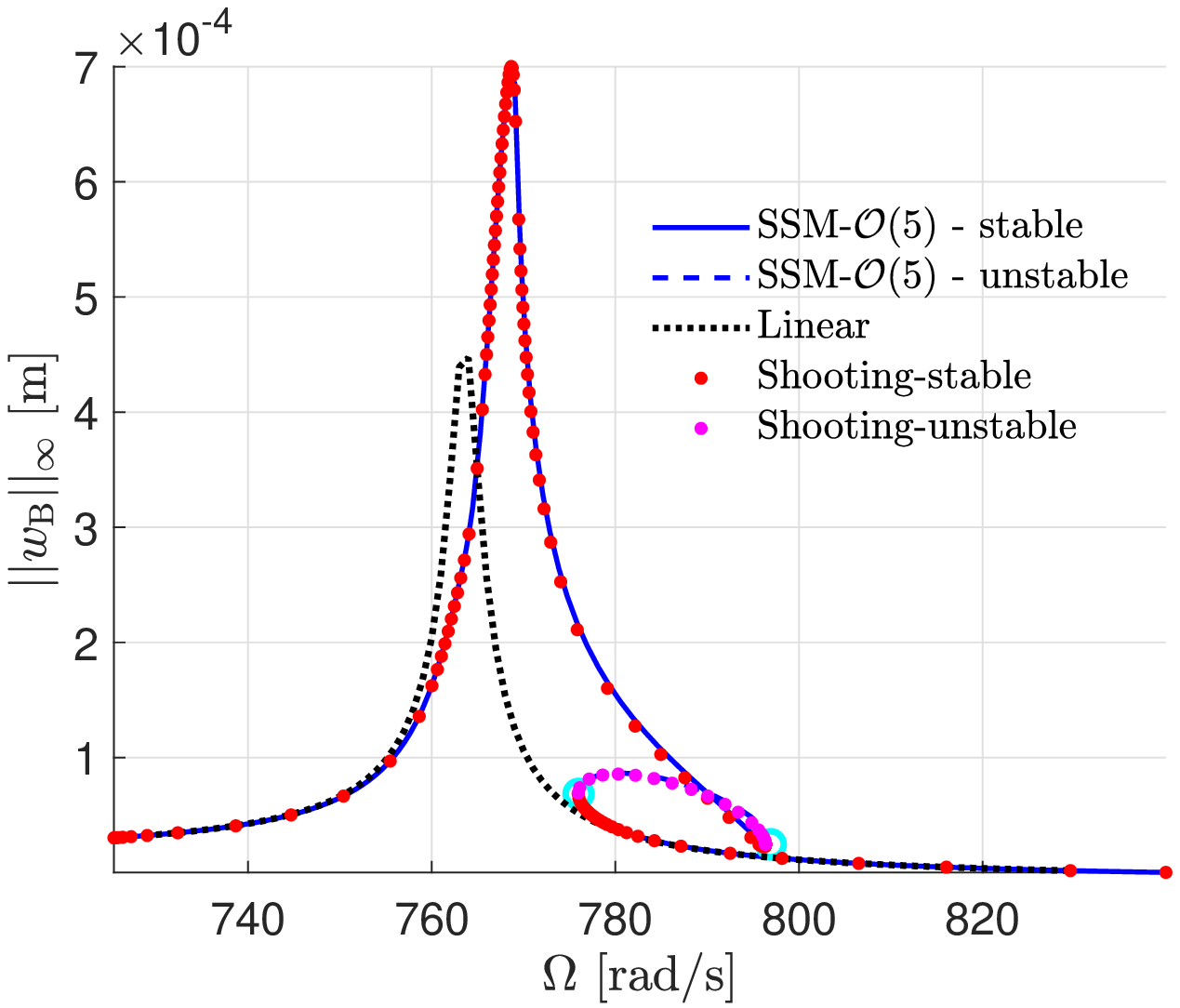}
\caption{FRC in physical coordinates for von K\'arm\'an plate discretized with 200 elements and 606 DOF. The upper and lower panels give the amplitude of deflection at point A and B respectively. Here the black dotted lines are results of linear analysis. The red and magenta dots are results of shooting-based continuation of the full nonlinear system.}
\label{fig:plate_FRC_physical}
\end{figure}

We also present the results of the linear analysis in Fig.~\ref{fig:plate_FRC_physical} to demonstrate the essential nature of geometric nonlinearity.
In the linear analysis, we ignore the nonlinear force and solve the corresponding FRC analytically in the frequency domain. Specifically, the linear equation of motion can be written in the following form
\begin{equation}
\boldsymbol{M}\ddot{\boldsymbol{x}}+\boldsymbol{C}\dot{\boldsymbol{x}}+\boldsymbol{K}\boldsymbol{x}=\epsilon\boldsymbol{f}\cos\Omega t=\epsilon\mathrm{Re}(e^{\mathrm{i}\Omega t})\boldsymbol{f}.
\end{equation}
Letting $\boldsymbol{x}(t)=\mathrm{Re}(\hat{\boldsymbol{x}}e^{\mathrm{i}\Omega t})$ gives
\begin{equation}
\left(-\Omega^2\boldsymbol{M}+\mathrm{i}\Omega\boldsymbol{C}+\boldsymbol{K}\right)\hat{\boldsymbol{x}}=\epsilon\boldsymbol{f}
\end{equation}
and hence
\begin{equation}
\boldsymbol{x}(t)=\mathrm{Re}\left(\left(-\Omega^2\boldsymbol{M}+\mathrm{i}\Omega\boldsymbol{C}+\boldsymbol{K}\right)^{-1}\epsilon\boldsymbol{f}e^{\mathrm{i}\Omega t}\right).
\end{equation}
The results by linear analysis match well with the ones of SSM reduction when the response amplitude is small or the excitation frequency $\Omega$ is far away from the natural frequency $\omega_2$. The linear results significantly deviate from the results of SSM reduction when the response amplitude is large. In these cases, the deformation is large and the effects of geometrical nonlinearity are significant.

Energy transfer between modes due to internal resonance is also observed in this example. As can be seen in Fig.~\ref{fig:plate_FRC_physical}, when the vibration amplitude at node $B$ arrives at a maximum at $\Omega\approx\omega_2$, a notch is observed in the FRC of node $A$ at the same excitation frequency. Indeed, similar phenomenon is observed in the FRC of $(\rho_1,\rho_2)$, as shown in Fig.~\ref{fig:plate_FRC_modal}. Note that the normal coordinates $\rho_1$ and $\rho_2$ of reduced dynamics represent the responses of the second and third bending modes, respectively. Interestingly, the hardening of the third bending mode $\rho_2$ results in the self-crossing of the FRC of the second bending mode $\rho_1$. One may note the similarity between the FRC of $||w_\mathrm{A}||_\infty$ and the one of $\rho_2$, and the similarity between the FRC of $||w_\mathrm{B}||_\infty$ and the one of $\rho_1$. Such similarities can be explained by the fact that node A and node B are (nearly) located at the peak response of the third and second bending modes, respectively (also at the nodal lines of the second and third mode respectively, cf.~Figs.~\ref{fig:plate_mesh}-\ref{fig:plate_mode}).

\begin{figure}[!ht]
\centering
\includegraphics[width=0.45\textwidth]{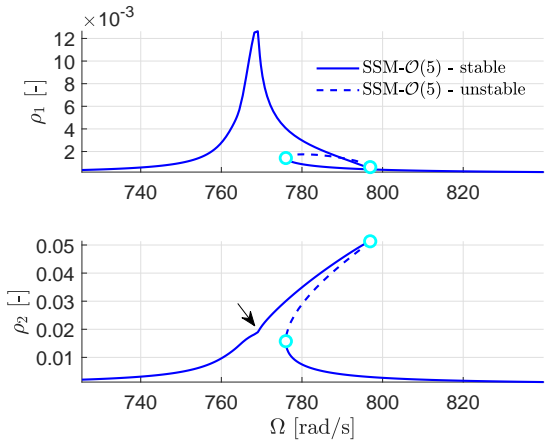}
\caption{FRC in normal coordinates for von K\'arm\'an plate discretized with 200 elements and 606 DOF. Mode interactions are observed. Specifically, when $\rho_1$ arrives its peak, a notch is observed in the FRC of $\rho_2$. In addition, an unstable branch is observed in $\rho_1$ as well. Such unstable solutions will be missing if we only include this mode in the SSM analysis.}
\label{fig:plate_FRC_modal}
\end{figure}

SSM reduction displays a significant speed-up gain relative to the shooting method in the above computations. Specifically, the computational time for SSM reduction is about one minute while the one for shooting method is about 6 days. In order to further demonstrate the speed-up gain relative to the collocation method and the harmonic balance method, we consider a discrete plate with $n_{\mathrm{p}}=5$, $N_{\mathrm{e}}=50$, resulting in 156 DOF. In this case, the point A with coordinate $(0.2l,0.3l)$ is not at any node of the finite element discretization. We take the neighbor node with coordinates $(0.2l,0.4l)$ as the location of the imposed harmonic excitation. This node is also referred to as point A where the load is applied. The FRC obtained using SSM reduction, and three methods applied to the full system (harmonic balance, collocation, and shooting) are plotted in Fig.~\ref{fig:plate_FRC_physical_5}, which again validates the accuracy of SSM reduction. In addition, the computational times for SSM reduction, the collocation method, and the shooting method are 49 seconds, five days, and 17 hours, respectively. With nominal step size 10, the continuation with the harmonic balance method terminates after seven continuation steps due to the failure of convergence. Such a continuation run took about 38.7 hours. 

\begin{figure}[!ht]
\centering
\includegraphics[width=0.45\textwidth]{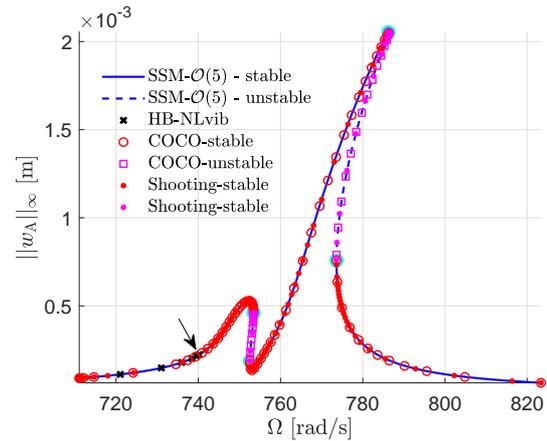}
\caption{FRC in physical coordinates for von K\'arm\'an plate discretized with 50 elements 156 DOF. Here the continuation of harmonic balance method terminates around $\Omega\approx740$ (see the arrow) after seven successful continuation steps.}
\label{fig:plate_FRC_physical_5}
\end{figure}

\begin{sloppypar}
We further perform SSM reduction to the plate discretized under an increasing number of elements to further demonstrate the remarkable computational efficiency of the reduction method. With $n_{\mathrm{p}}$=20, 40, 100, 200, the corresponding number of elements is $N_{\mathrm{e}}$=800, 3,200, 20,000, 80,000, and the number of DOF is $n$=2,406, 9,606, 60,006, 240,006, yielding very high-dimensional systems. The computational times for calculating FRC of these discrete finite element models have been presented in Fig.~\ref{fig:plate_SSMtime}. When the number of DOF is 240,006, the computational time for SSM analysis is about 19 hours. Among the 19 hours, nearly 8 hours are used for the computation of the autonomous part of the SSM, and nearly 11 hours are used for the computation of the non-autonomous part of the SSM (cf.~\eqref{eq:expphi-}) at 337 sampled excitation frequencies. In other words, each computation of a non-autonomous SSM takes about 2 minutes. By contrast, the continuation of fixed points in reduced dynamics only took 20 seconds. One may significantly reduce the computation time for non-autonomous SSM by parallel computing, or ignoring the non-autonomous part of the SSM for small forcing amplitudes, as we discussed in section~\ref{sec:compLoad}.
\end{sloppypar}

\begin{figure}[!ht]
\centering
\includegraphics[width=0.47\textwidth]{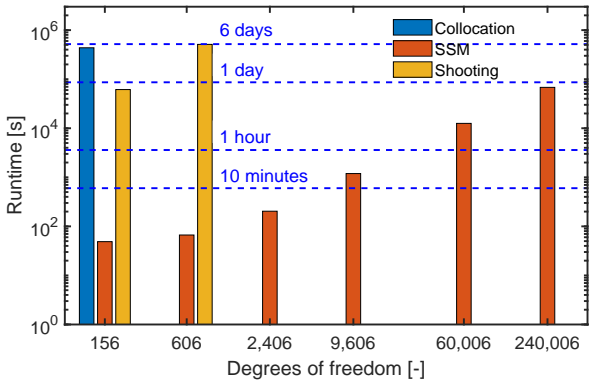}
\caption{Computational times of the FRC of the von K\'arm\'an plate discretized discretized with different number of DOF. The number of DOF is given by $3N_{\mathrm{e}}+6$ when the plate is discretized with $N_{\mathrm{e}}$ elements. Here we have $N_{\mathrm{e}}\in$\{50, 200, 800, 3,200, 20,000, 80,000\}.}
\label{fig:plate_SSMtime}
\end{figure}

We conclude this example by having a close look into the 8 hours spent on the calculation of the autonomous part of the SSM. Specifically, we are interested in how the 8 hours are distributed into the times spent on the computation of the SSM at each order. As can be seen in Fig.~\ref{fig:plate_SSMtime_orders}, the computational time increases nearly exponentially with the increment of the orders, and more than 6 hours among the 8 hours are used in the computation of the fifth order SSM. In addition, Fig.~\ref{fig:plate_SSMtime_orders} shows that the memory cost also increases significantly with the increment of orders. This shows that distributed memory needs to be utilized in the computation of SSM at higher orders for such high degree of freedom.

\begin{figure}[!ht]
\centering
\includegraphics[width=0.4\textwidth]{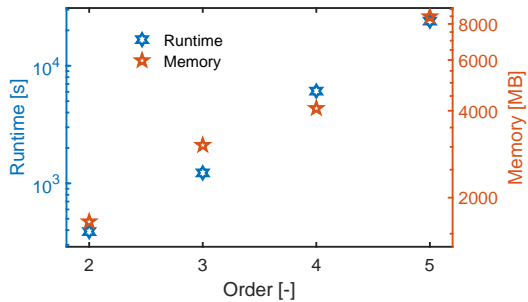}
\caption{Runtime and memory used in the computation at each order of the \textcolor{black}{autonomous} SSM of the von K\'arm\'an plate discretized with 240,006 DOF.}
\label{fig:plate_SSMtime_orders}
\end{figure}

\textcolor{black}{
\subsection{A shallow shell structure}}
\textcolor{black}{This example is adapted from the shallow-arc example of~\cite{SHOBHIT}. We consider a finite element model of a geometrically nonlinear shallow shell structure, illustrated in Fig.~\ref{fig:shell_mesh}. The shell is simply supported at the two opposite edges aligned along the $y-$axis in Fig.~\ref{fig:shell_mesh}.}

\begin{figure}[!ht]
\centering
\includegraphics[width=0.4\textwidth]{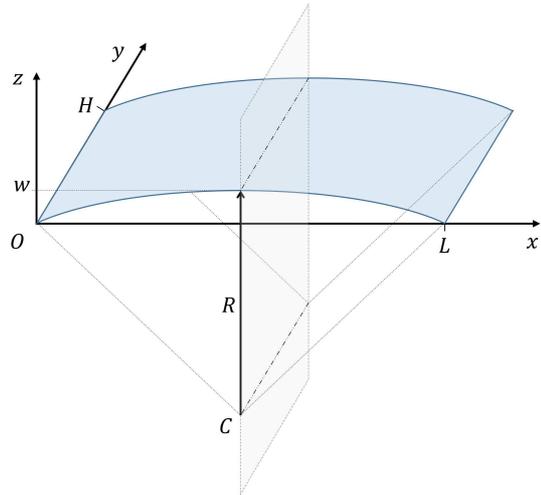}
\caption{\textcolor{black}{The schematic of a shallow shell structure~\cite{SHOBHIT}.}}
\label{fig:shell_mesh}
\end{figure}

\textcolor{black}{
Let $L$, $H$ and $t$ be the length, width and thickness of the shell, $w$ be a curvature parameter (defined as the height of the midpoint relative to the end, cf.~Fig.~\ref{fig:shell_mesh}). We set $L=2\,\textrm{m}$, $H=1\,\textrm{m}$, $t=0.01\,\textrm{m}$ and $w=0.041\,\textrm{m}$. Material properties are specified with the density $\rho=2700\,\mathrm{kg}/\textrm{m}^3$, Young's modulus $E=70\times 10^9$ Pa and Poisson's ratio $\nu=0.0.33$. Note that we have chosen a different value of $w$ compared to~\cite{SHOBHIT}, where $w=0.1\,\mathrm{m}$. Here we set $w=0.041\,\mathrm{m}$ because numerical experiments show that this choice induces a 1:2 internal resonance between the first two modes.}

\textcolor{black}{
Similarly to the previous plate example, this model is discretized using flat, triangular shell elements and each node in the elements has six DOF. The discrete model here contains 400 elements (cf.~Fig.~11(b) in~\cite{SHOBHIT} for the schematic of the mesh of the discrete model), resulting in $n=1320$ DOF. Again, the matrices $\boldsymbol{M}$ and $\boldsymbol{K}$ and the coefficients of nonlinear terms are provided by the open-source finite element code~\cite{FEcode}. Here, we choose $\alpha$ and $\beta$ in the Rayleigh damping~\eqref{eq:RayleighDamping} such that the damping ratios of the first two modes are equal to 0.002.}

\textcolor{black}{The eigenvalues of the first two pairs of modes of the discrete model are given by
\begin{equation}
    \lambda_{1,2}=-0.30\pm\mathrm{i}149.22,\quad\lambda_{3,4}=-0.60\pm\mathrm{i}298.78.
\end{equation}
Therefore, the system indeed has a near 1:2 internal resonance between the first two pairs of modes. Next, we apply a concentrated load $10\cos\Omega t\,\mathrm{N}$ in $z-$direction at the mesh node located at $(x,y)=(0.25L,0.5H)$. We are concerned with the forced response curve in terms of the $z$-displacement of the node for $\Omega\in[0.92\mathrm{Im}(\lambda_1),1.07\mathrm{Im}(\lambda_1)]$.}

\textcolor{black}{The FRCs obtained by SSM reduction computations at different orders are presented in Fig.~\ref{fig:shell_FRC_orders}. We observed that the FRC converges well at $\mathcal{O}(5)$ expansion of the SSM. The computation time of the FRC by SSM reduction at $\mathcal{O}(5)$ is about two minutes.}

\begin{figure}[!ht]
\centering
\includegraphics[width=0.45\textwidth]{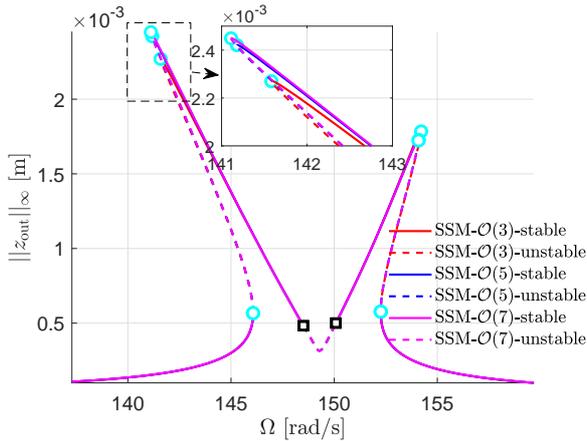}
\caption{\textcolor{black}{The FRC in the amplitude of $z$-displacement at the mesh node $(x,y)=(0.5L,0.5H)$ of the shallow shell structure discretized with 400 elements and 1320 DOF. These FRCs are obtained using SSM computations at different orders.}}
\label{fig:shell_FRC_orders}
\end{figure}

\textcolor{black}{Once again, we apply the \textsc{coco}-based shooting toolbox~\cite{coco-shoot} to extract the FRC of the full nonlinear system and compare with the results obtained from SSM reduction. In particular, the Newmark algorithm is used to perform numerical integration during shooting. Unlike the previous example, we need to adopt a smaller number of integration steps per excitation period in this model because the number of DOF here is nearly doubled and at the same time,  the FRC has a more complex shape (cf.~Figs.~\ref{fig:plate_FRC_physical} and~\ref{fig:shell_FRC_orders}). When we set 100 integration steps per excitation period and the time threshold of shooting-based continuation run to be 180 hours (7.5 days), the shooting-based continuation run was not able to cover the full FRC (see the end point of the red lines near $\Omega=155$ in Fig.~\ref{fig:FRC_shell_comparison}). The FRC obtained by the shooting method with 100 integration steps per excitation period matches well with the one from SSM reduction overall. However, small discrepancies were observed. These discrepancies are resulted from the low accuracy of numerical integration. Indeed, when the number of integration steps per excitation period is increased to 200, another continuation run for $\Omega\in[145,152]$ was performed and the discrepancies in the full solution are reduced significantly, as seen in Fig.~\ref{fig:FRC_shell_comparison}. We observe that even in the restricted frequency range of $\Omega\in[145,152]$, the time taken by the new continuation run is already near five days. Hence, we conclude that the results from SSM reduction provide good accuracy, and remarkably, can be obtained in just about two minutes.}

\begin{figure}[!ht]
\centering
\includegraphics[width=0.45\textwidth]{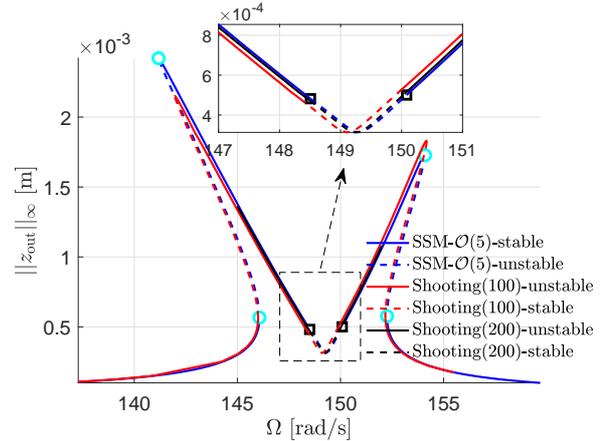}
\caption{\textcolor{black}{The FRC in the amplitude of $z$-displacement at the mesh node $(x,y)=(0.5L,0.5H)$ of the shallow shell structure discretized with 400 elements and 1320 DOF. Here the red and black lines are results of shooting-based continuation of the full nonlinear system with 100 and 200 integration steps per excitation period.}}
\label{fig:FRC_shell_comparison}
\end{figure}

\section{Conclusion}
\begin{sloppypar}
We have derived reduced-order models for harmonically excited mechanical systems with internal resonance. The phase space of a high-dimensional full system is reduced to a low-dimensional time-periodic spectral submanifold (SSM) tangent to resonant spectral subbundles of periodic orbits born out of the origin under periodic forcing. We have used the reduced-order model to extract forced response curves (FRCs) of periodic orbits of the full system around internally resonant modes. Specifically, in normal form coordinates for the reduced dynamics, time-dependent harmonic terms are all canceled, yielding slow-phase reduced dynamics, whose fixed points correspond to periodic orbits on the time-varying SSM. We have used parameter continuation to construct the FRCs as solution branches of fixed points. Such a solution branch is then mapped back to physical coordinates to obtain the forced response curve of the full system.
\end{sloppypar}

\begin{sloppypar}
We have demonstrated the accuracy and efficiency of the SSM-based reduction method using \textcolor{black}{seven} examples. In the first example, a chain of oscillators with 1:1:1 internal resonance was studied to show that the SSM analysis can be applied to systems with several resonant modes. In the second example, a hinged-clamped beam with 1:3 internal resonance was investigated to illustrate the advantage of SSM reduction over the method of multiple scales. \textcolor{black}{In the third example, an axially moving beam with 1:3 internal resonance was explored to demonstrate the effectiveness of SSM reduction for systems with gyroscopic and nonlinear damping forces.} 

We further considered \textcolor{black}{four} examples of the finite element models of beams, plates \textcolor{black}{and shell-based three-dimensional structures} to demonstrate the remarkable computational efficiency of the SSM reduction in obtaining FRCs. Specifically, the FRC over a given frequency span of a von K\'arm\'an beam discretized with various number of degrees of freedom (DOF), ranging from 22 to 29,998, has been calculated using the reduction and other methods whenever the latter methods were applicable. In the case of 118 DOF, the computational times for the extraction of FRC using SSM reduction, the harmonic balance method and the collocation method are 14 seconds, 12.5 hours and 58.5 hours, respectively. For the beam discretized with 29,998 DOF, SSM reduction only takes approximately 1 hour to obtain the FRC. Such a remarkable computational efficiency of the reduction is also observed in the \textcolor{black}{Timoshenko beam,} von K\'arm\'an plate \textcolor{black}{and shallow shell structures}. \textcolor{black}{We have calculated the FRC of a cantilever Timoshenko beam that undergoes large deformations, where the computational time for the SSM reduction is just 29 seconds. At the same time, the collocation method takes 3.8 hours and the harmonic balance method is not able cover the FRC of the full system even in a full day. Further, the FRC of a 240,006 DOF von K\'arm\'an plate over a predefined frequency span (with 337 sampled frequencies) is obtained via SSM reduction in less than one day. Finally, the FRC of a 1,320 DOF shallow shell structure is obtained via a four-dimensional SSM in just two minutes while the continuation-based shooting method was unable to cover the full system's FRC even in a full week's computation time.}
\end{sloppypar}
We have used parameter continuation to locate equilibria of the slow-phase reduced dynamics. An intrinsic limitation of parameter continuation is the dependence of initial solution. Such a dependence makes it challenging to find isolated solution branches, or, \emph{isolas}. In the case of no internal resonance, the equilibria can be found as the intersection of two surfaces in a three-dimensional space~\cite{ponsioen2019analytic}. Such a level-set based technique, however, becomes impractical in general when the dimension of SSM is higher than two. The computation of isolas using parameter continuation could be possible with the help of singularity theory~\cite{cirillo2017analysis} or multidimensional continuation~\cite{dankowicz2020multidimensional}.

Another limitation of our current implementation it that it does not give a estimation of the upper bound of forcing amplitudes $\epsilon$ for which the reduction results are reliable. In~\cite{ponsioen2019analytic}, the domain of convergence has been used to estimate the upper bound of the reliable response amplitudes. The method is based on the computation of all zeros of a polynomial function~\cite{ponsioen2019analytic,christiansen2006truncated}. When internal resonance is accounted, this turns into locating the zeros of a \emph{set} of polynomial functions, which is not a trivial task. As an alternative, one may determine the radius of convergence of power series based on the coefficients of the series, i.e., some variants of Cauchy–Hadamard theorem.

In the continuation of equilibria in reduced-order models, we have observed both saddle-node and Hopf bifurcation points in numerical examples. In Part II, we will relate these bifurcations to the bifurcation of periodic orbits. Note that a unique limit cycle will bifurcate from a Hopf bifurcation equilibrium. Such a limit cycle corresponds to a two-dimensional torus in full system. In Part II, we will also study the computation and bifurcation of quasi-periodic orbits using SSM theory.

\section{Appendix}
\subsection{Derivation of the leading-order approximation to the reduced dynamics on a resonant SSM}
The derivation of leading-order approximation with multiple harmonics has been presented in~\cite{SHOBHIT}. Here we restrict attention to one harmonic and give a simple derivation to adapt for this study.
\label{sec:proof-thm-ssm-nonauto}
\begin{sloppypar}
Substituting the leading order approximation into~\eqref{eq:SSM-nonauto-eq}, and collecting the terms that are independent of $\boldsymbol{p}$, yield
\begin{equation}
\label{eq:leading-order-invariant}
\boldsymbol{B}\boldsymbol{W}_{\mathbf{I}}\boldsymbol{S}_{\boldsymbol{0}}(\phi)+\Omega\boldsymbol{B}D_{\phi}\boldsymbol{X}_{\boldsymbol{0}}(\phi)=\boldsymbol{A}\boldsymbol{X}_{\boldsymbol{0}}(\phi)+\boldsymbol{F}^{\mathrm{ext}}(\phi).
\end{equation}
Substituting the ansatz
\begin{gather}
\boldsymbol{X}_{\boldsymbol{0}}(\phi)=\boldsymbol{x}_{\boldsymbol{0}}e^{\mathrm{i}\phi}+\bar{\boldsymbol{x}}_{\boldsymbol{0}}e^{-\mathrm{i}\phi},\nonumber\\
\boldsymbol{S}_{\boldsymbol{0}}(\phi)=\boldsymbol{s}_{\boldsymbol{0}}^+e^{\mathrm{i}\phi}+{\boldsymbol{s}}_{\boldsymbol{0}}^-e^{-\mathrm{i}\phi},
\end{gather}
and~\eqref{eq:forcing-conj} into~\eqref{eq:leading-order-invariant} and collecting the coefficients of $e^{\mathrm{i}\phi}$ and $e^{-\mathrm{i}\phi}$, we obtain
\begin{gather}
(\boldsymbol{A}-\mathrm{i}\Omega\boldsymbol{B})\boldsymbol{x}_{\boldsymbol{0}}=\boldsymbol{B}\boldsymbol{W}_{\mathbf{I}}\boldsymbol{s}_{\boldsymbol{0}}^+-\boldsymbol{F}^\mathrm{a}\label{eq:expphi},\\
(\boldsymbol{A}+\mathrm{i}\Omega\boldsymbol{B})\bar{\boldsymbol{x}}_{\boldsymbol{0}}=\boldsymbol{B}\boldsymbol{W}_{\mathbf{I}}{\boldsymbol{s}}_{\boldsymbol{0}}^--{\boldsymbol{F}}^{\mathrm{a}}\label{eq:expNegphi}.
\end{gather}
If $(\boldsymbol{A}-\mathrm{i}\Omega\boldsymbol{B})$ is nonsingular, we can simply set $\boldsymbol{s}_{\boldsymbol{0}}^+=\boldsymbol{0}$ and directly solve the linear system~\eqref{eq:expphi} to obtain $\boldsymbol{x}_{\boldsymbol{0}}$. However, if there exist eigenvalues equal to $\mathrm{i}\Omega$, e.g., $\lambda_i^{\mathcal{E}}=\mathrm{i}\Omega$, the coefficient matrix is singular (see Proposition 2 in~\cite{SHOBHIT}). In that case, we must choose $\boldsymbol{s}_{\boldsymbol{0}}$ such that the right-hand side vector is in the range of $(\boldsymbol{A}-\mathrm{i}\Omega\boldsymbol{B})$. This can be done by imposing orthogonality constraint between the right-hand side vector and the kernel of $(\boldsymbol{A}-\mathrm{i}\Omega\boldsymbol{B})^\ast$. Since $\boldsymbol{u}_i^\mathcal{E}$ spans such a kernel for $\lambda_i^{\mathcal{E}}=\mathrm{i}\Omega$~\cite{SHOBHIT}, we have 
\begin{equation}
(\boldsymbol{u}_i^\mathcal{E})^\ast\boldsymbol{B}\boldsymbol{W}_{\mathbf{I}}\boldsymbol{s}_{\boldsymbol{0}}^+-(\boldsymbol{u}_i^\mathcal{E})^\ast\boldsymbol{F}^\mathrm{a}=0.
\end{equation}
Substituting~\eqref{eq:auto-ssm-first-order} into the above equation and utilizing the orthonormalization of the left and right eigenvectors (cf.~\eqref{eq:ortho-norm-eigenvectors}) gives ${{S}}_{\boldsymbol{0},i}=(\boldsymbol{u}_i^\mathcal{E})^\ast\boldsymbol{F}^\mathrm{a}$, where ${{S}}_{\boldsymbol{0},i}$ is defined in~\eqref{thm:ssm-nonauto}.

In practice, $\lambda_i^{\mathcal{E}}=\mathrm{i}\Omega$ does not hold for any $\Omega\in\mathbb{R}$ given we have assumed $\mathrm{Re}\lambda_i^{\mathcal{E}}<0$. However, we have $\lambda_i^{\mathcal{E}}\approx\mathrm{i}\Omega$ for systems with weak damping, and the above derivation is still used to avoid the ill-conditioning in solving the linear equations~\eqref{eq:expphi}. When $\lambda_i^{\mathcal{E}}\approx\mathrm{i}\Omega$, we have $\bar{\lambda}_i^{\mathcal{E}}\approx-\mathrm{i}\Omega$. 
\end{sloppypar}

\subsection{Proof of Theorem~\ref{theo:polar}}
\label{sec:proof-theo-polar}
\subsubsection{A lemma}
We first introduce a lemma which will be used in the proof of Theorems~\ref{theo:polar} and~\ref{theo:cartesian}.
\begin{lemma}
\label{le:ext-res}
For all $(\boldsymbol{l},\boldsymbol{j})\in\mathcal{R}_i$, and $\boldsymbol{r}$ satisfying the \emph{external} resonance condition~\eqref{eq:res-forcing}, we have
\begin{equation}
\langle \boldsymbol{l}-\boldsymbol{j}-\boldsymbol{e}_i, \boldsymbol{r}\rangle=0,
\end{equation}
where $\boldsymbol{e}_i\in\mathbb{R}^m$ is the unit vector alinged with the $i$-th axis.
\end{lemma}
\begin{proof}
Note that if the \emph{inner} resonance condition \eqref{eq:res-inner} and the \emph{external} resonance condition~\eqref{eq:res-forcing} holding exactly (i.e., `$\approx$' becomes `$=$' in \eqref{eq:res-inner},\eqref{eq:res-forcing}), we have
\begin{equation}
\label{eq:inner-external-res}
r_i=\boldsymbol{l}\cdot\boldsymbol{r}-\boldsymbol{j}\cdot{\boldsymbol{r}},
\end{equation}
which can be rewritten as $\langle \boldsymbol{l}-\boldsymbol{j}-\boldsymbol{e}_i, \boldsymbol{r}\rangle=0$. Now even when the inner and external resonance conditions are approximately satisfied, eq.~\eqref{eq:inner-external-res} still holds as the entries in $\boldsymbol{l}$ and $\boldsymbol{j}$ are integers.
\end{proof}

\subsubsection{Proof of the theorem}
\label{proofoftheorem}
\begin{sloppypar}
Based on Theorem~\ref{th:SSM-existence-uniqueness} along with equations~\eqref{eq:red-exp-eps}, \eqref{eq:red-auto-block}, \eqref{eq:leading-nonauto}, \eqref{eq:red-nonauto-lead} and~\eqref{eq:red-nonaut-block}, the reduced dynamics in normal form coordinates $(q_i,\bar{q}_i)$ is given by
\begin{equation}
\label{eq:polar-proof1}
\begin{pmatrix}\dot{q}_i\\\dot{\bar{q}}_i\end{pmatrix}=\boldsymbol{R}_{i}(\boldsymbol{p})+\epsilon\boldsymbol{S}_{\boldsymbol{0},i}(\Omega t)+\mathcal{O}(\epsilon |\boldsymbol{p}|)
\end{equation}
for $i=1,\cdots,m$. From~\eqref{thm:ssm-auto} and~\eqref{eq:polar-form}, we have
\begin{gather}
\label{eq:polar-proof2}
\scalebox{0.8}{$
\begin{aligned}
    & \boldsymbol{R}_{i}(\boldsymbol{p})  =\begin{pmatrix}\lambda_i^{\mathcal{E}}q_i\\\bar{\lambda}_i^{\mathcal{E}}\bar{q}_i\end{pmatrix}+\sum_{(\boldsymbol{l},\boldsymbol{j})\in\mathcal{R}_i}\begin{pmatrix}\gamma(\boldsymbol{l},\boldsymbol{j})\boldsymbol{q}^{\boldsymbol{l}}\bar{\boldsymbol{q}}^{\boldsymbol{j}}\\\bar{\gamma}(\boldsymbol{l},\boldsymbol{j})\boldsymbol{q}^{\boldsymbol{j}}\bar{\boldsymbol{q}}^{\boldsymbol{l}}\end{pmatrix}\\
    & = \begin{pmatrix}\lambda_i^{\mathcal{E}}\rho_ie^{\mathrm{i}(\theta_i+r_i\Omega t)}\\\bar{\lambda}_i^{\mathcal{E}}\rho_ie^{-\mathrm{i}(\theta_i+r_i\Omega t)}\end{pmatrix}+
     \sum_{(\boldsymbol{l},\boldsymbol{j})\in\mathcal{R}_i}\begin{pmatrix}\gamma(\boldsymbol{l},\boldsymbol{j})\boldsymbol{\rho}^{\boldsymbol{l}+\boldsymbol{j}}e^{\mathrm{i}(\langle \boldsymbol{l}-\boldsymbol{j},\boldsymbol{\theta} \rangle+\langle \boldsymbol{l}-\boldsymbol{j},\boldsymbol{r} \rangle\Omega t)}\\\bar{\gamma}(\boldsymbol{l},\boldsymbol{j})\boldsymbol{\rho}^{\boldsymbol{l}+\boldsymbol{j}}e^{\mathrm{i}(\langle \boldsymbol{j}-\boldsymbol{l},\boldsymbol{\theta} \rangle+\langle \boldsymbol{j}-\boldsymbol{l},\boldsymbol{r} \rangle\Omega t)}\end{pmatrix}\\
    & = \begin{pmatrix}\lambda_i^{\mathcal{E}}\rho_ie^{\mathrm{i}(\theta_i+r_i\Omega t)}\\\bar{\lambda}_i^{\mathcal{E}}\rho_ie^{-\mathrm{i}(\theta_i+r_i\Omega t)}\end{pmatrix}\\
    & +\sum_{(\boldsymbol{l},\boldsymbol{j})\in\mathcal{R}_i}\begin{pmatrix}\gamma(\boldsymbol{l},\boldsymbol{j})\boldsymbol{\rho}^{\boldsymbol{l}+\boldsymbol{j}}e^{\mathrm{i}(\langle \boldsymbol{l}-\boldsymbol{j}-\mathbf{e}_i,\boldsymbol{\theta} \rangle+\langle \boldsymbol{l}-\boldsymbol{j}-\mathbf{e}_i,\boldsymbol{r} \rangle\Omega t)}e^{\mathrm{i}(\theta_i+r_i\Omega t)}\\\bar{\gamma}(\boldsymbol{l},\boldsymbol{j})\boldsymbol{\rho}^{\boldsymbol{l}+\boldsymbol{j}}e^{\mathrm{i}(\langle \boldsymbol{j}-\boldsymbol{l}+\mathbf{e}_i,\boldsymbol{\theta} \rangle+\langle \boldsymbol{j}-\boldsymbol{l}+\mathbf{e}_i,\boldsymbol{r} \rangle\Omega t)}e^{-\mathrm{i}(\theta_i+r_i\Omega t)}\end{pmatrix}\\
   & =\begin{pmatrix}\lambda_i^{\mathcal{E}}\rho_ie^{\mathrm{i}(\theta_i+r_i\Omega t)}\\\bar{\lambda}_i^{\mathcal{E}}\rho_ie^{-\mathrm{i}(\theta_i+r_i\Omega t)}\end{pmatrix}+
    \sum_{(\boldsymbol{l},\boldsymbol{j})\in\mathcal{R}_i}\begin{pmatrix}\gamma(\boldsymbol{l},\boldsymbol{j})\boldsymbol{\rho}^{\boldsymbol{l}+\boldsymbol{j}}e^{\mathrm{i}\varphi_i(\boldsymbol{l},\boldsymbol{j})}e^{\mathrm{i}(\theta_i+r_i\Omega t)}\\\bar{\gamma}(\boldsymbol{l},\boldsymbol{j})\boldsymbol{\rho}^{\boldsymbol{l}+\boldsymbol{j}}e^{-\mathrm{i}\varphi_i(\boldsymbol{l},\boldsymbol{j})}e^{-\mathrm{i}(\theta_i+r_i\Omega t)}\end{pmatrix},
\end{aligned}$}
\end{gather}
where we have used Lemma~\ref{le:ext-res} and~\eqref{eq:varphi-ang} in the last equality. Using~\eqref{thm:ssm-nonauto},~\eqref{eq:ssm-nonauto} and~\eqref{eq:fi}, we have
\begin{equation}
\label{eq:polar-proof3}
\boldsymbol{S}_{\boldsymbol{0},i}(\Omega t)=\begin{pmatrix}f_ie^{\mathrm{i}r_i\Omega t}\\\bar{f}_ie^{-\mathrm{i}r_i\Omega t}\end{pmatrix}.
\end{equation}
\end{sloppypar}

Substituting equations~\eqref{eq:polar-proof2},~\eqref{eq:polar-proof3} and~\eqref{eq:polar-form} into~\eqref{eq:polar-proof1}, and factoring out $e^{\mathrm{i}(\theta_i+r_i\Omega t)}$ and its complex conjugate, we obtain
\begin{align}
    & \begin{pmatrix}\dot{\rho}_i+\mathrm{i}(\dot\theta_i+r_i\Omega)\rho_i\\\dot{\rho}_i-\mathrm{i}(\dot\theta_i+r_i\Omega)\rho_i\end{pmatrix}= \begin{pmatrix}\lambda_i^{\mathcal{E}}\rho_i\\\bar{\lambda}_i^{\mathcal{E}}\rho_i\end{pmatrix}\nonumber\\
    & +\sum_{(\boldsymbol{l},\boldsymbol{j})\in\mathcal{R}_i}
\begin{pmatrix}\gamma(\boldsymbol{l},\boldsymbol{j})\boldsymbol{\rho}^{\boldsymbol{l}+\boldsymbol{j}}e^{\mathrm{i}\varphi_i(\boldsymbol{l},\boldsymbol{j})}\\
\bar{\gamma}(\boldsymbol{l},\boldsymbol{j})\boldsymbol{\rho}^{\boldsymbol{l}+\boldsymbol{j}}e^{-\mathrm{i}\varphi_i(\boldsymbol{l},\boldsymbol{j})}
\end{pmatrix}+\epsilon\begin{pmatrix}f_ie^{-\mathrm{i}\theta_i}\\\bar{f}_ie^{\mathrm{i}\theta_i})\end{pmatrix}\nonumber\\& +\mathcal{O}(\epsilon |\boldsymbol{\rho}|)\boldsymbol{g}_i^\mathrm{p}(\phi),\label{eq:red_dyn_rho_i}
\end{align}
where $\boldsymbol{g}_i^\mathrm{p}:\mathbb{S}\to\mathbb{R}^2$ is a periodic function and ${\phi}=\Omega t$. Note that the second component in the above equation is simply the complex conjugate of the first component. Hence, equation~\eqref{eq:red_dyn_rho_i} holds if and only if the first component holds. Separation of the real and imaginary parts of the first component yields
\begin{align}
 & \dot{\rho}_i =  \mathrm{Re}(\lambda_i^{\mathcal{E}})\rho_i\nonumber\\
 & +\sum_{(\boldsymbol{l},\boldsymbol{j})\in\mathcal{R}_i}\boldsymbol{\rho}^{\boldsymbol{l}+\boldsymbol{j}}\mathrm{Re}(\gamma(\boldsymbol{l},\boldsymbol{j}))\cos\varphi_i(\boldsymbol{l},\boldsymbol{j})\nonumber\\
 & - \sum_{(\boldsymbol{l},\boldsymbol{j})\in\mathcal{R}_i}\boldsymbol{\rho}^{\boldsymbol{l}+\boldsymbol{j}}\mathrm{Im}(\gamma(\boldsymbol{l},\boldsymbol{j}))\sin\varphi_i(\boldsymbol{l},\boldsymbol{j})+ \epsilon\mathrm{Re}(f_i)\cos\theta_i\nonumber\\
 & +\epsilon\mathrm{Im}(f_i)\sin\theta_i+\mathcal{O}(\epsilon|\boldsymbol{\rho}|)g_{i,1}^\mathrm{p}(\phi)\label{eq:proof1-rho},
\end{align}
\begin{align}
& (\dot{\theta}_i+r_i\Omega)\rho_i =  \mathrm{Im}(\lambda_i^{\mathcal{E}})\rho_i\nonumber\\
& +\sum_{(\boldsymbol{l},\boldsymbol{j})\in\mathcal{R}_i}\boldsymbol{\rho}^{\boldsymbol{l}+\boldsymbol{j}}\mathrm{Re}(\gamma(\boldsymbol{l},\boldsymbol{j}))\sin\varphi_i(\boldsymbol{l},\boldsymbol{j})\nonumber\\
&-\sum_{(\boldsymbol{l},\boldsymbol{j})\in\mathcal{R}_i}\boldsymbol{\rho}^{\boldsymbol{l}+\boldsymbol{j}}\mathrm{Im}(\gamma(\boldsymbol{l},\boldsymbol{j}))\cos\varphi_i(\boldsymbol{l},\boldsymbol{j})-\epsilon\mathrm{Re}(f_i)\sin\theta_i\nonumber\\
& +\epsilon\mathrm{Im}(f_i)\cos\theta_i+\mathcal{O}(\epsilon|\boldsymbol{\rho}|)g_{i,2}^\mathrm{p}(\phi)\label{eq:proof1-theta},
\end{align}
where $g_{i,1}^\mathrm{p}$ and $g_{i,2}^\mathrm{p}$ are the first and the second component of the $\boldsymbol{g}_{i}^\mathrm{p}$, and we have $\dot{\phi}=\Omega$. The above two equations provide us~\eqref{eq:ode-reduced-slow-polar} after rearranging terms. This concludes the proof of statement (i).
\begin{sloppypar}
To prove statements (ii) and (iii), we first consider the leading-order reduced dynamics \begin{equation}
\label{eq:leading-order-red}    \dot{\boldsymbol{p}}=\boldsymbol{R}(\boldsymbol{p})+\epsilon\boldsymbol{S}_{\boldsymbol{0}}(\Omega t).
\end{equation}
We define $r_\mathrm{d}$ to be the largest common divisor for the set of rational numbers $\{r_i\}_{i=1}^m$ and set $T=2\pi/(r_\mathrm{d}\Omega)$. Then, from transformation~\eqref{eq:polar-form}, we deduce that any fixed point of the dynamical system~\eqref{eq:ode-reduced-slow-polar-leading} corresponds to a $T$-periodic solution of the leading-order reduced dynamics~\eqref{eq:leading-order-red} on the SSM, $\mathcal{W}(\mathcal{E},\Omega t)$. This is because all the polar radii $\rho_i$ and the phase differences $\theta_i$ are simultaneously constant at a fixed point. In addition, the periodic orbit inherits the stability of the fixed point.
\end{sloppypar}

We then need to show the persistence of a hyperbolic periodic orbit of the leading-order truncated dynamics under the addition of $\mathcal{O}(\epsilon |\boldsymbol{\rho}|)\boldsymbol{g}(\Omega t)$ to complete the proof of the statements (ii) and (iii). Since these statements are not affected by the choice of coordinates, they hold in Theorem~\ref{theo:cartesian} as well. For brevity, we show the persistence in detail only in the proof of Theorem~\ref{theo:cartesian}. As we will see, the persistence holds under proper inner and external resonance conditions. In particular, we ask for the smallness of $|\mathrm{Re}(\lambda_i^\mathcal{E})|$ and $|\mathrm{Im}(\lambda_i^\mathcal{E})-r_i\Omega|$ for $1\leq i\leq m$ such that the dynamics of $(\rho_i,\theta_i)$ is relatively slow compared to the phase dynamics $\dot{\phi}=\Omega$. This enables the construction of a slow-fast dynamical system. The method of averaging is then applied to complete the proof. Indeed, the leading-order dynamics is an approximated autonomous averaged system associated with the full reduced dynamics for $(\boldsymbol{\rho},\boldsymbol{\theta})$.

\subsection{Proof of Theorem~\ref{theo:cartesian}}
\label{sec:proof-theo-cartesian}
In this case,~\eqref{eq:polar-proof1} still holds. With~\eqref{thm:ssm-auto} and~\eqref{eq:cartesian-form}, we have
\begin{gather}
\label{eq:cartesian-proof2}
\scalebox{0.86}{$
\begin{aligned}
    & \boldsymbol{R}_{0,i}(\boldsymbol{p})  =\begin{pmatrix}\lambda_i^{\mathcal{E}}q_i\\\bar{\lambda}_i^{\mathcal{E}}\bar{q}_i\end{pmatrix}+\sum_{(\boldsymbol{l},\boldsymbol{j})\in\mathcal{R}_i}\begin{pmatrix}\gamma(\boldsymbol{l},\boldsymbol{j})\boldsymbol{q}^{\boldsymbol{l}}\bar{\boldsymbol{q}}^{\boldsymbol{j}}\\\bar{\gamma}(\boldsymbol{l},\boldsymbol{j})\boldsymbol{q}^{\boldsymbol{j}}\bar{\boldsymbol{q}}^{\boldsymbol{l}}\end{pmatrix}\\
    & = \begin{pmatrix}\lambda_i^{\mathcal{E}}(q_{i,\mathrm{s}}^{\mathrm{R}}+\mathrm{i}q_{i,\mathrm{s}}^{\mathrm{I}})e^{\mathrm{i}r_i\Omega t}\\\bar{\lambda}_i^{\mathcal{E}}(q_{i,\mathrm{s}}^{\mathrm{R}}-\mathrm{i}q_{i,\mathrm{s}}^{\mathrm{I}})e^{-\mathrm{i}r_i\Omega t}\end{pmatrix}+\sum_{(\boldsymbol{l},\boldsymbol{j})\in\mathcal{R}_i}\begin{pmatrix}\gamma(\boldsymbol{l},\boldsymbol{j})\boldsymbol{q}_s^{\boldsymbol{l}}\bar{\boldsymbol{q}}_s^{\boldsymbol{j}}e^{\mathrm{i}\langle \boldsymbol{l}-\boldsymbol{j},\boldsymbol{r} \rangle\Omega t}\\\bar{\gamma}(\boldsymbol{l},\boldsymbol{j})\bar{\boldsymbol{q}}_s^{\boldsymbol{l}}\boldsymbol{q}_s^{\boldsymbol{j}}e^{\mathrm{i}\langle \boldsymbol{j}-\boldsymbol{l},\boldsymbol{r} \rangle\Omega t}\end{pmatrix}\\
   & =\begin{pmatrix}\lambda_i^{\mathcal{E}}(q_{i,\mathrm{s}}^{\mathrm{R}}+\mathrm{i}q_{i,\mathrm{s}}^{\mathrm{I}})e^{\mathrm{i}r_i\Omega t}\\\bar{\lambda}_i^{\mathcal{E}}(q_{i,\mathrm{s}}^{\mathrm{R}}-\mathrm{i}q_{i,\mathrm{s}}^{\mathrm{I}})e^{-\mathrm{i}r_i\Omega t}\end{pmatrix}+\\
   & \sum_{(\boldsymbol{l},\boldsymbol{j})\in\mathcal{R}_i}\begin{pmatrix}\gamma(\boldsymbol{l},\boldsymbol{j})\boldsymbol{q}_s^{\boldsymbol{l}}\bar{\boldsymbol{q}}_s^{\boldsymbol{j}}e^{\mathrm{i}\langle \boldsymbol{l}-\boldsymbol{j}-\boldsymbol{e}_i,\boldsymbol{r} \rangle\Omega t}e^{\mathrm{i}r_i\Omega t}\\\bar{\gamma}(\boldsymbol{l},\boldsymbol{j})\bar{\boldsymbol{q}}_s^{\boldsymbol{l}}\boldsymbol{q}_s^{\boldsymbol{j}}e^{\mathrm{i}\langle \boldsymbol{j}-\boldsymbol{l}+\boldsymbol{e}_i,\boldsymbol{r} \rangle\Omega t}e^{-\mathrm{i}r_i\Omega t}\end{pmatrix}\\
   & = \begin{pmatrix}\lambda_i^{\mathcal{E}}(q_{i,\mathrm{s}}^{\mathrm{R}}+\mathrm{i}q_{i,\mathrm{s}}^{\mathrm{I}})e^{\mathrm{i}r_i\Omega t}\\\bar{\lambda}_i^{\mathcal{E}}(q_{i,\mathrm{s}}^{\mathrm{R}}-\mathrm{i}q_{i,\mathrm{s}}^{\mathrm{I}})e^{-\mathrm{i}r_i\Omega t}\end{pmatrix}+\sum_{(\boldsymbol{l},\boldsymbol{j})\in\mathcal{R}_i}\begin{pmatrix}\gamma(\boldsymbol{l},\boldsymbol{j})\boldsymbol{q}_s^{\boldsymbol{l}}\bar{\boldsymbol{q}}_s^{\boldsymbol{j}}e^{\mathrm{i}r_i\Omega t}\\\bar{\gamma}(\boldsymbol{l},\boldsymbol{j})\bar{\boldsymbol{q}}_s^{\boldsymbol{l}}\boldsymbol{q}_s^{\boldsymbol{j}}e^{-\mathrm{i}r_i\Omega t}\end{pmatrix},
\end{aligned}$}
\end{gather}
where we have used Lemma~\ref{le:ext-res} in the last equality. In addition, \eqref{eq:polar-proof3} still holds.

Substituting equations~\eqref{eq:cartesian-proof2},~\eqref{eq:polar-proof3} and~\eqref{eq:cartesian-form} into~\eqref{eq:polar-proof1}, and factoring out $e^{\mathrm{i}r_i\Omega t}$ and its complex conjugate yield
\begin{align}
    & \begin{pmatrix}\dot{q}_{i,\mathrm{s}}^{\mathrm{R}}+\mathrm{i}\dot{q}_{i,\mathrm{s}}^{\mathrm{I}}+(-q_{i,\mathrm{s}}^{\mathrm{I}}+\mathrm{i}q_{i,\mathrm{s}}^{\mathrm{R}})r_i\Omega\\\dot{q}_{i,\mathrm{s}}^{\mathrm{R}}-\mathrm{i}\dot{q}_{i,\mathrm{s}}^{\mathrm{I}}+(-q_{i,\mathrm{s}}^{\mathrm{I}}-\mathrm{i}q_{i,\mathrm{s}}^{\mathrm{R}})r_i\Omega\end{pmatrix}= \begin{pmatrix}\lambda_i^{\mathcal{E}}(q_{i,\mathrm{s}}^{\mathrm{R}}+\mathrm{i}q_{i,\mathrm{s}}^{\mathrm{I}})\\\bar{\lambda}_i^{\mathcal{E}}(q_{i,\mathrm{s}}^{\mathrm{R}}-\mathrm{i}q_{i,\mathrm{s}}^{\mathrm{I}})\end{pmatrix}\nonumber\\
    &+\sum_{(\boldsymbol{l},\boldsymbol{j})\in\mathcal{R}_i}
\begin{pmatrix}\gamma(\boldsymbol{l},\boldsymbol{j})\boldsymbol{q}_s^{\boldsymbol{l}}\bar{\boldsymbol{q}}_s^{\boldsymbol{j}}\\
\bar{\gamma}(\boldsymbol{l},\boldsymbol{j})\bar{\boldsymbol{q}}_s^{\boldsymbol{l}}\boldsymbol{q}_s^{\boldsymbol{j}}
\end{pmatrix}+\epsilon\begin{pmatrix}f_i\\\bar{f}_i\end{pmatrix}\nonumber\\
& +\mathcal{O}(\epsilon|\boldsymbol{q}_\mathrm{s}|)\boldsymbol{g}_i^\mathrm{c}(\phi),
\end{align}
where $\boldsymbol{g}_i^\mathrm{c}:\mathbb{S}\to\mathbb{R}^2$ is a periodic function and $\phi=\Omega t$. Note that the second component in the above equation is simply the complex conjugate of the first component. It follows that the equation holds if and only if the first component holds. Separation of real and imaginary parts of the first component yields
\begin{align}
& \dot{q}_{i,\mathrm{s}}^{\mathrm{R}}-q_{i,\mathrm{s}}^{\mathrm{I}}r_i\Omega  =  \mathrm{Re}(\lambda_i^{\mathcal{E}})q_{i,\mathrm{s}}^{\mathrm{R}}-\mathrm{Im}(\lambda_i^{\mathcal{E}})q_{i,\mathrm{s}}^{\mathrm{I}}\nonumber\\
& +\sum_{(\boldsymbol{l},\boldsymbol{j})\in\mathcal{R}_i}\mathrm{Re}\left(\gamma(\boldsymbol{l},\boldsymbol{j})\boldsymbol{q}_s^{\boldsymbol{l}}\bar{\boldsymbol{q}}_s^{\boldsymbol{j}}\right)+\epsilon\mathrm{Re}(f_i)+\mathcal{O}(\epsilon|\boldsymbol{q}_\mathrm{s}|)g_{i,1}^\mathrm{c}(\phi),\\
& \dot{q}_{i,\mathrm{s}}^{\mathrm{I}}+q_{i,\mathrm{s}}^{\mathrm{R}}r_i\Omega  = \mathrm{Re}(\lambda_i^{\mathcal{E}})q_{i,\mathrm{s}}^{\mathrm{I}}+\mathrm{Im}(\lambda_i^{\mathcal{E}})q_{i,\mathrm{s}}^{\mathrm{R}}\nonumber\\
& +\sum_{(\boldsymbol{l},\boldsymbol{j})\in\mathcal{R}_i}\mathrm{Im}\left(\gamma(\boldsymbol{l},\boldsymbol{j})\boldsymbol{q}_s^{\boldsymbol{l}}\bar{\boldsymbol{q}}_s^{\boldsymbol{j}}\right)+\epsilon\mathrm{Im}(f_i)+\mathcal{O}(\epsilon|\boldsymbol{q}_\mathrm{s}|)g_{i,2}^\mathrm{c}(\phi),
\end{align}
where $g_{i,1}^\mathrm{c}$ and $g_{i,2}^\mathrm{c}$ are the first and the second component of the $\boldsymbol{g}_{i}^\mathrm{c}$, and we have $\dot{\phi}=\Omega.$ After some algebraic manipulations, we obtain~\eqref{eq:ode-reduced-slow-cartesian}.

The proof of statements (ii) and (iii) is analogous to that given in Section~\ref{proofoftheorem}. Here we focus on the persistence of the hyperbolic periodic orbits of the leading-order truncated dynamics under the addition higher order terms.

Let $\mathbf{x}=(q_{1,\mathrm{s}}^{\mathrm{R}},q_{1,\mathrm{s}}^{\mathrm{I}},\cdots,q_{m,\mathrm{s}}^{\mathrm{R}},q_{m,\mathrm{s}}^{\mathrm{I}})$. Equation~\eqref{eq:ode-reduced-slow-cartesian} can be rewritten as
\begin{equation}
\label{eq:ode-qc}
    \dot{\mathbf{x}}=\mathbf{A}\mathbf{x}+\mathbf{F}(\mathbf{x})+\epsilon \mathbf{F}^\mathrm{ext}+\mathcal{O}(\epsilon|\mathbf{x}|)\mathbf{G}(\phi),\quad \dot{\phi}=\Omega
\end{equation}
where $\mathbf{A}= \diag(\boldsymbol{A}_1,\cdots,\boldsymbol{A}_m)$ with
\begin{gather}
\boldsymbol{A}_i=\begin{pmatrix}\mathrm{Re}(\lambda_i^{\mathcal{E}}) & r_i\Omega-\mathrm{Im}(\lambda_i^{\mathcal{E}})\\
\mathrm{Im}(\lambda_i^{\mathcal{E}})-r_i\Omega & \mathrm{Re}(\lambda_i^{\mathcal{E}})\end{pmatrix}\label{eq:ai},\\
\mathbf{F}(\mathbf{x})=\begin{pmatrix}\sum_{(\boldsymbol{l},\boldsymbol{j})\in\mathcal{R}_1}\begin{pmatrix}\mathrm{Re}\left(\gamma(\boldsymbol{l},\boldsymbol{j})\boldsymbol{q}_s^{\boldsymbol{l}}\bar{\boldsymbol{q}}_s^{\boldsymbol{j}}\right)\\\mathrm{Im}\left(\gamma(\boldsymbol{l},\boldsymbol{j})\boldsymbol{q}_s^{\boldsymbol{l}}\bar{\boldsymbol{q}}_s^{\boldsymbol{j}}\right)\end{pmatrix}\\\vdots\\\sum_{(\boldsymbol{l},\boldsymbol{j})\in\mathcal{R}_m}\begin{pmatrix}\mathrm{Re}\left(\gamma(\boldsymbol{l},\boldsymbol{j})\boldsymbol{q}_s^{\boldsymbol{l}}\bar{\boldsymbol{q}}_s^{\boldsymbol{j}}\right)\\\mathrm{Im}\left(\gamma(\boldsymbol{l},\boldsymbol{j})\boldsymbol{q}_s^{\boldsymbol{l}}\bar{\boldsymbol{q}}_s^{\boldsymbol{j}}\right)\end{pmatrix}\end{pmatrix},
\end{gather}
$\mathbf{F}^\mathrm{ext}=\left(\mathrm{Re}(f_1),\mathrm{Im}(f_1),\cdots,\mathrm{Re}(f_m),\mathrm{Im}(f_m)\right)$ is a constant vector, and $\mathbf{G}(\phi)$ is a periodic function. Let $\mathbf{x}^{\star}$ be a hyperbolic fixed point of the leading-order truncation, i.e.,
\begin{equation}
\label{eq:qpc-ep}
    \mathbf{A}\mathbf{x}^\star+\mathbf{F}(\mathbf{x}^\star)+\epsilon \mathbf{F}^\mathrm{ext}=0,
\end{equation}
and let the corresponding periodic orbit in the parameterization coordinates be $\boldsymbol{p}^\star(t)$ (see equation~\eqref{eq:cartesian-form}).
We will prove the persistence of this hyperbolic periodic orbit with the perturbation of $\mathcal{O}(\epsilon|\mathbf{x}|)\mathbf{G}(\phi)$ via the following three steps: (i) we estimate the magnitude of the fixed point $\mathbf{x}^\star$; (ii) we introduce transverse coordinates $\mathbf{y}=\mathbf{x}-\mathbf{x}^\star$ and then show the dynamics of $\mathbf{y}$ is slow relative to $\dot{\phi}=\Omega$; (iii) we use the method of averaging to demonstrate that the hyperbolic fixed point $\mathbf{y}=0$ is perturbed as a periodic orbit $\mathbf{y}_\mathrm{p}(t)$ of the same hyperbolicity as that of $\mathbf{y}=0$ (see Guckenheimer \& Holmes~\cite{guckenheimer2013nonlinear}). Hence, it is clear that the corresponding trajectory $\boldsymbol{p}_\mathrm{p}(t)$ that perturbed from $\boldsymbol{p}^\star(t)$ is also a periodic orbit of the same hyperbolicity.

\begin{sloppypar}
\textbf{Step 1:} Let the lowest order of nonlinearity in $\mathbf{F}(\mathbf{x})$ be $k$, assume that \begin{equation}
\label{eq:res-assumption}    
\boldsymbol{\lambda}^{\mathcal{E}}-\mathrm{i}\boldsymbol{r}\Omega=\epsilon^q\boldsymbol{t}, \quad \bar{\boldsymbol{\lambda}}^{\mathcal{E}}+\mathrm{i}\boldsymbol{r}\Omega=\epsilon^q\bar{\boldsymbol{t}}
\end{equation} for $q=1-\frac{1}{k}$ and some $\boldsymbol{t}\sim\mathcal{O}(1)$. Now, we can show that $\mathbf{x}^\star\sim\mathcal{O}(\epsilon^{1-q})$.
Indeed, substituting assumption~\eqref{eq:res-assumption} into equation~\eqref{eq:ai} gives $\mathbf{A}=\epsilon^q\hat{\mathbf{A}}$, where $\hat{\mathbf{A}}= \diag(\hat{\boldsymbol{A}}_1,\cdots,\hat{\boldsymbol{A}}_m)$ with
\begin{equation}
\hat{\boldsymbol{A}}_i=\begin{pmatrix}\mathrm{Re}(t_i) & -\mathrm{Im}(t_i)\\
\mathrm{Im}(t_i) & \mathrm{Re}(t_i)\end{pmatrix}.\label{eq:ai-hat}\end{equation} In addition, we have $\mathbf{F}(\mathbf{x})= \hat{\mathbf{F}}(\mathbf{x})\cdot \mathbf{x}^{\otimes k}$, where $\mathbf{x}^{\otimes k}:=\mathbf{x}\otimes \cdots\otimes \mathbf{x}$ ($k$-times) and $k\geq2$ because $|\boldsymbol{l}|+|\boldsymbol{j}|\geq2$, and $\hat{\mathbf{F}}(\mathbf{x})=\hat{\mathbf{F}}(0)+\hat{\mathbf{F}}_1\mathbf{x}+\mathcal{O}(|\mathbf{x}|^2)$ with $\hat{\mathbf{F}}_1$ appropriately defined. Thus, introducing the transformation
\begin{equation}
\label{eq:xtoxhat}
    \mathbf{x}=\mu\hat{\mathbf{x}},\quad \mu=\epsilon^{1-q},
\end{equation}
equation~\eqref{eq:qpc-ep} can be rewritten as \begin{equation}\epsilon\mathcal{F}(\hat{\mathbf{x}}^\star,\mu)=0,\end{equation} where
\begin{equation}
\label{eq:fcal}
    \mathcal{F}(\hat{\mathbf{x}}^\star,\mu)=\hat{\mathbf{A}}\hat{\mathbf{x}}^\star+\hat{\mathbf{F}}(\mu\hat{\mathbf{x}}^\star)\cdot (\hat{\mathbf{x}}^\star)^{\otimes k}+\mathbf{F}^{\mathrm{ext}}.
\end{equation}
Since $\hat{\mathbf{x}}^\star$ is a hyperbolic fixed point, the partial derivative of $\mathcal{F}$ with respect to the first argument, evaluated at $(\hat{\mathbf{x}}^\star,\mu)$, is invertible. Then, the implicit function theorem implies that $\hat{\mathbf{x}}^\star$ depends on $\mu$ smoothly and we have $\hat{\mathbf{x}}^\star=\hat{\mathbf{x}}^\star(\mu)$, i.e., $\mathbf{x}^\star=\mu\hat{\mathbf{x}}^\star(\mu)$. In particular, $\hat{\mathbf{x}}^\star(0)=\hat{\mathbf{x}}^\star(\mu)+\mathcal{O}(\mu)$. Furthermore, since the invertible matrix $\hat{\mathbf{A}}\sim\mathcal{O}(1)$, $\mathbf{F}^\mathrm{ext}\sim\mathcal{O}(1)$, and $\hat{\mathbf{F}}(0)$ can be made arbitrarily small by scaling the eigenvectors of the master spectral subspace $\mathcal{E}$, we
infer from $\mathcal{F}(\hat{\mathbf{x}}^\star,\mu)=0$ that for small enough values of $\mu$, $\hat{\mathbf{x}}^\star(\mu)\sim\mathcal{O}(1)+\mathcal{O}(\mu)$, i.e., $\mathbf{x}^\star\sim\mathcal{O}(\epsilon^{1-q})$.
\end{sloppypar}

\textbf{Step 2:} Following the analysis in Step 1 (see~\eqref{eq:xtoxhat}-\eqref{eq:fcal}), equation~\eqref{eq:ode-qc} can be rewritten as
\begin{equation}
    \mu\dot{\hat{\mathbf{x}}}=\epsilon \mathcal{F}(\hat{\mathbf{x}},\mu)+\mathcal{O}(\epsilon |\mu\hat{\mathbf{x}}|)\mathbf{G}(\phi),\quad\dot{\phi}=\Omega.
\end{equation}
The first equation above can be simplified as
\begin{equation}
    \dot{\hat{\mathbf{x}}} = \epsilon^q\mathcal{F}(\hat{\mathbf{x}},\mu)+\mathcal{O}(\epsilon |\hat{\mathbf{x}}|)\mathbf{G}(\phi).
\end{equation}
Letting $\mathbf{y}=\hat{\mathbf{x}}-\hat{\mathbf{x}}^\star$, substituting $\hat{\mathbf{x}}=\mathbf{y}+\hat{\mathbf{x}}^\star$ into the above equation, performing Taylor expansion around the fixed point $\hat{\mathbf{x}}^\star$, and utilizing the fact that $\mathcal{F}(\hat{\mathbf{x}}^\star,\mu)=0$, we obtain
\begin{align}
\dot{\mathbf{y}}&=\epsilon^q\mathcal{F}(\mathbf{y}+\hat{\mathbf{x}}^\star,\mu)+\mathcal{O}(\epsilon |\mathbf{y}+\hat{\mathbf{x}}^\star|)\mathbf{G}(\phi)\nonumber\\
& = \epsilon^q D_1\mathcal{F}(\hat{\mathbf{x}}^\star,\mu)\mathbf{y}+\epsilon^q\mathcal{O}(|\mathbf{y}|^2)+\mathcal{O}(\epsilon |\hat{\mathbf{x}}^\star+\mathbf{y}|)\mathbf{G}(\phi).
\end{align}
where $D_1\mathcal{F}$ denotes the partial derivative of $\mathcal{F}$ with respect to its first argument. Next we introduce the transformation $\mathbf{y}=\epsilon^r\hat{\mathbf{y}}$ for some $r>0$ and obtain
\begin{align}
\dot{\hat{\mathbf{y}}}& =\epsilon^q D_1\mathcal{F}(\hat{\mathbf{x}}^\star,\mu)\hat{\mathbf{y}}+\epsilon^{q+r}\mathcal{O}(|\hat{\mathbf{y}}|^2)\nonumber\\
& +\mathcal{O}(\epsilon^{1-r} |\hat{\mathbf{x}}^\star+\epsilon^r\hat{\mathbf{y}}|)\mathbf{G}(\phi)\nonumber\\
 & = \epsilon^q D_1\mathcal{F}(\hat{\mathbf{x}}^\star,\mu))\hat{\mathbf{y}}+\epsilon^{q+r}\mathcal{O}(|\hat{\mathbf{y}}|^2)\nonumber\\
 & +\epsilon^{1-r}\mathcal{O}(|\hat{\mathbf{x}}^\star|)\mathbf{G}(\phi)+\epsilon\mathcal{O}(|\hat{\mathbf{y}}|)\mathbf{G}(\phi)\label{eq:yhat1}.
\end{align}
Now, we choose $r$ such that $q+r=1-r$, i.e.,
\begin{equation}
    r=\frac{1-q}{2}=\frac{1}{2k}.
\end{equation}
Then, equation~\eqref{eq:yhat1} is simplified to yield
\begin{align}
\dot{\hat{\mathbf{y}}} = & \epsilon^q D_1\mathcal{F}(\hat{\mathbf{x}}^\star,\mu)\hat{\mathbf{y}}+\epsilon^{\frac{1+q}{2}}\left(\mathcal{O}(|\hat{\mathbf{y}}|^2)+\mathcal{O}(|\hat{\mathbf{x}}^\star|)\mathbf{G}(\phi)\right)\nonumber\\
& +\epsilon\mathcal{O}(|\hat{\mathbf{y}}|)\mathbf{G}(\phi)\label{eq:yhat2}.
\end{align}

\textbf{Step 3}: Defining $\nu_1=\epsilon^q$ and $\nu_2=\sqrt{\mu}$, we rewrite equation~\eqref{eq:yhat2} as 
\begin{equation}
\dot{\hat{\mathbf{y}}} = \nu_1 D_1\mathcal{F}(\hat{\mathbf{x}}^\star,\nu_2^2)\hat{\mathbf{y}}+\nu_1\nu_2\mathbf{H}(\nu_2,\hat{\mathbf{y}},\phi)
\end{equation}
where $\mathbf{H}(\nu_2,\hat{\mathbf{y}},\phi)=\mathcal{O}(|\hat{\mathbf{y}}|^2)+(\mathcal{O}(|\hat{\mathbf{x}}^\star|)+\nu_2\mathcal{O}(|\hat{\mathbf{y}}|))\mathbf{G}(\phi)$. We define $\mathbf{A}_0=D_1\mathcal{F}(\hat{\mathbf{x}}^\star,0)$, which is $\nu_2^2$-close to the Jacobian $D_1\mathcal{F}(\hat{\mathbf{x}}^\star,\nu_2^2)$ and, hence, these two matrices share the same hyperbolicity for small-enough values of $\nu_2$. Following the arguments of the proof of the averaging theorem in~\cite{guckenheimer2013nonlinear}, we consider two flows as follows
\begin{gather}
    \dot{\hat{\mathbf{y}}} = \nu_1 \mathbf{A}_0\hat{\mathbf{y}},\quad\dot{\phi}=\Omega,\label{eq:linearflow}\\
    \dot{\hat{\mathbf{y}}} = \nu_1 D_1\mathcal{F}(\hat{\mathbf{x}}^\star,\nu_2^2)\hat{\mathbf{y}}+\nu_1\nu_2\mathbf{H}(\nu_2,\hat{\mathbf{y}},\phi),\,\,\dot{\phi}=\Omega.\label{eq:nonlinaerflow}
\end{gather}
Let $T=2\pi/(r_\mathrm{d}\Omega)$, where $r_\mathrm{d}$ has been defined as the largest common divisor for the set of of rational numbers $\{r_i\}_{i=1}^m$ in the proof of Theorem~\ref{sec:proof-theo-polar}. We define the period-$T$ maps of the above two flows as $\mathcal{P}_{0}$ and $\mathcal{P}_{\nu}$ respectively. Furthermore, we define $\mathcal{H}_0$ and $\mathcal{H}_\nu$ as the zero functions associated with the fixed points of the Poincar\'e maps $\mathcal{P}_{0}$ and $\mathcal{P}_{\nu}$ as
\begin{gather}
    \mathcal{H}_0(\hat{\mathbf{y}},\nu_1)=\frac{1}{\nu_1}(\mathcal{P}_{0}\hat{\mathbf{y}}-\hat{\mathbf{y}}),\\
    \mathcal{H}_\nu(\hat{\mathbf{y}},\nu_1,\nu_2)=\frac{1}{\nu_1}(\mathcal{P}_{\nu}\hat{\mathbf{y}}-\hat{\mathbf{y}}).
\end{gather}
From the linear flow~\eqref{eq:linearflow}, we obtain
\begin{equation}
    \mathcal{P}_{0}:\hat{\mathbf{y}}\mapsto e^{\nu_1\mathbf{A}_0T}\hat{\mathbf{y}}.
\end{equation}
Then, $\hat{\mathbf{y}}=\mathbf{0}$ is a fixed point of the map $\mathcal{P}_0$ and as a result, also the zero of the function $\mathcal{H}_0(\hat{\mathbf{y}},\nu_1)$. In addition
\begin{equation}
    \lim_{\nu_1\to0}\partial_{\hat{\mathbf{y}}}\mathcal{H}_0=\lim_{\nu_1\to0}\frac{e^{\nu_1\mathbf{A}_0T}-\mathbb{I}}{\nu_1}=\mathbf{A}_0T,
\end{equation}
which is invertible. Furthermore, since $\mathcal{P}_\nu$ is $\nu_1\nu_2$-close to $\mathcal{P}_0$, we also have
\begin{equation}
    \lim_{(\nu_1,\nu_2)\to\mathbf{0}}\partial_{\hat{\mathbf{y}}}\mathcal{H}_\nu=\lim_{(\nu_1,\nu_2)\to\mathbf{0}}\frac{\mathcal{P}_\nu-\mathbb{I}}{\nu_1}=\mathbf{A}_0T.
\end{equation}
Now, by the implicit function theorem, the trivial fixed point $\hat{\mathbf{y}}=0$ of the map $\mathcal{P}_{0}$ is perturbed as a nontrivial fixed point of the map $\mathcal{P}_\nu$ under the addition of the higher-order terms for small-enough values of $\nu_1,\nu_2$. In addition, the nontrivial fixed point shares the same hyperbolicity as that of the trivial one. Therefore, we obtain a periodic orbit $\hat{\mathbf{y}}_\mathrm{p}(t)$ to~\eqref{eq:nonlinaerflow}, and then $\mathbf{y}_\mathrm{p}(t)=\epsilon^r\hat{\mathbf{y}}_\mathrm{p}(t)$.

\subsection{Settings of COCO}
\label{sec:apdix-coco}
Some settings are tuned as follows to speed up the FRC computation in examples~\ref{sec:vonKarmanBeam}-\ref{sec:vonKarmanPlate} with \texttt{po} toolbox of \textsc{coco}
\begin{itemize}
\item \begin{sloppypar}Disable mesh adaptation. When the mesh is changed, \textsc{coco} will reconstruct the continuation problem, which could be time-consuming if the problem is of high dimension. \textcolor{black}{We have disabled mesh adaption in the von K\'arm\'an beam example the von K\'arm\'an plate example. However, we found that the default mesh is not able to produce accurate results in the Timoshenko beam example when the deformation amplitude is large. So we allow for mesh adaptation every ten continuation steps in the Timoshenko beam example};\end{sloppypar}
\item Disable \texttt{MXCL}. The collocation toolbox in \textsc{coco} has a posteriori error estimator to evaluate the accuracy of obtained numerical solution. If the error exceeds a threshold value, \textsc{coco} will stop the continuation run. An often used technique to avoid the occurrence of \texttt{MXCL} is providing a fine mesh and adaptively changing the mesh after a few continuation steps. It is noted that the error in the estimator is based on the Euclidean norm, which means that \texttt{MXCL} will be triggered easily for high-dimensional problems. In \textcolor{black}{the von K\'arm\'an beam and plate examples}, we use a fixed (default) mesh with ten subintervals. Five base points and four collocation nodes are used in each subinterval. \textcolor{black}{In the Timoshenko beam example, the \texttt{MXCL} is also disabled};
\item Increase maximum step size and residual. We use atlas the 1d algorithm in \textsc{coco} to perform continuation in this paper. The default maximum continuation step size is 0.5 and maximum residual allowed for predictor is 0.1. The step size in atlas 1d measures distances in the Euclidean norm of \emph{all} continuation variables and parameters. So we allow large continuation step size for high-dimensional continuation problems. In addition, we increase the maximum residual for the predictor as well to an effective end. Here we have increased the maximum step size and residual to 100 and 10 respectively in the von K\'arm\'an beam example. These two thresholds are set to be \textcolor{black}{1000 and 10000 in the Timoshenko beam example, and} 500 and 50 in the von K\'arm\'an plate example. In the continuation runs of the von K\'arm\'an beam example, the residual of the predictor hit the threshold 10 in some continuation steps and the observed maximum continuation step size is about 30, which is much larger than the default.~\textcolor{black}{In the continuation run of the Timoshenko beam example, there are some continuation steps with step size more than 900, which is also much larger than the default.} In the continuation runs of the von K\'arm\'an plate example, the observed maximum residual of predictor is slightly larger than one while the observed maximum continuation step size is about 34, which is again much larger than the default.
\end{itemize}

Note that the 2020 March release of \textsc{coco} also supports $k$-dimensional atlas algorithm where step size measures distance of with Euclidean norm of \emph{active continuation parameters} only~\cite{dankowicz2020multidimensional}. With $(\boldsymbol{x}_0/(2n),\dot{\boldsymbol{x}}_0/(2n),\Omega,T)$ as active continuation parameters, we also performed continuation using \texttt{po} with atlas-kd for the von K\'arm\'an beam discretized with 20 elements and 58 DOF. The default maximum continuation step size (equal to one) in atlas-kd is utilized. We have the decreased minimum continuation step size to $10^{-4}$ such that gap between adjacent charts is not encountered (see~\cite{dankowicz2013recipes} for more details). We have set $\theta<0.5$ in the algorithm such that the predictor in atlas-kd is consistent with the one in atlas-1d. Given the residual of predictor is evaluated as the same way as atlas-1d, we have also increased the maximum residual for predictor to 10.

The continuation run with atlas-1d generates the FRC with 175 points in about six and half hours for the discrete beam with 20 elements. The observed maximum continuation step size in this run is about 30. In contrast, the continuation run with atlas-kd generates the FRC with 248 points in about 11 hours. The residual of predictor in this run again hits the threshold 10 in some continuation steps, and the observed maximum continuation step size is just 0.1. When the maximum residual for predictor is increased to 100, the continuation run with atlas-kd generates the FRC with 107 points in about six hours, and the observed maximum continuation step size is increased to 0.17. It follows that the computational times for the two atlas algorithms are comparable if we allow large continuation step size in the atlas-1d algorithm.

\section*{Declarations}
\section*{Conflict of interest}
The authors declare that they have no conflict of interest.

\section*{Data availability}
The data used to generate the numerical results included in this paper are available from the corresponding author on request.

\section*{Code availability}
\begin{sloppypar}
The code used to generate the numerical results   included in this paper are available as part of the open-source Matlab script SSMTool 2.1 under~\url{https://github.com/haller-group/SSMTool-2.1}.
\end{sloppypar}

\bibliography{manuscript_revised_v1.bbl} 

\begin{thebibliography}{10}
\providecommand{\url}[1]{{#1}}
\providecommand{\urlprefix}{URL }
\expandafter\ifx\csname urlstyle\endcsname\relax
  \providecommand{\doi}[1]{DOI~\discretionary{}{}{}#1}\else
  \providecommand{\doi}{DOI~\discretionary{}{}{}\begingroup
  \urlstyle{rm}\Url}\fi

\bibitem{allman1976simple}
Allman, D.: A simple cubic displacement element for plate bending.
\newblock International Journal for Numerical Methods in Engineering
  \textbf{10}(2), 263--281 (1976)

\bibitem{allman1996implementation}
Allman, D.: Implementation of a flat facet shell finite element for
  applications in structural dynamics.
\newblock Computers \& Structures \textbf{59}(4), 657--663 (1996)

\bibitem{antonio2012frequency}
Antonio, D., Zanette, D.H., L{\'o}pez, D.: Frequency stabilization in nonlinear
  micromechanical oscillators.
\newblock Nature Communications \textbf{3}(1), 1--6 (2012)

\bibitem{ascher1995numerical}
Ascher, U.M., Mattheij, R.M., Russell, R.D.: Numerical solution of boundary
  value problems for ordinary differential equations.
\newblock SIAM (1995)

\bibitem{balachandran1991observations}
Balachandran, B., Nayfeh, A.: Observations of modal interactions in resonantly
  forced beam-mass structures.
\newblock Nonlinear Dynamics \textbf{2}(2), 77--117 (1991)

\bibitem{bilal2020experiments}
Bilal, N., Tripathi, A., Bajaj, A.: On experiments in harmonically excited
  cantilever plates with 1: 2 internal resonance.
\newblock Nonlinear Dynamics pp. 1--18 (2020)

\bibitem{breunung2018explicit}
Breunung, T., Haller, G.: Explicit backbone curves from spectral submanifolds
  of forced-damped nonlinear mechanical systems.
\newblock Proceedings of the Royal Society A: Mathematical, Physical and
  Engineering Sciences \textbf{474}(2213), 20180083 (2018)

\bibitem{cabre2005parameterization-iii}
Cabr{\'e}, X., Fontich, E., De~La~Llave, R.: The parameterization method for
  invariant manifolds iii: overview and applications.
\newblock Journal of Differential Equations \textbf{218}(2), 444--515 (2005)

\bibitem{cabre2003parameterization-i}
Cabr{\'e}, X., Fontich, E., de~la Llave, R.: The parameterization method for
  invariant manifolds i: manifolds associated to non-resonant subspaces.
\newblock Indiana University mathematics journal pp. 283--328 (2003)

\bibitem{cabre2003parameterization-ii}
Cabr{\'e}, X., Fontich, E., de~la Llave, R.: The parameterization method for
  invariant manifolds ii: regularity with respect to parameters.
\newblock Indiana University mathematics journal pp. 329--360 (2003)

\bibitem{cammarano2014bifurcations}
Cammarano, A., Hill, T., Neild, S., Wagg, D.: Bifurcations of backbone curves
  for systems of coupled nonlinear two mass oscillator.
\newblock Nonlinear Dynamics \textbf{77}(1), 311--320 (2014)

\bibitem{chang1993non}
Chang, S., Bajaj, A.K., Krousgrill, C.M.: Non-linear vibrations and chaos in
  harmonically excited rectangular plates with one-to-one internal resonance.
\newblock Nonlinear Dynamics \textbf{4}(5), 433--460 (1993)

\bibitem{chen2017direct}
Chen, C., Zanette, D.H., Czaplewski, D.A., Shaw, S., L{\'o}pez, D.: Direct
  observation of coherent energy transfer in nonlinear micromechanical
  oscillators.
\newblock Nature Communications \textbf{8}(1), 1--7 (2017)

\bibitem{christiansen2006truncated}
Christiansen, S., Madsen, P.A.: On truncated {T}aylor series and the position
  of their spurious zeros.
\newblock Applied Numerical Mathematics \textbf{56}(1), 91--104 (2006)

\bibitem{cirillo2017analysis}
Cirillo, G., Habib, G., Kerschen, G., Sepulchre, R.: Analysis and design of
  nonlinear resonances via singularity theory.
\newblock Journal of Sound and Vibration \textbf{392}, 295--306 (2017)

\bibitem{dankowicz2013recipes}
Dankowicz, H., Schilder, F.: Recipes for continuation.
\newblock SIAM (2013)

\bibitem{dankowicz2020multidimensional}
Dankowicz, H., Wang, Y., Schilder, F., Henderson, M.E.: Multidimensional
  manifold continuation for adaptive boundary-value problems.
\newblock Journal of Computational and Nonlinear Dynamics \textbf{15}(5) (2020)

\bibitem{detroux2015harmonic}
Detroux, T., Renson, L., Masset, L., Kerschen, G.: The harmonic balance method
  for bifurcation analysis of large-scale nonlinear mechanical systems.
\newblock Computer Methods in Applied Mechanics and Engineering \textbf{296},
  18--38 (2015)

\bibitem{dhooge2003matcont}
Dhooge, A., Govaerts, W., Kuznetsov, Y.A.: {MATCONT}: a {MATLAB} package for
  numerical bifurcation analysis of {ODE}s.
\newblock ACM Transactions on Mathematical Software (TOMS) \textbf{29}(2),
  141--164 (2003)

\bibitem{doedel2007auto}
Doedel, E.J., Champneys, A.R., Dercole, F., Fairgrieve, T.F., Kuznetsov, Y.A.,
  Oldeman, B., Paffenroth, R., Sandstede, B., Wang, X., Zhang, C.: {AUTO-07P}:
  {C}ontinuation and bifurcation software for ordinary differential equations
  (2007)

\bibitem{geradin2014mechanical}
G{\'e}radin, M., Rixen, D.J.: Mechanical vibrations: theory and application to
  structural dynamics.
\newblock John Wiley \& Sons (2014)

\bibitem{guckenheimer2013nonlinear}
Guckenheimer, J., Holmes, P.: Nonlinear oscillations, dynamical systems, and
  bifurcations of vector fields, vol.~42.
\newblock Springer Science \& Business Media (2013)

\bibitem{haller2016nonlinear}
Haller, G., Ponsioen, S.: Nonlinear normal modes and spectral submanifolds:
  existence, uniqueness and use in model reduction.
\newblock Nonlinear Dynamics \textbf{86}(3), 1493--1534 (2016)

\bibitem{haro2016parameterization}
Haro, A., Canadell, M., Figueras, J.L., Luque, A., Mondelo, J.M.: The
  parameterization method for invariant manifolds.
\newblock Springer (2016)

\bibitem{haro2006parameterization-num}
Haro, A., de~la Llave, R.: A parameterization method for the computation of
  invariant tori and their whiskers in quasi-periodic maps: numerical
  algorithms.
\newblock Discrete \& Continuous Dynamical Systems-B \textbf{6}(6), 1261 (2006)

\bibitem{haro2006parameterization}
Haro, A., de~la Llave, R.: A parameterization method for the computation of
  invariant tori and their whiskers in quasi-periodic maps: rigorous results.
\newblock Journal of Differential Equations \textbf{228}(2), 530--579 (2006)

\bibitem{SHOBHIT}
Jain, S., Haller, G.: How to compute invariant manifolds and their reduced
  dynamics in high-dimensional finite-element models?
\newblock Nonlinear Dynamics  (2021).
\newblock \doi{10.1007/s11071-021-06957-4}

\bibitem{FEcode}
Jain, S., Marconi, J., Tiso, P.: Yet{A}nother{FE}code v1.1.1 (2020).
\newblock Http://doi.org/10.5281/zenodo.4011281

\bibitem{ssmtool21}
Jain, S., Thurnher, T., Li, M., Haller, G.: {SSMTool} 2.1: Computation of
  invariant manifolds \& their reduced dynamics in high-dimensional mechanics
  problems.
\newblock \url{https://github.com/haller-group/SSMTool-2.1}.
\newblock Accessed: 2021-6-9

\bibitem{ssmtool2}
Jain, S., Thurnher, T., Li, M., Haller, G.: {SSMTool} 2.0: Computation of
  invariant manifolds \& their reduced dynamics in high-dimensional mechanics
  problems (v1.0.0).
\newblock Zenodo  (2021).
\newblock Http://doi.org/10.5281/zenodo.4614202

\bibitem{jain2018exact}
Jain, S., Tiso, P., Haller, G.: Exact nonlinear model reduction for a von
  {K}{\'a}rm{\'a}n beam: Slow-fast decomposition and spectral submanifolds.
\newblock Journal of Sound and Vibration \textbf{423}, 195--211 (2018)

\bibitem{jiang2005construction}
Jiang, D., Pierre, C., Shaw, S.: The construction of non-linear normal modes
  for systems with internal resonance.
\newblock International Journal of Non-Linear Mechanics \textbf{40}(5),
  729--746 (2005)

\bibitem{jiang2005nonlinear}
Jiang, D., Pierre, C., Shaw, S.: Nonlinear normal modes for vibratory systems
  under harmonic excitation.
\newblock Journal of Sound and Vibration \textbf{288}(4-5), 791--812 (2005)

\bibitem{kang2017dynamic}
Kang, H.J., Guo, T.D., Zhao, Y.Y., Fu, W.B., Wang, L.H.: Dynamic modeling and
  in-plane 1: 1: 1 internal resonance analysis of cable-stayed bridge.
\newblock European Journal of Mechanics-A/Solids \textbf{62}, 94--109 (2017)

\bibitem{keller2018numerical}
Keller, H.B.: Numerical methods for two-point boundary-value problems.
\newblock Courier Dover Publications (2018)

\bibitem{krack2019harmonic}
Krack, M., Gross, J.: Harmonic balance for nonlinear vibration problems.
\newblock Springer (2019)

\bibitem{kurt2014effect}
Kurt, M., Slavkin, I., Eriten, M., McFarland, D.M., Gendelman, O.V., Bergman,
  L.A., Vakakis, A.F.: Effect of 1: 3 resonance on the steady-state dynamics of
  a forced strongly nonlinear oscillator with a linear light attachment.
\newblock Archive of Applied Mechanics \textbf{84}(8), 1189--1203 (2014)

\bibitem{coco-shoot}
Li, M., Dankowicz, H.: A {COCO}-based shooting toolbox for dynamical systems.
\newblock \url{https://github.com/mingwu-li/forward}.
\newblock Accessed: 2021-4-4

\bibitem{liu2021balancing}
Liu, J., M{\"o}ller, M., Schuttelaars, H.M.: Balancing truncation and round-off
  errors in fem: One-dimensional analysis.
\newblock Journal of Computational and Applied Mathematics \textbf{386}, 113219
  (2021)

\bibitem{nayfeh1989modal}
Nayfeh, A.H., Balachandran, B.: Modal interactions in dynamical and structural
  systems.
\newblock Applied Mechanics Reviews \textbf{42}(11s), S175--S201 (1989)

\bibitem{nayfeh1974nonlinear}
Nayfeh, A.H., Mook, D.T., Sridhar, S.: Nonlinear analysis of the forced
  response of structural elements.
\newblock The Journal of the Acoustical Society of America \textbf{55}(2),
  281--291 (1974)

\bibitem{nayfeh1988undesirable}
Nayfeh, A.H., et~al.: On the undesirable roll characteristics of ships in
  regular seas.
\newblock Journal of Ship Research \textbf{32}(02), 92--100 (1988)

\bibitem{neild2015use}
Neild, S.A., Champneys, A.R., Wagg, D.J., Hill, T.L., Cammarano, A.: The use of
  normal forms for analysing nonlinear mechanical vibrations.
\newblock Philosophical Transactions of the Royal Society A: Mathematical,
  Physical and Engineering Sciences \textbf{373}(2051), 20140404 (2015)

\bibitem{neild2011applying}
Neild, S.A., Wagg, D.J.: Applying the method of normal forms to second-order
  nonlinear vibration problems.
\newblock Proceedings of the Royal Society A: Mathematical, Physical and
  Engineering Sciences \textbf{467}(2128), 1141--1163 (2011)

\bibitem{opreni2021model}
Opreni, A., Vizzaccaro, A., Frangi, A., Touz{\'e}, C.: Model order reduction
  based on direct normal form: Application to large finite element mems
  structures featuring internal resonance.
\newblock arXiv preprint arXiv:2103.10545  (2021)

\bibitem{peeters2009nonlinearII}
Peeters, M., Vigui{\'e}, R., S{\'e}randour, G., Kerschen, G., Golinval, J.C.:
  Nonlinear normal modes, {Part II}: Toward a practical computation using
  numerical continuation techniques.
\newblock Mechanical Systems and Signal Processing \textbf{23}(1), 195--216
  (2009)

\bibitem{pellicano2000nonlinear}
Pellicano, F., Vestroni, F.: Nonlinear dynamics and bifurcations of an axially
  moving beam.
\newblock J. Vib. Acoust. \textbf{122}(1), 21--30 (2000)

\bibitem{pesheck2002new}
Pesheck, E., Pierre, C., Shaw, S.: A new {G}alerkin-based approach for accurate
  non-linear normal modes through invariant manifolds.
\newblock Journal of Sound and Vibration \textbf{249}(5), 971--993 (2002)

\bibitem{ponsioen2020model}
Ponsioen, S., Jain, S., Haller, G.: Model reduction to spectral submanifolds
  and forced-response calculation in high-dimensional mechanical systems.
\newblock Journal of Sound and Vibration \textbf{488}, 115640 (2020)

\bibitem{ponsioen2018automated}
Ponsioen, S., Pedergnana, T., Haller, G.: Automated computation of autonomous
  spectral submanifolds for nonlinear modal analysis.
\newblock Journal of Sound and Vibration \textbf{420}, 269--295 (2018)

\bibitem{ponsioen2019analytic}
Ponsioen, S., Pedergnana, T., Haller, G.: Analytic prediction of isolated
  forced response curves from spectral submanifolds.
\newblock Nonlinear Dynamics \textbf{98}(4), 2755--2773 (2019)

\bibitem{reddy2015introduction}
Reddy, J.N.: An Introduction to Nonlinear Finite Element Analysis: with
  applications to heat transfer, fluid mechanics, and solid mechanics.
\newblock Oxford University Press, USA (2015)

\bibitem{rosenberg1966nonlinear}
Rosenberg, R.: On nonlinear vibrations of systems with many degrees of freedom.
\newblock In: Advances in Applied Mechanics, vol.~9, pp. 155--242. Elsevier
  (1966)

\bibitem{shaw2016periodic}
Shaw, A.D., Hill, T., Neild, S., Friswell, M.: Periodic responses of a
  structure with 3: 1 internal resonance.
\newblock Mechanical Systems and Signal Processing \textbf{81}, 19--34 (2016)

\bibitem{shaw1993normal}
Shaw, S.W., Pierre, C.: Normal modes for non-linear vibratory systems.
\newblock Journal of Sound and Vibration \textbf{164}(1), 85--124 (1993)

\bibitem{szalai2017nonlinear}
Szalai, R., Ehrhardt, D., Haller, G.: Nonlinear model identification and
  spectral submanifolds for multi-degree-of-freedom mechanical vibrations.
\newblock Proceedings of the Royal Society A: Mathematical, Physical and
  Engineering Sciences \textbf{473}(2202), 20160759 (2017)

\bibitem{tang2019nonlinear}
Tang, Y.Q., Ma, Z.G.: Nonlinear vibration of axially moving beams with internal
  resonance, speed-dependent tension, and tension-dependent speed.
\newblock Nonlinear Dynamics \textbf{98}(4), 2475--2490 (2019)

\bibitem{thomas2005non}
Thomas, O., Touz{\'e}, C., Chaigne, A.: Non-linear vibrations of free-edge thin
  spherical shells: modal interaction rules and 1: 1: 2 internal resonance.
\newblock International Journal of Solids and Structures \textbf{42}(11-12),
  3339--3373 (2005)

\bibitem{thomas2007non}
Thomas, O., Touz{\'e}, C., Luminais, {\'E}.: Non-linear vibrations of free-edge
  thin spherical shells: experiments on a 1: 1: 2 internal resonance.
\newblock Nonlinear Dynamics \textbf{49}(1), 259--284 (2007)

\bibitem{touze2006nonlinear}
Touz{\'e}, C., Amabili, M.: Nonlinear normal modes for damped geometrically
  nonlinear systems: Application to reduced-order modelling of harmonically
  forced structures.
\newblock Journal of Sound and Vibration \textbf{298}(4-5), 958--981 (2006)

\bibitem{vakakis2008nonlinear}
Vakakis, A.F., Gendelman, O.V., Bergman, L.A., McFarland, D.M., Kerschen, G.,
  Lee, Y.S.: Nonlinear targeted energy transfer in mechanical and structural
  systems, vol. 156.
\newblock Springer Science \& Business Media (2008)

\bibitem{vakakis2001normal}
Vakakis, A.F., Manevitch, L.I., Mikhlin, Y.V., Pilipchuk, V.N., Zevin, A.A.:
  Normal modes and localization in nonlinear systems.
\newblock Springer (2001)

\bibitem{veraszto2020explicit}
Veraszto, Z., Ponsioen, S., Haller, G.: Explicit third-order model reduction
  formulas for general nonlinear mechanical systems.
\newblock Journal of Sound and Vibration \textbf{468}, 115039 (2020)

\bibitem{vizzaccaro2020direct}
Vizzaccaro, A., Shen, Y., Salles, L., Blaho{\v{s}}, J., Touz{\'e}, C.: Direct
  computation of nonlinear mapping via normal form for reduced-order models of
  finite element nonlinear structures.
\newblock arXiv preprint arXiv:2009.12145  (2020)

\bibitem{von2001harmonic}
Von~Groll, G., Ewins, D.J.: The harmonic balance method with arc-length
  continuation in rotor/stator contact problems.
\newblock Journal of Sound and Vibration \textbf{241}(2), 223--233 (2001)

\bibitem{wood2018saturation}
Wood, H., Roman, A., Hanna, J.: The saturation bifurcation in coupled
  oscillators.
\newblock Physics Letters A \textbf{382}(30), 1968--1972 (2018)

\bibitem{xiaodong2006non}
Yang, X., Chen, L.Q.: Non-linear forced vibration of axially moving
  viscoelastic beams.
\newblock Acta Mechanica Solida Sinica \textbf{19}(4), 365--373 (2006)

\bibitem{zavodney1989non}
Zavodney, L.D., Nayfeh, A.: The non-linear response of a slender beam carrying
  a lumped mass to a principal parametric excitation: theory and experiment.
\newblock International journal of non-linear mechanics \textbf{24}(2),
  105--125 (1989)

\end{thebibliography}
\bibliographystyle{spmpsci}

\end{document}